\documentclass[12pt]{amsart}
\usepackage{amsmath,amsthm,amssymb}
\usepackage{enumerate}
\usepackage{graphicx}
\usepackage{float}
\usepackage{epstopdf}
\usepackage{cite}
\usepackage[margin=1in]{geometry}
\usepackage{tikz}
\usetikzlibrary{matrix}
\usetikzlibrary{arrows}
\usetikzlibrary{positioning}
\usepackage{overpic}
\usepackage{mathtools}
\usepackage{subcaption}
\usepackage{tikz,graphicx}
\usetikzlibrary{matrix}
\usepackage{todonotes}
\usepackage[all]{xy}

\theoremstyle{plain}
\newtheorem{thm}{Theorem}[section]
\newtheorem{cor}[thm]{Corollary}
\newtheorem{lem}[thm]{Lemma}

\theoremstyle{definition}
\newtheorem{defn}[thm]{Definition}
\theoremstyle{remark}
\newtheorem{rem}[thm]{Remark}
\newtheorem{ex}[thm]{Example}

\newcommand{\R}{\mathbb{R}}
\newcommand{\T}{\mathbb{T}}

\newcommand{\Z}{\mathbb{Z}}
\newcommand{\Q}{\mathbb{Q}}

\newcommand{\cfk}{\mathit{CFK}}
\newcommand{\hfk}{\mathit{HFK}}
\newcommand{\mcL}{\mathcal{L}}
\newcommand{\h}{\mathcal{H}}
\newcommand{\cH}{\mathcal{H}}
\newcommand{\x}{\mathsf{x}}
\newcommand{\X}{\mathsf{X}}
\newcommand{\str}{\mathsf{str}}
\newcommand{\Y}{\mathsf{Y}}
\newcommand{\sm}{\mathsf{sm}}
\DeclareMathOperator{\unknot}{unknot}
\DeclareMathOperator{\gr}{gr}
\DeclareMathOperator{\id}{id}
\DeclareMathOperator{\Sym}{Sym}

\DeclareMathOperator{\In}{In}
\DeclareMathOperator{\Out}{Out}

\DeclareMathOperator{\Tor}{Tor}

\DeclareMathOperator{\rk}{rk}
\newcommand{\sln}{\mathfrak{sl}_{n}}

\begin{document}

\title{A Spectral Sequence from Khovanov Homology to Knot Floer Homology}
\author{Nathan Dowlin}
\thanks{The author was partially supported by NSF grant DMS-1606421}

\maketitle

\begin{abstract}

A well-known conjecture of Rasmussen states that for any knot $K$ in $S^{3}$, the rank of the reduced Khovanov homology of $K$ is greater than or equal to the rank of the reduced knot Floer homology of $K$. This rank inequality is supposed to arise as the result of a spectral sequence from Khovanov homology to knot Floer homology. Using an oriented cube of resolutions construction for a homology theory related to knot Floer homology, we prove this conjecture.

\end{abstract}



\section{Introduction}

Khovanov homology and knot Floer homology are two knot invariants which have many commonalities despite significant differences in their constructions. In their simplest forms, they both assign to each knot or link in $S^{3}$ a bigraded abelian group. The main purpose of this paper is to give a proof of Rasmussen's conjecture \cite{Rasmussen3}:

\begin{thm} \label{mainthmintro}
For any knot $K$ in $S^{3}$, there is a spectral sequence from $\overline{Kh}(K)$ to $\delta$-graded $\widehat{\hfk}(m(K))$, where $\overline{Kh}(K)$ is the reduced Khovanov homology of $K$ and $\widehat{\hfk}(m(K))$ is the reduced knot Floer homology of the mirror of $K$. Both homologies are defined with coefficients in $\Q$.

\end{thm}

\begin{cor} \label{rk}
For any knot $K$ in $S^{3}$ there is a rank inequality 
\[ \rk(\overline{Kh}(K)) \ge \rk(\widehat{\hfk}(K)) \]
\end{cor}

Khovanov homology was developed by Khovanov in \cite{Khov1} based on the representation theory of the quantum group $U_{q}(\mathfrak{sl}_{2})$. It categorifies the Jones polynomial, where a \emph{categorification} of a knot polynomial is a multi-graded abelian group whose graded Euler characteristic returns the polynomial in question. The two gradings are the homological grading $\gr_{h}$ and the quantum grading $\gr_{q}$. There is also a diagonal grading called the $\delta$-grading given by $\gr_{\delta}=\gr_{q}-2\gr_{h}$.

Knot Floer homology of Ozsv\'{a}th-Szab\'{o} \cite{OS1} and Rasmussen \cite{Rasmussen2} is constructed as a Lagrangian Floer homology in an auxiliary symplectic manifold. It categorifies the Alexander polynomial, and is known to detect the unknot \cite{ozsvath2004holomorphic} and the trefoils \cite{ghiggini2008knot}. The two gradings on knot Floer homology are the Maslov grading $M$ and the Alexander grading $A$, with $\delta$-grading $2M-2A$. In Theorem \ref{mainthmintro}, the $\delta$-grading on Khovanov homology descends to the $\delta$-grading on knot Floer homology (up to an overall shift).

Spectral sequences from Khovanov homology to a Floer-theoretic invariant have become a recurring theme in knot theory. In 2005, Ozsv\'{a}th and Szab\'{o} constructed a spectral sequence from the (odd) Khovanov homology of $K$ to the Heegaard Floer homology of the double branched cover of $K$, and in 2010 Kronheimer and Mrowka constructed a spectral sequence from Khovanov homology to instanton Floer homology \cite{unknotdetector}. The latter is particularly relevant to this paper, as knot Floer homology is conjecturally isomorphic to instanton Floer homology. The  spectral sequence in this paper seems to be the analog on the Heegaard Floer side.

A knot $K$ is called $R Kh$-\emph{thin} (resp. $R \hfk$-\emph{thin}) if the reduced Khovanov homology (resp. knot Floer homology) of $K$ with coefficients in $R$ lies in a single $\delta$-grading. The spectral sequence has the following immediate implications:

\begin{itemize}

\item If $K$ is $\Q Kh$-thin, then $K$ is $\Q \hfk$-thin.

\item If $K$ is $\Z_{2}Kh$-thin and $K'$ is obtained from $K$ via Conway mutation, then 
\[ \widehat{\hfk}(K) \cong \widehat{\hfk}(K') \] as $\delta$-graded $\Q$-vector spaces.

\end{itemize}

\noindent
The latter follows from the fact that Khovanov homology with $\Z_{2}$ coefficients is mutation invariant (\hspace{1sp}\cite{BloomMutation}, \cite{WehrliMutation}) together with the fact that any $\Z_{2}Kh$-thin knot is $\Q Kh$-thin by the universal coefficient theorem. We don't need to include mirroring in these statements because
\[ \widehat{\hfk}(K,i) \cong \widehat{\hfk}(m(K),-i) \]

\noindent
where $\widehat{\hfk}(K,i)$ denotes the reduced knot Floer homology of $K$ in $\delta$-grading $i$.

Theorem \ref{mainthmintro} and its generalization to links at the end of the section also gives a new proof of the following results, where the original proofs all rely on the Kronheimer-Mrowka spectral sequence or the Ozsv\'{a}th-Szab\'{o} spectral sequence:

\begin{itemize}

\item Khovanov homology detects the unknot \cite{unknotdetector}, the unlink \cite{batson2015link, xie2018earrings} the trefoils \cite{baldwin2018khovanov}, and the Hopf links \cite{baldwin2018khovanov2}.

\item If $K$ is $\Q Kh$-thin, then the Alexander polynomial detects both the genus of $K$ and the fiberedness of $K$.

\end{itemize}

\noindent
The genus and fiberedness detection follows from the fact that if $K$ is $\Q Kh$-thin then it is $\Q \hfk$-thin, and for a $\Q \hfk$-thin knot the Alexander polynomial detects bigraded $\widehat{\hfk}(K)$ up to a shift based on the signature. One can then use the fact that knot Floer homology detects genus \cite{ozsvath2004holomorphic} and fiberedness \cite{ni2007knot}.

Although it has structural similarities with the Kronheimer-Mrowka spectral sequence, the construction is quite different. It is based on an \emph{oriented} cube of resolutions for knot Floer homology rather than an unoriented cube of resolutions, and it is defined combinatorially. It also comes with additional grading information.

\begin{thm}

Let $E_{k}(K)$ denote the spectral sequence in Theorem \ref{mainthmintro} with $E_{0} = \overline{Kh}(K)$. Then the differential $d_{k}$ has bigrading $(2k+3, 4k+4)$ with respect to $(\gr_{h}, \gr_{q})$. 

\end{thm}

\noindent
In particular, all differentials increase the homological grading, and they do so by an odd number. They are also homogeneous of degree -2 with respect to the Khovanov $\delta$-grading.

The first discovered knot for which this spectral sequence is non-trivial is the $(4,5)$ torus knot. In the case of instanton Floer homology, this spectral sequence still hasn't been computed, though some progress has been made in this direction \cite{lobb2013spectral}. However, in our case it can be computed quite easily using a grading argument. There is a single non-trivial differential of bigrading $(5, 8)$, as shown in Figure \ref{t45}.

\begin{figure}[h!]
    \centering
\begin{tikzpicture}[scale=.74]
  \draw[->] (0,0) -- (10.5,0) node[right] {$t$};
  \draw[->] (0,0) -- (0,8.5) node[above] {$q$};
  \draw[step=1] (0,0) grid (10,8);
  \draw (0.5,-.2) node[below] {$0$};
  \draw (1.5,-.2) node[below] {$1$};
  \draw (2.5,-.2) node[below] {$2$};
  \draw (3.5,-.2) node[below] {$3$};
  \draw (4.5,-.2) node[below] {$4$};
  \draw (5.5,-.2) node[below] {$5$};
  \draw (6.5,-.2) node[below] {$6$};
  \draw (7.5,-.2) node[below] {$7$};
  \draw (8.5,-.2) node[below] {$8$};
  \draw (9.5,-.2) node[below] {$9$};
  \draw (-.2,0.5) node[left] {$12$};
  \draw (-.2,1.5) node[left] {$14$};
  \draw (-.2,2.5) node[left] {$16$};
  \draw (-.2,3.5) node[left] {$18$};
  \draw (-.2,4.5) node[left] {$20$};
  \draw (-.2,5.5) node[left] {$22$};
  \draw (-.2,6.5) node[left] {$24$};
  \draw (-.2,7.5) node[left] {$26$};
  \fill (0.5, 0.5) circle (.15);
  \fill (2.5, 2.5) circle (.15);
  \fill (3.5, 3.5) circle (.15);
  \fill (4.5, 3.5) circle (.15);
  \fill (5.5, 5.5) circle (.15);
  \fill (6.5, 4.5) circle (.15);
  \fill (7.5, 6.5) circle (.15);
  \fill (8.5, 6.5) circle (.15);
  \fill (9.5, 7.5) circle (.15);
  \draw[-stealth, red, thick, scale=1](4.68,3.5) [bend right = 50] to (9.45,7.4);
\end{tikzpicture}
\caption{The only non-zero differential in $E_{k}(T(4,5))$ has homological grading $5$ and quantum grading $8$.} \label{t45}
\end{figure}
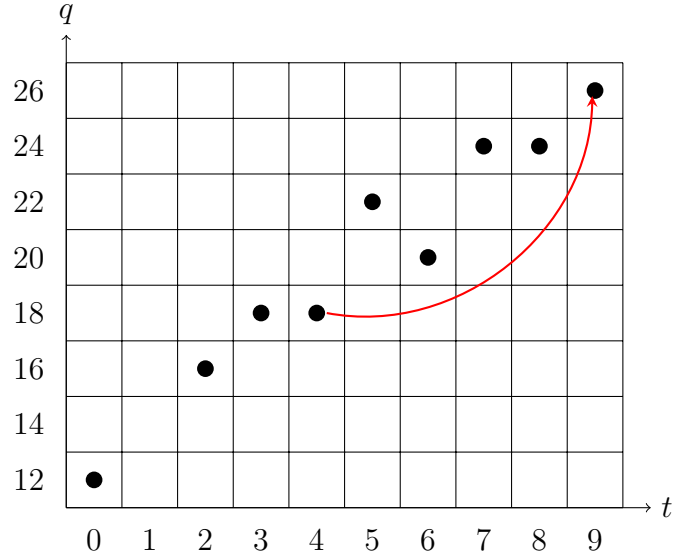

Painting with the broadest of strokes, the spectral sequence comes from a cube of resolutions for $\hfk_{2}(K)$, a knot homology theory defined by the author \cite{Dowlin} for knots and singular knots satisfying the following properties:

\begin{itemize}

\item The reduced theory $\widehat{\hfk}_{2}(K)$ is isomorphic to reduced, $\delta$-graded $\widehat{\hfk}(K)$.
\vspace{1mm}

\item For any completely singular diagram $S$, $\hfk_{2}(S) \cong Kh(\sm(S))$ where $\sm(S)$ is the digram obtained by replacing each singularization in $S$ with the unoriented smoothing as in Figure \ref{fig2t}.

\end{itemize}

\noindent
We construct an oriented cube of resolutions for $\hfk_{2}(K)$ such that the edge maps are the Khovanov edge maps. The induced higher face maps give the spectral sequence. The exact triangle goes the opposite direction from the Khovanov one, so the spectral sequence converges to the knot Floer homology of the mirror of $K$. The construction is similar in spirit to \cite{alishahi2018link}, and many similar methods are used for comparing it to Khovanov homology.

\begin{figure}[h!]
\tiny
\begin{subfigure}{.5\textwidth}
 \centering
\def\svgwidth{5cm}
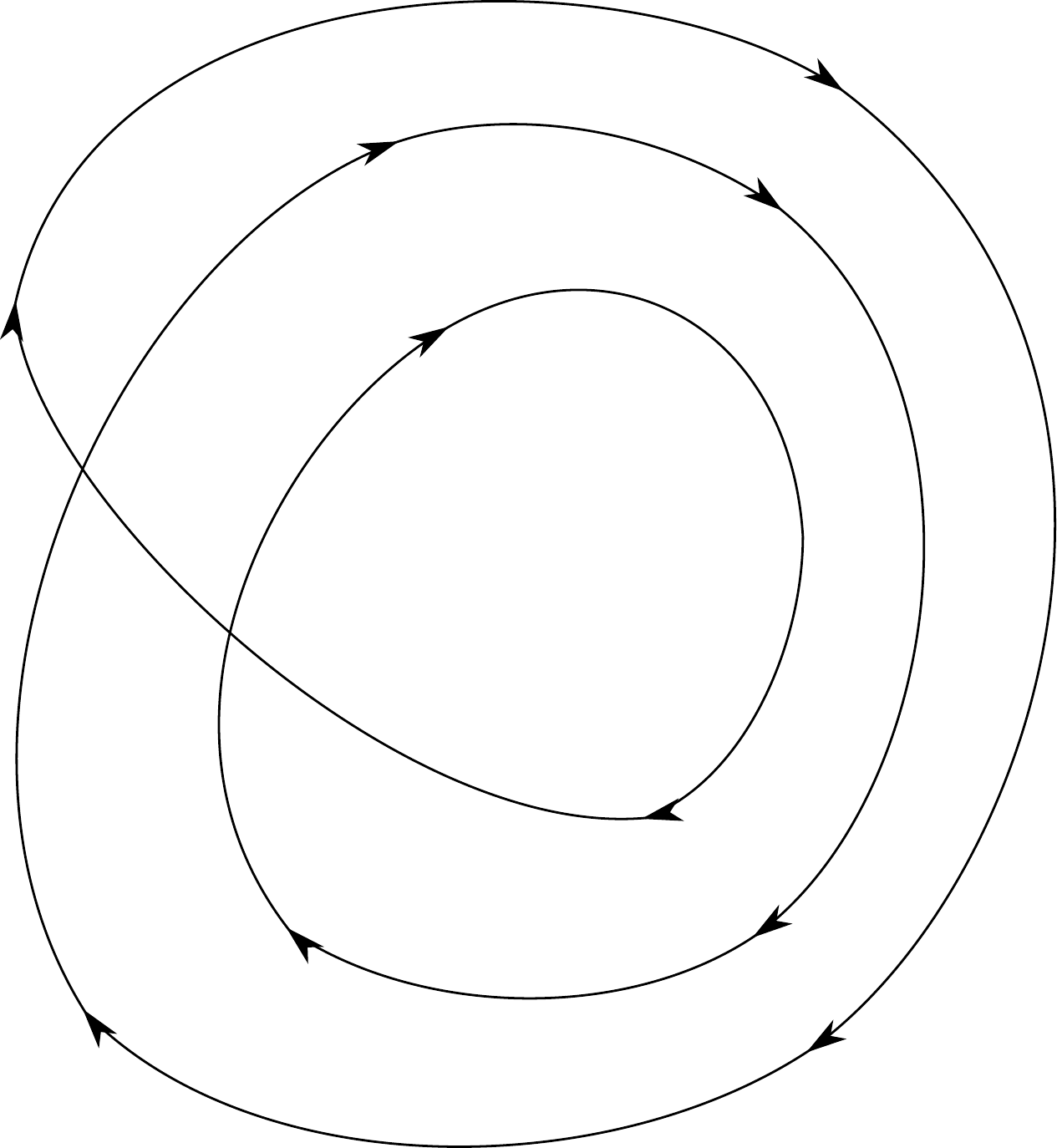 
\caption{A singular braid $S$}
\end{subfigure}%
\begin{subfigure}{.5\textwidth}
  \centering
  \def\svgwidth{5cm}
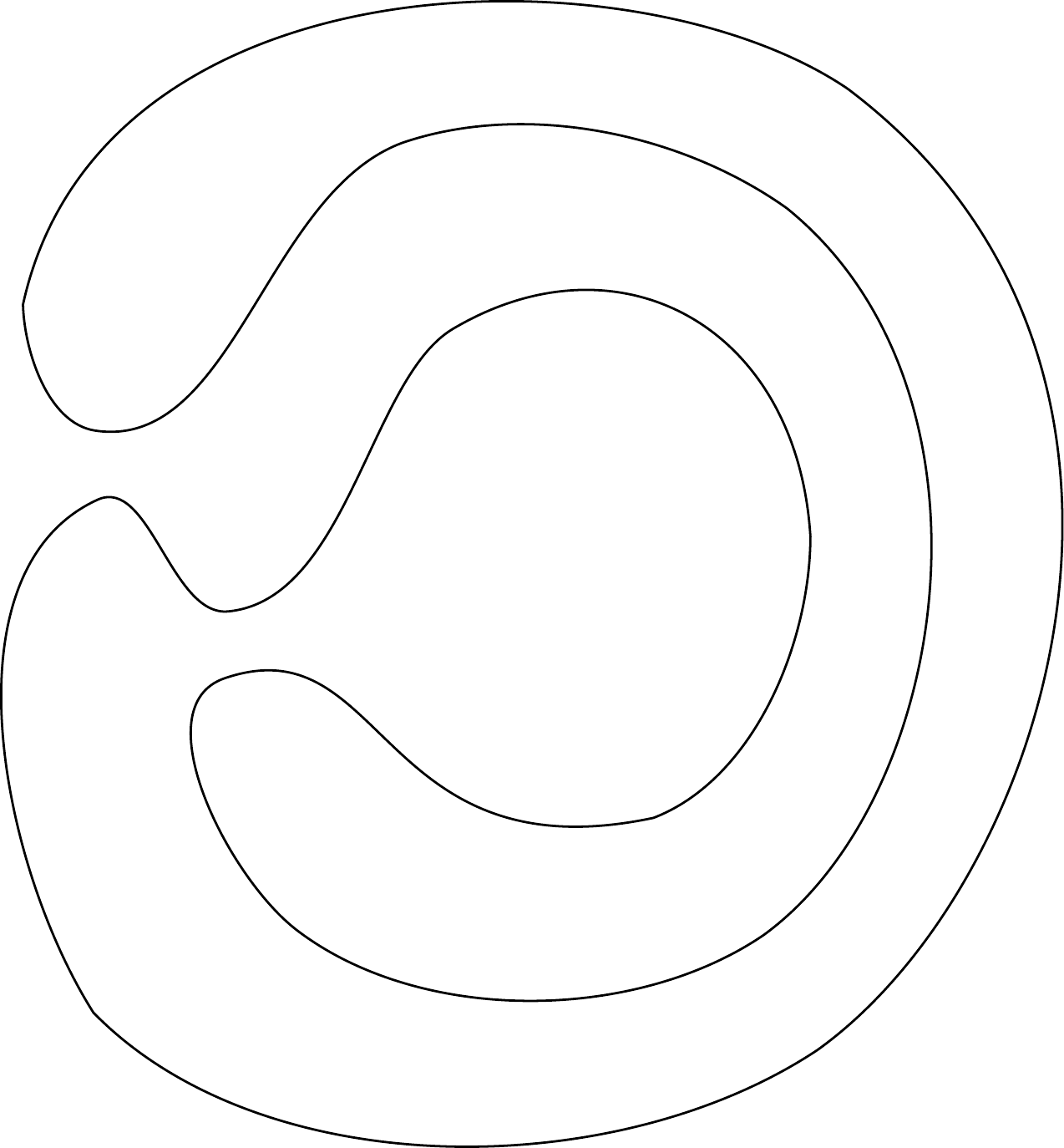 \caption{The smoothing $\sm(S)$}
  \label{fig2b}
\end{subfigure}
\caption{}\label{fig2t}
\label{crossings}
\end{figure}

The oriented cube of resolutions for knot Floer homology was originally defined with twisted coefficients by Ozsv\'{a}th and Szab\'{o} \cite{Szabo}, and was later modified to the standard ground ring by Manolescu \cite{Manolescu}. The Ozsv\'{a}th and Szab\'{o} construction has the advantage of being combinatorial due to a regular sequence argument, while Manolescu's has the advantage of being more easily related to $\widehat{\hfk}(K)$. 

The construction of the oriented cube of resolutions for $\hfk_{2}$ requires working with untwisted coefficients but also for it to be combinatorial -- it turns out that this is possible if the diagram used has `sufficiently many' singularizations, but not for the standard diagram for a knot or link. The set of (partially singular) diagrams for which this property holds will be denoted $\mathcal{D^{R}}$, since the property relates to regular sequences.

This strangeness in the objects allowed is overcome by the following two lemmas:

\begin{lem}

For any partially singular diagram $D$, there is an isomorphism \[\widehat{\hfk}_{2}(D) \cong \widehat{\hfk}(\sm(D))\] as $\delta$-graded vector spaces.

\end{lem}

\begin{lem}

For any link $L$ there is a diagram $D \in \mathcal{D^{R}}$ such that $\sm(D)$ is a diagram for $L$.

\end{lem}
\noindent
An example of a diagram $D$ in $ \mathcal{D^{R}}$ whose smoothing is the trefoil is shown in Figure \ref{tref}.

\begin{figure}[h!]
\tiny
\begin{subfigure}{.5\textwidth}
 \centering
\def\svgwidth{5cm}
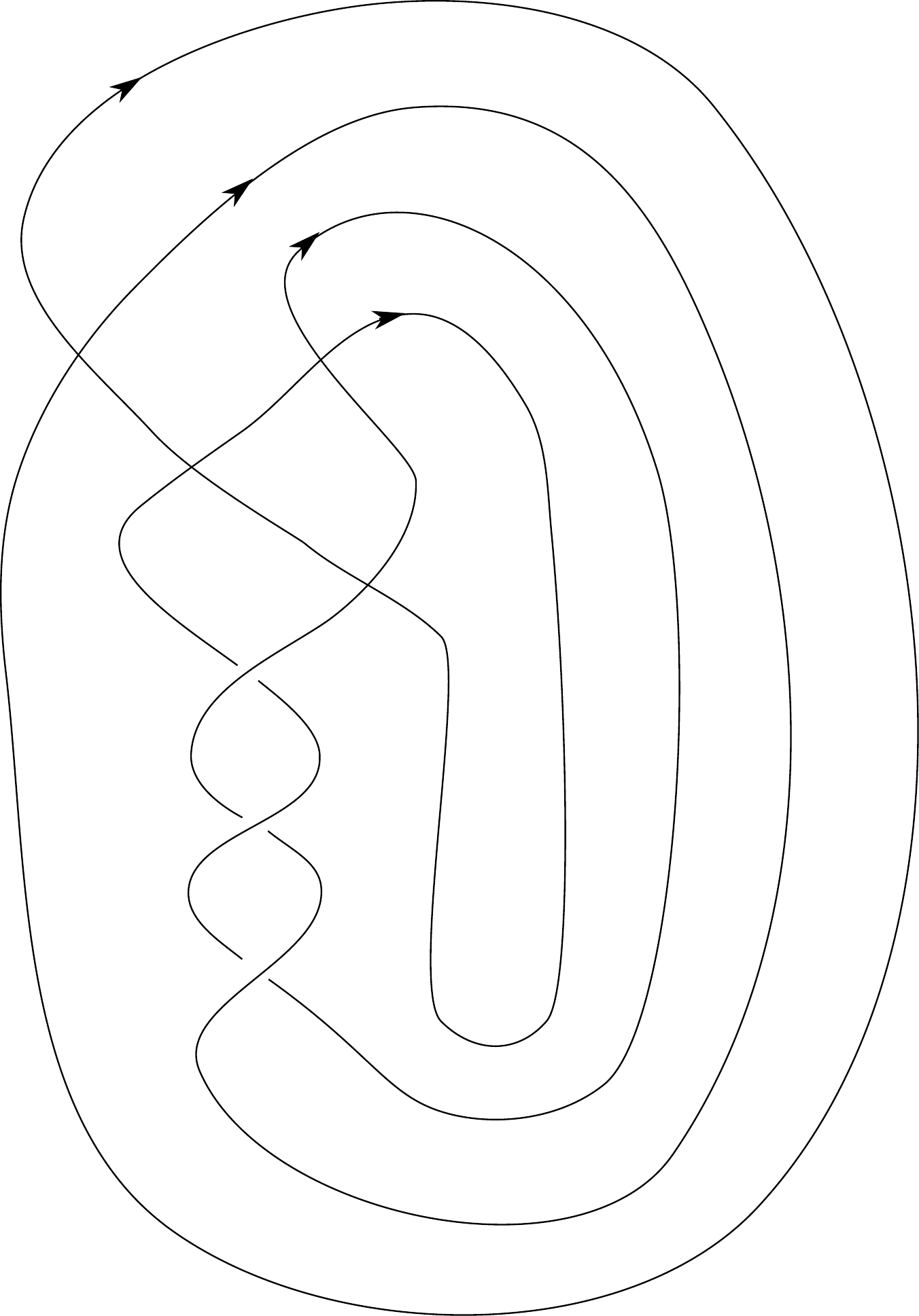 
\caption{A diagram $D$ in $\mathcal{D^{R}}$}
\end{subfigure}%
\begin{subfigure}{.5\textwidth}
  \centering
  \def\svgwidth{5cm}
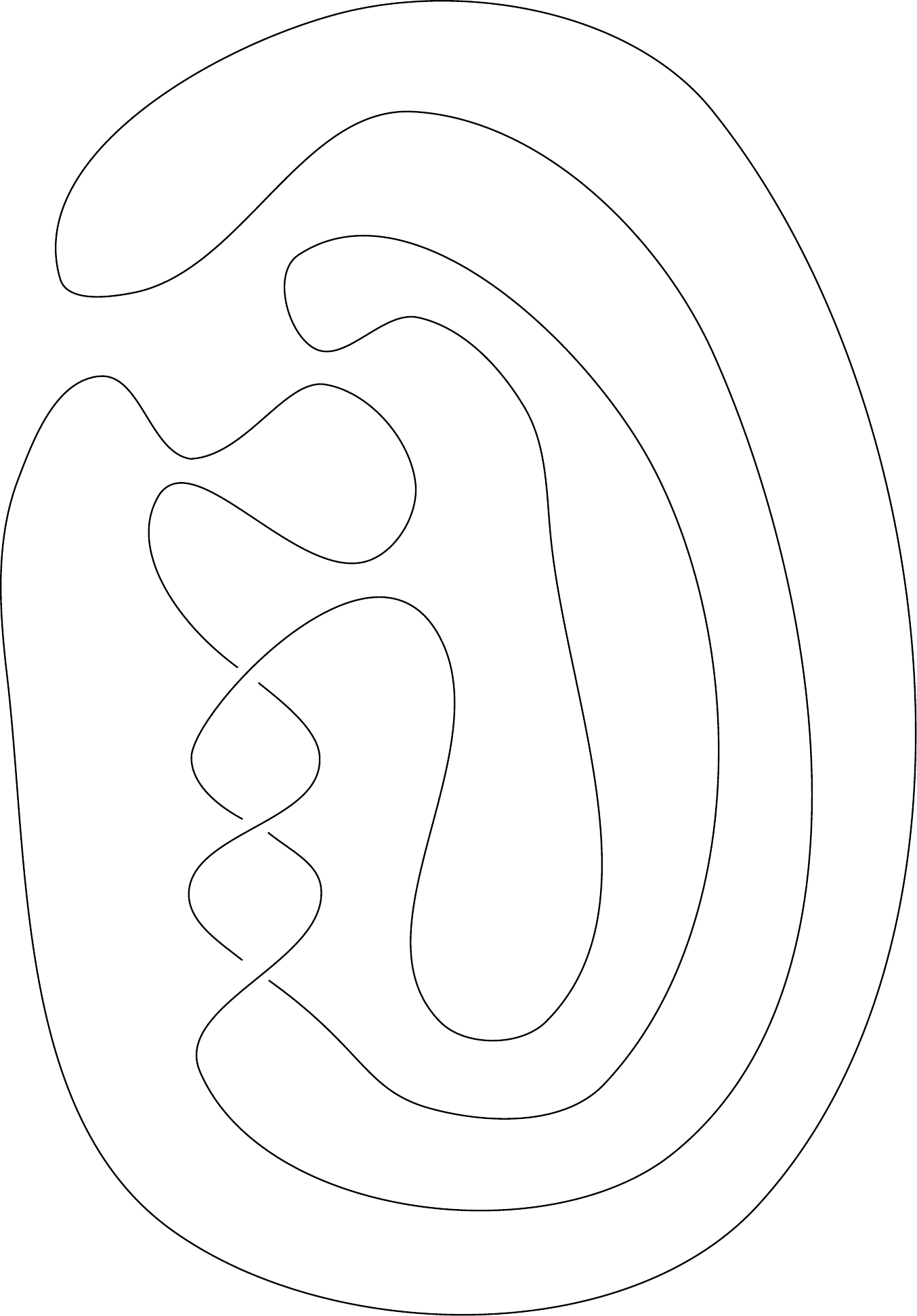  \caption{The smoothing $\sm(D)$ is isotopic to the right-handed trefoil}
\end{subfigure}
\caption{}\label{tref}
\end{figure}

The structure of the proof of Theorem \ref{mainthmintro} is as follows: given a knot $K$, let $D \in \mathcal{D^{R}}$ be a diagram with $\sm(D)$ a diagram for $K$, and let $\widehat{C}_{2}(D)$ denote the oriented cube of resolutions for $\widehat{\hfk}_{2}(D)$. The spectral sequence induced by the cube filtration on $\widehat{C}_{2}(D)$ has $E_{2}$ page isomorphic to $\overline{Kh}(m(\sm(D)))=\overline{Kh}(m(K))$ and $E_{\infty}$ page $\widehat{\hfk}_{2}(\sm(D)) = \hfk_{2}(K)$. Mirroring $K$ gives the spectral sequence in Theorem \ref{mainthmintro}. One caveat is that the complex $\widehat{C}_{2}(D)$ only admits a \emph{relative} grading, as there is some difficulty in understanding absolute gradings when moving from $D$ to $\sm(D)$.

This spectral sequence can be generalized to links using the pointed Khovanov homology of Baldwin, Levine, and Sarkar \cite{BLS}:

\begin{thm}

Let $L$ be a link in $S^{3}$, and let $\mathbf{p}=\{p_{1},...,p_{l}\}$ be a collection of basepoints, exactly one on each component of $L$. Then up to an overall grading shift, there is a spectral sequence from $Kh(L, \mathbf{p})$ to $\delta$-graded $\widehat{\hfk}(m(L))$, where $Kh(L, \mathbf{p})$ is the pointed Khovanov homology of $(L, \mathbf{p})$.

\end{thm}

\noindent
The homology $Kh(L, \mathbf{p})$ will be denoted $\widehat{Kh}(L)$ since it does not depend on the points $\mathbf{p}$.

\begin{cor}
For any $l$-component link $L$ in $S^{3}$, there is a rank inequality 
\[ 2^{l-1} \rk(\overline{Kh}(L)) \ge \rk(\widehat{\hfk}(L)) \]
\end{cor}

Khovanov homology is known to be isomorphic to Khovanov-Rozansky $\mathfrak{sl}_{2}$ homology \cite{KR,Hughes}. Together with Rasmussen's spectral sequence from HOMFLY-PT homology to $\mathfrak{sl}_{2}$ homology \cite{Rasmussen}, the previous theorem gives a proof of the following conjecture of Dunfield, Gukov, and Rasmussen \cite{Gukov} though it is missing one of the gradings on knot Floer homology:

\begin{cor}

For any link $L$ in $S^{3}$, there is a spectral sequence from $\overline{H}_{H}(L)$ to $\delta$-graded $\widehat{\hfk}(m(L))$, where $\overline{H}_{H}(L)$ is the reduced HOMFLY-PT homology of $L$.

\end{cor}

The paper is organized as follows: in Section 2, we describe the construction of the complex $\widehat{C}_{2}(L)$ and the spectral sequence since it is purely algebraic and requires no background on the two theories. Section 3 provides background on Khovanov homology, knot Floer homology, the knot Floer cube of resolutions, and the theory $\hfk_{2}$. Section 4 is devoted to computing $\hfk_{2}^{-}(S)$ for any complete resolution $S$ as an $R$-module. Section 5 gives a proof that the $E_{2}$ page is Khovanov homology, and Section 6 shows that the $E_{\infty}$ page is knot Floer homology. Section 7 puts it all together to prove the existence of the spectral sequence and computes it for the $(4,5)$ torus knot.

\subsection*{Acknowledgments} I would like to thank Akram Alishahi, John Baldwin, Adam Levine, Andy Manion, Ciprian Manolescu, and Zolt\'{a}n Szab\'{o} for helpful discussions and valuable comments.

\section{Construction of the complex $\widehat{C}_{2}(D)$}\label{consection}

\subsection{The cube of resolutions} The complex $C^{-}_{2}(D)$ is constructed as an oriented cube of resolutions, and $\widehat{C}_{2}(D)$ is obtained by reducing this complex. Let $D_{+}$, $D_{-}$, $D_{s}$, and $D_{\x}$ be planar diagrams which agree away from a crossing, and at that crossing $D_{+}$ has a positive crossing, $D_{-}$ has a negative crossing, $D_{s}$ has the oriented smoothing, and $D_{\x}$ has the singularization. The complex $C^{-}_{2}(D)$ being an oriented cube of resolutions means that it decomposes over the crossings as mapping cones:
\[ C^{-}_{2}(D_{+}) = C^{-}_{2}(D_{\x}) \to C^{-}_{2}(D_{s}) \]
\[ C^{-}_{2}(D_{-}) = C^{-}_{2}(D_{s}) \to C^{-}_{2}(D_{\x}) \]

\noindent
For a positive crossing, $D_{\x}$ is called the 0-resolution and $D_{s}$ the 1-resolution, and vice versa for the negative crossing (see Figure \ref{resolutions3}).

\begin{figure}[!h]
\centering
\scriptsize
\def\svgwidth{8cm}
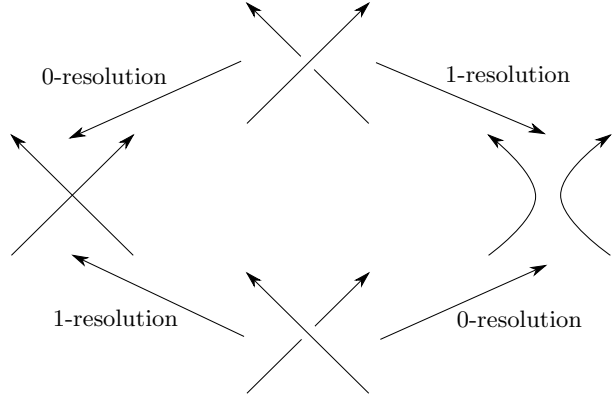 
\caption{0- and 1-resolutions of positive and negative crossings}\label{resolutions3}
\end{figure}

Let $c(D)$ denote the crossings of $D$, and let $I:c(D) \to \{0,1\}$ be a vertex in the cube. The diagram $D_{I}$ is obtained by replacing each crossing $c$ with the $I(c)$-resolution -- such a diagram is called a \emph{complete resolution}. Iterating over all crossings of $D$ gives cube complex where the complex at a vertex $I$ is $C^{-}_{2}(D_{I})$. The height of a vertex $I$ is given by
\[ \sum_{c} I(c) \]

The height induces a filtration on the complex $(C^{-}_{2}(D),d)$. For a general cube of resolutions complex one can write $d=d_{0}+d_{1}+...+d_{k}$, where $k$ is the number of crossings in $D$. However, on $C^{-}_{2}(D)$, we have $d=d_{0}+d_{1}$, i.e. the higher face maps are all zero. The differential $d_{0}$ is called the vertex map and $d_{1}$ the edge map.

\subsection{The ground ring and the non-local ideal}

Let $D$ be a decorated, partially singular braid diagram. Viewing $D$ as an oriented 4-valent graph with additional over/under data at the crossings, let $e_{1},...,e_{m}$ denote the edges of $D$, and let $R=\Q[U_{1},...,U_{m}]$. Without loss of generality, suppose that the decorated edge is labeled $e_{1}$ and that it is on the left-most strand of the braid. Let $S=D_{I}$ be a complete resolution of $D$, so that $S$ has the same edges as $D$. It can have both bivalent and 4-valent vertices coming from smoothings and singularizations, respectively.

Viewing $S$ as an oriented, decorated graph in $\R^{2}$, let $\Omega$ be a smoothly embedded disc in $\R^{2}$ whose boundary intersects $S$ transversely and only at edges. Moreover, suppose $\Omega$ does not contain the decorated edge $e_{1}$, though it is allowed to intersect $e_{1}$. Let $\In(\Omega)$ (resp. $\Out(\Omega)$) denote the set of edges which intersect the boundary of $\Omega$ and are oriented into (resp. out of) $\Omega$. Then $\Omega$ determines a homogeneous polynomial $P(\Omega)$ in $R$ defined as

\[ P(\Omega) = \prod_{e_{i} \in \In(\Omega)} U_{i} - \prod_{e_{j} \in \Out(\Omega)} U_{j} \]

\begin{defn}

The non-local ideal $N(S)$ is the ideal in $R$ generated by the polynomials $P(\Omega)$ over all choices of $\Omega$.

\end{defn}

\noindent
Note that the polynomials $P(\Omega)$ are homogeneous because the number of incoming edges at each vertex $v$ is equal the the number of outgoing edges at $v$.

\subsection{The family of diagrams $\mathcal{D^{R}}$}

Let $v_{4}(D)$ denote the set of singular points in $D$. The resolution $S$ has two types of 4-valent vertices -- those which were 4-valent vertices in $D$, and those which are singularizations of crossings in $D$. We will write the latter as $v_{4}(I)$ so that $v_{4}(S) = v_{4}(D) \cup v_{4}(I)$. For each $v \in v_{4}(S)$, define \[L(v) = U_{a}^{(v)}+U_{b}^{(v)}-U_{c}^{(v)}-U_{d}^{(v)}\] 
where $e^{(v)}_{a}, e^{(v)}_{b}$ are the two outgoing edges at $v$ and $e^{(v)}_{c},e^{(v)}_{d}$ are the two incoming edges at $v$ as in Figure \ref{4vabcd}. Similarly, define 
\[ L^{+}(v) = U_{a}^{(v)}+U_{b}^{(v)}+U_{c}^{(v)}+U_{d}^{(v)} \]

\begin{figure}[h!]
 \centering
\def\svgwidth{5cm}
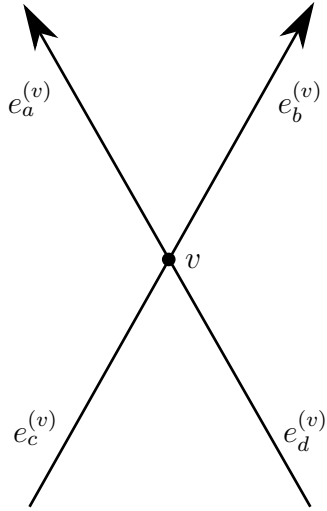 
\caption{A singularization $v$ with the four adjacent edges labeled} \label{4vabcd}
\end{figure}

\begin{defn}

The family $\mathcal{D^{R}}$ consists of all decorated, partially singular braid diagrams $D$ satisfying the following properties. First, we require that the decorated edge is left-most in the diagram. Second, for any complete resolution $S=D_{I}$ of $D$: 

\begin{itemize}

\item $S$ is connected
\item The linear terms $L(v)$ for $v \in v_{4}(I)$ form a regular sequence over $R/N(S)$

\end{itemize}

\end{defn}

\noindent
In Section \ref{lastsection} we give a sufficient condition for a diagram to be in $\mathcal{D^{R}}$.

\subsection{Definition of $\widehat{C}_{2}(D)$}

We now have the necessary tools to define the complex $C^{-}_{2}(D)$ for any $D \in \mathcal{D^{R}}$. First we will define a complex $C_{2}^{-}(D)$, then the reduced complex $\widehat{C}_{2}(D)$ will be defined in terms of the minus complex.

The complex $C^{-}_{2}(S)$ for each complete resolution $S=D_{I}$ will be defined first, then the edge maps will be added. For each complete resolution $S$, the linear ideal $L_{I}(S)$ is the ideal generated by $L(v)$ over all $v$ in $v_{4}(I)$. For the remaining 4-valent vertices $v_{4}(D)$, we instead have a Kozsul complex
\[ \mathcal{L}_{D}^{+} = \bigotimes_{v \in v_{4}(D)}  \xymatrix{R\ar@<1ex>[r]^{L(v)}&R\ar@<1ex>[l]^{L^{+}(v)}} \]

\begin{defn}
The complex $C^{-}_{2}(S)$ is given by 
\[ C^{-}_{2}(S) = R/(N(S)+L_{I}(S)) \otimes \mathcal{L}_{D}^{+} \]
\end{defn}

\begin{lem}
The differential satisfies $d^{2}=0$ on $C^{-}_{2}(S)$.
\end{lem}

\begin{proof}

The factor $\mathcal{L}_{D}^{+}$ is a matrix factorization, meaning $d^{2}=\omega I$ where \[\omega = \sum_{v \in v_{4}(D)} L(v)L^{+}(v)\] For each $v \in v_{4}(D)$, the quadratic relation $Q(v) = U_{a}^{(v)}U_{b}^{(v)} - U_{c}^{(v)}U_{d}^{(v)}$ is an element of $N(S)$, so over $R/(N(S) + L_{I}(S))$, $\omega = \sum_{v \in v_{4}(D)} (U_{a}^{(v)})^{2}+(U_{b}^{(v)})^{2} - (U_{c}^{(v)})^{2}-(U_{d}^{(v)})^{2}$.

For each 4-valent vertex $v$ in $v_{4}(I)$, both $L(v)$ and $Q(v)$ are in $N(S) + L_{I}(S)$, so $(U_{a}^{(v)})^{2}+(U_{b}^{(v)})^{2} - (U_{c}^{(v)})^{2}-(U_{d}^{(v)})^{2}$ is as well. Thus, $\sum_{v \in v_{4}(I)} (U_{a}^{(v)})^{2}+(U_{b}^{(v)})^{2} - (U_{c}^{(v)})^{2}-(U_{d}^{(v)})^{2}$ is in $N(S) + L_{I}(S)$. So over $R/(N(S) + L_{I}(S))$,

\begin{equation*}
\begin{split}
\omega = & \sum_{v \in v_{4}(D)} (U_{a}^{(v)})^{2}+(U_{b}^{(v)})^{2} - (U_{c}^{(v)})^{2}-(U_{d}^{(v)})^{2} \\
= & \sum_{v \in v_{4}(S)} (U_{a}^{(v)})^{2}+(U_{b}^{(v)})^{2} - (U_{c}^{(v)})^{2}-(U_{d}^{(v)})^{2} =0 \\
\end{split}
\end{equation*}

\end{proof}

We write $I \lessdot J$ if $I$ and $J$ are two vertices in the cube that differ at a single crossing $c$, with $I(c)=0$ and $J(c)=1$. The edge maps consist of a set of maps $d_{I,J}$ for all $I \lessdot J$. Let $I\lessdot J$, and suppose $I$ and $J$ differ at a positive crossing. We define $\phi_{+}$ to be the unique $R$-module map
\[    \phi_{+}:   R/(N(D_{I})+L_{I}(D_{I})) \to R/(N(D_{J})+L_{J}(D_{J})) \]

\noindent
such that $\phi_{+}(1)=1$. The edge map $d_{I,J}: C^{-}_{2}(D_{I}) \to C^{-}_{2}(D_{J})$ is given by $\phi_{+} \otimes \id$.

Similarly, let $I\lessdot J$, and suppose $I$ and $J$ differ at a negative crossing. We define $\phi_{-}$ to be the unique $R$-module map
\[    \phi_{-}:   R/(D_{I})+L_{I}(D_{I})) \to R/(N(D_{J})+L_{J}(D_{J})) \]

\noindent
such that $\phi_{-}(1)=U_{b}^{(v)} - U_{c}^{(v)}$, where $v$ is the 4-valent vertex in $D_{J}$ corresponding to the singularization of the crossing $c$. The edge map $d_{I,J}$ is given by $\phi_{-} \otimes \id$.

In order to make the edges anti-commute instead of commute, an edge assignment is required. 

\begin{defn}

An edge assignment $\epsilon$ is a function from the edges of the cube to $\Z_{2}$ such that for each 2-dimensional face, the sum of the assignments of the four edges is $1 \in \Z_{2}$. If $I \lessdot J$, we denote the sign assignment on the edge from $I$ to $J$ by $\epsilon_{I,J}$.

\end{defn}

We define the edge map $d_{1}$ by 
\[  d_{1} = \sum_{I \lessdot J} (-1)^{\epsilon_{I,J}} d_{I,J} \]

\noindent
The total complex is given by 
\[ C_{2}^{-}(D) = \bigoplus_{I} C_{2}^{-}(D_{I}) \]

\noindent 
with differential $d_{0}+d_{1}$.

This can be generalized to a reduced homology theory as follows. Let $D$ be a diagram in $\mathcal{D^{R}}$ such that $\sm(D)$ is an $l$-component link, and choose points $\mathbf{p}=\{p_{1},...,p_{l}\}$ on edges $e_{i_{1}},...,e_{i_{l}}$ such that each $e_{i_{j}}$ lies on a different component of $\sm(D)$. 

\begin{defn}
 The \emph{reduced} complex $\widehat{C}_{2}(D)$ is defined to be the tensor product
 \[ \widehat{C}_{2}(D) =  C_{2}^{-}(D) \otimes \bigotimes_{j=1}^{l} R \xrightarrow{U_{i_{j}}} R \]
 
\end{defn}

\noindent
These complexes comes equipped with a grading a relative grading $\gr_{2}$, as well as the filtration induced by the cube. The height of a vertex $I$ in the cube is given by 
\[ |I| = \sum_{c} I(c) \] 
In the reduced complex, the maps $R \xrightarrow{U_{i_{j}}} R$ are defined to have cube grading $1$.

The grading $\gr_{2}$ is defined so that both the differential and multiplication by $U_{i}$ are homogeneous of degree $-2$.  The entire factor $\mathcal{L}_{D}^{+}$ lies in the same grading, so it suffices to specify the grading of the generator $1 \in R/(N(D_{I})+L_{I}(D_{I}))$. Let $|v_{4}(D_{I})|$ denote the number of 4-valent vertices in $D_{I}$. Then up to a grading shift (since it is a relative grading),
\[ \gr_{2}( 1 \in R/(N(D_{I})+L_{I}(D_{I}))) = |v_{4}(D_{I})| -|I| \]

\noindent
It can be checked that the differential is homogeneous of degree $-2$ with respect to this grading, so it is consistent with the relative version.

\begin{rem}

A more natural definition for $C_{2}^{-}(D_{I})$ would be to treat all of the 4-valent vertices the same by assigning the complex
\[ R/N(D_{I}) \otimes \mcL^{+} \]

\noindent
where
\[ \mathcal{L}^{+} = \bigotimes_{v \in v_{4}(D_{I})}  \xymatrix{R\ar@<1ex>[r]^{L(v)}&R\ar@<1ex>[l]^{L^{+}(v)}} \]

For diagrams in $\mathcal{D^{R}}$, this complex is quasi-isomorphic to the one we define. However, it is not clear how to define the edge maps on this larger complex. Essentially, it is the problem of lifting a map from the tensor product
$R/N(S) \otimes R/L_{I}(S)$ to the derived tensor product $R/N(S) \otimes ^{L} R/L_{I}(S)$. If this could be accomplished, then the construction would work for any braid diagram and not just those in $\mathcal{D^{R}}$.

\end{rem}

\section{Background}

In this section, we will give background on Khovanov homology, knot Floer homology, and the theory $\hfk_{2}$.

\subsection{Khovanov homology} We will start with background on Khovanov homology, including the minus version of the theory as defined in \cite{alishahi2017lee} and the pointed Khovanov homology of Baldwin, Levine, and Sarkar \cite{BLS}.

\subsubsection{The Khovanov complex}
 Let $L$ be a link in $S^{3}$ with diagram $D \subset \mathbb{R}^{2}$. Let $c(D)$ denote the crossings in $D$, and viewing  $D$ as a 4-valent graph, let $E=\{e_{1}, e_{2}, ..., e_{m}\}$ denote the edges of $D$. The \emph{ Khovanov edge ring} is defined to be 
\[  R_{X} := \mathbb{Q}[X_{1}, X_{2},...,X_{m}]/\{X_{1}^{2}=X_{2}^{2}=...=X_{m}^{2}=0 \}   \]

\noindent
with each variable $X_{i}$ corresponding to the edge $e_{i}$. 

While our construction is an oriented cube of resolutions complex, the Khovanov complex is an \emph{unoriented} cube of resolutions. Each crossing $c$ can be resolved in two ways, the 0-resolution and the 1-resolution (see Figure \ref{resolutions2}). For each $I: c(D) \to \{0,1\}$, let $D_{I}$ denote the diagram obtained by replacing the crossing $c$ with the $I(c)$-resolution. The diagram $D_{I}$ is a disjoint union of circles - denote the number of circles by $k_{I}$. The vertex $I$ determines an equivalence relation on $E$, where $e_{p} \sim_{I} e_{q}$ if $e_{p}$ and $e_{q}$ lie on the same component of $D_{I}$.

\begin{figure}[ht]
\vspace{4mm}
\centering
\begin{overpic}[width = .7\textwidth]{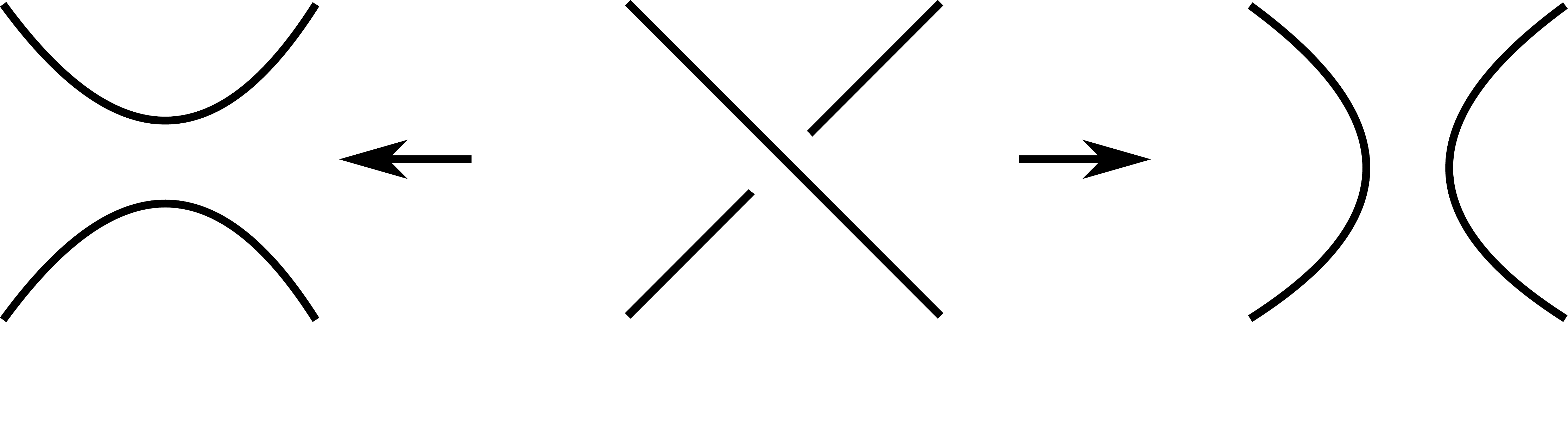}
\put(9.5, 19){$\bullet$}
\put(9.5, 13.75){$\bullet$}
\put(86.3,16.5){$\bullet$}
\put(91.7,16.5){$\bullet$}
\put(0,29){$e_{a}$}
\put(0,4){$e_{c}$}
\put(19, 29){$e_{b}$}
\put(19,4){$e_{d}$}
\put(0,-2){0-resolution}
\put(40,29){$e_a$}
\put(40,4){$e_c$}
\put(59,29){$e_b$}
\put(59,4){$e_d$}
\put(80,29){$e_a$}
\put(80,4){$e_c$}
\put(99,29){$e_b$}
\put(99,4){$e_d$}
\put(80,-2){1-resolution}
\end{overpic}
\caption{}\label{resolutions2}
\end{figure}

The module $C_{Kh}(D_{v})$ is defined to be a quotient of the ground ring:
\[ CKh(D_{I}) := R_{X} / \{X_{p}=X_{q} \text{ if } e_{p} \sim_{I} e_{q} \}  \]

There is a partial ordering on the vertices of the cube obtained by setting $I \le J$ if $I(c) \le J(c)$ for all $c$. As in the previous section, we will write  $I \lessdot J$ if  $I \le J$ and they differ at a single crossing, i.e. there is some $c_{0}$ for which $I(c_{0})=0$ and $J(c_{0})=1$, and $I(c)=J(c)$ for all $c \ne c_{0}$. Corresponding to each edge of the cube, i.e. a pair $(I \lessdot J)$, there is an embedded cobordism in $\mathbb{R}^{2} \times [0,1]$ from $D_{I}$ to $D_{J}$ constructed by attaching a 1-handle near the crossing $c$ where $I(c)<J(c)$. This cobordism is always a pair of pants, either going from one circle to two circles (when $k_{I}=k_{J}-1$) or from two circles to one circle (when $k_{I}=k_{J}+1$). We call the former a \emph{merge} cobordism and the latter a \emph{split} cobordism.

For each vertex $I$ of the cube, the complex $CKh(D_{I})$ is naturally isomorphic to $\mathcal{A}^{\otimes k_{I}}$, where $\mathcal{A}$ is the Frobenius algebra $\mathbb{Q}[X]/(X^{2}=0)$. The multiplication and comultiplication maps of $\mathcal{A}$ are:

\begin{displaymath}
m:\mathcal{A}\otimes_{\mathbb{Q}}\mathcal{A}\rightarrow\mathcal{A}:\begin{cases}
\begin{array}{lcc}
1 \mapsto 1&,&X_{1} \mapsto X\\
X_{1}X_{2} \mapsto 0&,&X_{2} \mapsto X
\end{array}
\end{cases}
\end{displaymath}
and 
\begin{displaymath}
\Delta:\mathcal{A}\rightarrow\mathcal{A}\otimes_{\mathbb{Q}}\mathcal{A}:\begin{cases}
1 \mapsto X_{1}+X_{2}\\
X \mapsto X_{1}X_{2}
\end{cases}
\end{displaymath}




The chain complex $CKh(D)$ is defined to be the direct sum of the $CKh(D_{I})$ over all vertices in the cube:

\[ CKh(D) := \bigoplus_{I} CKh(D_{I}) \]

The differential decomposes over the edges of the cube. When $I \lessdot J$ corresponds to a merge cobordism, define 
\[   \partial_{I,J}: CKh(D_{I}) \to CKh(D_{J})    \]

\noindent
to be the Frobenius multiplication map, and when $I \lessdot J$ corresponds to a split cobordism, define $\partial_{I,J}$ to be the comultiplication map. Note that as $R_{X}$-module maps, the map $m$ is projection, while $\Delta$ is multiplication by $X_{b}+X_{c}$, where $e_{a}$, $e_{b}$, $e_{c}$, $e_{d}$ are the edges at the corresponding crossing as in Figure \ref{resolutions2}. Note that $X_{b}+X_{c}=X_{a}+X_{d}$.

Let $\epsilon$ be a sign assignment on the cube. Then 
\[   \partial = \sum_{I \lessdot J} (-1)^{\epsilon_{I,J}} \partial_{I,J}   \]

The Khovanov complex is bigraded, with a homological grading and a quantum grading. Up to an overall grading shift, the homological grading is just the height in the cube. Setting $n_{+}$ equal to the number of positive crossings in $D$ and $n_{-}$ the number of negative crossings in $D$, we have 
\[ \mathrm{gr}_{h}(CKh(D_{I})) = |I|-n_{-}\]

For each vertex $v$ of the cube, the quantum grading of the generator $1$ of $CKh(D_{I})$ is given by 
\[ \mathrm{gr}_{q}(1 \in CKh(D_{I})) = n_{+}-2n_{-}+|I|+k_{I} \]

\noindent
and each variable $X_{i}$ has quantum grading $-2$. With respect to the bigrading $(\mathrm{gr}_{h}, \mathrm{gr}_{q})$, the differential $\partial$ has bigrading $(1,0)$. The Khovanov homology $Kh(L)$ is the homology of this complex
\[ Kh(L) = H_{*}(CKh(D), \partial)       \]

\noindent
Since the homology does not depend on the diagram $D$ for $L$, we write $Kh(L)$ instead of $Kh(D)$. 

\begin{rem}

The Khovanov complex typically uses an explicit sign assignment instead of an arbitrary one, but it is well-known that any two sign assignments give chain homotopy equivalent complexes, where the chain homotopy equivalence is obtained by changing the sign of basis elements at each vertex. The more general framework will make it easier to compare the Khovanov complex to our construction.

\end{rem}

\subsubsection{Various versions of Khovanov homology} Making simple modifications to the ground ring results in different versions of Khovanov homology. Let \[ R^{-}_{X} = \mathbb{Q}[X_{1}, X_{2},...,X_{m}]/\{X_{1}^{2}=X_{2}^{2}=...=X_{m}^{2} \}  \text{ and } \overline{R}_{X} = R^{-}_{X}/X_{1}=0\]

\begin{defn}

The minus complex $CKh^{-}(D)$ is obtained by repeating the construction of $CKh(D)$ over the ground ring $R^{-}_{X}$. The associated homology is denoted $Kh^{-}(L)$.

\end{defn}

\begin{defn}

The reduced complex $\overline{CKh}(D)$ is obtained by repeating the construction of $CKh(D)$ over the ground ring $\overline{R}_{X}$. The reduced homology is denoted $\overline{Kh}(L)$.

\end{defn}

Note that $\overline{CKh}(D)$ is quasi isomorphic to the complex \[CKh^{-}(D) \otimes (R^{-}_{X}\{-1,-1\} \xrightarrow{X_{1}} R^{-}_{X}\{0,1\})\]

\noindent
where $R^{-}_{X}\{i,j\}$ is a copy of the ground ring $R^{-}_{X}$ whose generator $1$ has bigrading $(i,j)$. This homology is not a link invariant, but rather an invariant of the link $L$ together with a marked component (the one on which $e_{1}$ lies), so the reduced homology is actually only an invariant for knots.

In \cite{BLS}, Baldwin, Levine, and Sarkar generalize the notion of reduced to give a link invariant called pointed Khovanov homology. Let $D$ be a diagram for a link $L$, and let $\mathbf{p}=\{p_{1},...,p_{k} \}$ be a set of basepoints on the edges of $D$, with $p_{j}$ on edge $e_{i_{j}}$. The data $(L, \mathbf{p})$ is called a \emph{pointed link}.

\begin{defn}

Let $D$ be a diagram for a pointed link $(L, \mathbf{p})$. The \emph{pointed Khovanov homology} $Kh(L, \mathbf{p})$ is defined to be the homology of the following complex:
\[ CKh(D, \mathbf{p}) = CKh^{-}(D) \otimes \bigotimes_{j=1}^{k}  (R_{X}^{-}\{-1,-1\}  \xrightarrow{X_{i_{j}}} R_{X}^{-}\{0,1\} )   \]

\end{defn}

\begin{rem}

Our convention of having the reducing maps be edge maps (i.e. have homological grading 1 instead of 0) is different from that of \cite{BLS}, but it fits better with the framework we will be working with.

\end{rem}

Let $(L, \mathbf{p})$ be an $l$-component pointed link with exactly one $p_{j}$ on each component. The homology $Kh(L, \mathbf{p})$ is a link invariant, i.e. it doesn't depend on the location of the basepoints. We will denote this homology $\widehat{Kh}(L)$.

\subsubsection{Moving to the ground ring $R$} Khovanov homology is a module over $\Q[X_{1},...,X_{m}]$ with various relations, depending on the version. However, for all of them we can simply view it as a $\Q[X_{1},...,X_{m}]$-module. In order to relate Khovanov homology to $\widehat{C}_{2}(D)$, we will need to view it instead as a module over $R=\Q[U_{1},...,U_{m}]$. 

Suppose $D$ is a braid diagram on $N$ strands, numbered $1,...,N$ from left to right. For $i=1,...,m$ let $\mathsf{str}(i)$ denote the strand on which $e_{i}$ lies. We define the action of $U_{i}$ by 
\[ U_{i} = (-1)^{\str(i)} X_{i} \]

\noindent
Then for each resolution $D_{I}$, the complex $CKh^{-}(D_{I})$ is given by 
\[ CKh^{-}(D_{I}) = R/ \{U_{i}^{2}=U_{j}^{2} \text{ for all }i,j,\text{ and } U_{i}=U^{\str(i)-\str(j)}_{j} \text{ if } e_{i} \sim_{I} e_{j} \} \]

\noindent
The merge map is multiplication, while the split map is multiplication by $(-1)^{\str(b)}(U_{b} - U_{c})$. The factor $(-1)^{\str(b)}$ can be absorbed into the sign assignment; for each edge $I \lessdot J$, define 
\[ \epsilon_{I,J}' = \begin{cases} 
      \epsilon_{I,J} & \text{ if the edge is a merge map} \\
      \epsilon_{I,J}+\str(b) & \text{ if the edge is a split map}
   \end{cases} \]
   
\noindent
Then the differential can be written
\[   \partial = \sum_{I \lessdot J} (-1)^{\epsilon'_{I,J}} \partial'_{I,J}   \]

\noindent
where $\partial'_{I,J}$ is projection for a merge map and multiplication by $U_{b} - U_{c}$ for a split map.

As with Khovanov homology, the reduced complex is obtained by setting $U_{1}=0$, and the unreduced complex is obtained by setting $U_{1}^{2}=0$. The pointed complex is given by 
\[ CKh(D, \mathbf{p}) = CKh^{-}(D) \otimes \bigotimes_{j=1}^{k}  (R \xrightarrow{U_{i_{j}}} R)   \]

\noindent
where the tensor product is taken over $R$.

\subsection{Knot Floer homology}
We will assume that the reader has a basic knowledge of knot Floer homology -- for background see \cite{OS1}, or \cite{manolescu2014introduction} for a survey paper.

Knot Floer homology is a Lagrangian Floer theory where the symplectic manifold and the two Lagrangians are constructed from a Heegaard diagram for the link $L$. When $L$ is a link in $S^{3}$, this Heegaard diagram can always be on $\Sigma = S^{2}$. We will limit our discussion to these planar Heegaard diagrams.

\subsubsection{Planar Heegaard Diagrams} 

\begin{defn}

A (multipointed) planar Heegaard diagram $\mathcal{H}$ for a link $L$ in $S^{3}$ is a tuple $(\Sigma, \alpha, \beta, \bf{O}, \bf{X})$ where 

$\bullet$ $\Sigma = S^{2}$ 

$\bullet$ $\alpha$ (resp. $\beta$) is a set of disjoint embedded circles $\alpha_{1},..., \alpha_{k-1}$ (resp. $\beta_{1},...,\beta_{k-1}$) such that the $\{\alpha_{i}\}$ and $\{ \beta_{i} \}$ intersect transversely.

$\bullet$ $\bf{O}$ (resp. $\bf{X})$ is a set of basepoints $O_{1},...,O_{k}$ (resp. $X_{1},...,X_{k}$) such that each component of $\Sigma  \backslash \alpha$ contains both an $\bf{O}$ basepoint and an $\bf{X}$ basepoint, and similarly for each component of $\Sigma  \backslash \beta$.

\end{defn}

The Heegaard diagram $\cH$ determines a $k$-bridge presentation for $L$ as follows: each $X$ basepoint can be connected to a unique $O$ basepoint by an arc in $S^{2}$ that doesn't cross any $\alpha$ circles -- we will call this an $\alpha$-\emph{side arc}. Similarly, each $O$ basepoint can be connected to a unique $X$ basepoint by a $\beta$-\emph{side arc}, or an arc that doesn't cross any $\beta$ circles. If we join these arcs together and specify that the $\alpha$-side arcs pass under the $\beta$-side arcs, this gives a $k$-bridge presentation for a link $L$. It is oriented so that each $\alpha$-side arc is oriented from $O$ to $X$ and each $\beta$-side arc is oriented from $X$ to $O$.

An important piece of data that we'll need to keep track of for this construction is which two $O$ basepoints are connected to each $X$ basepoint. Let $O_{a(i)}$ be the basepoint connected to $X_{i}$ on the $\alpha$ side and $O_{b(i)}$ the basepoint connected to $X_{i}$ on the $\beta$ side.

\begin{defn}

A \emph{punctured} Heegaard diagram for a knot or link is a Heegaard diagram with three extra pieces of data: $p, \alpha_{g+k}$, and $\beta_{g+k}$. The point $p \in \Sigma$ can be viewed as a puncture on $S^{2}$, and the extra $\alpha$ an $\beta$ circles are required to separate the $\bf{O}$ and $\bf{X}$ basepoints from $p$. 

\end{defn}

\begin{figure}[h!]
\tiny
\begin{subfigure}{.5\textwidth}
 \centering
\def\svgwidth{5cm}
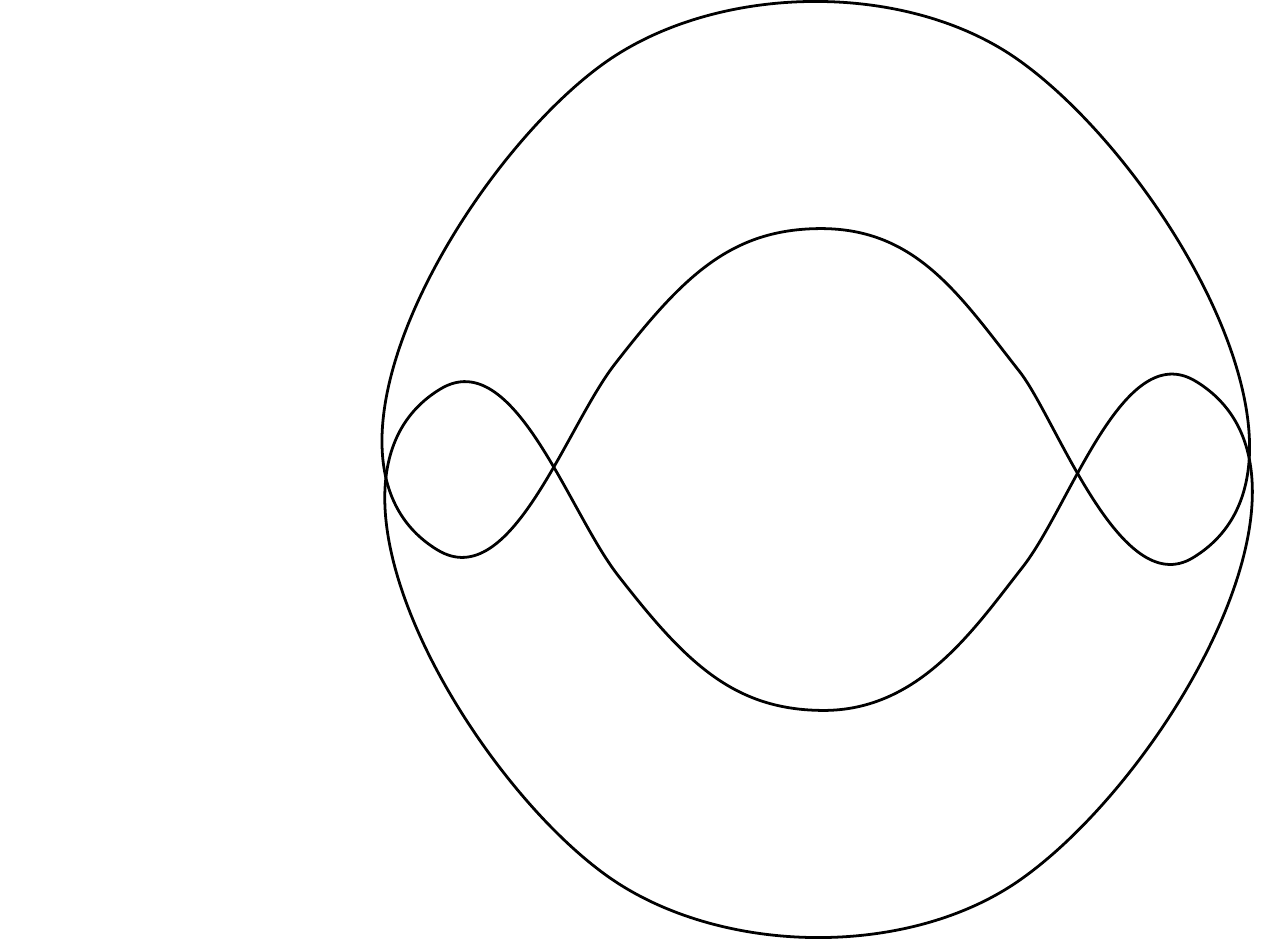
\end{subfigure}%
\begin{subfigure}{.5\textwidth}
  \centering
  \def\svgwidth{5cm}
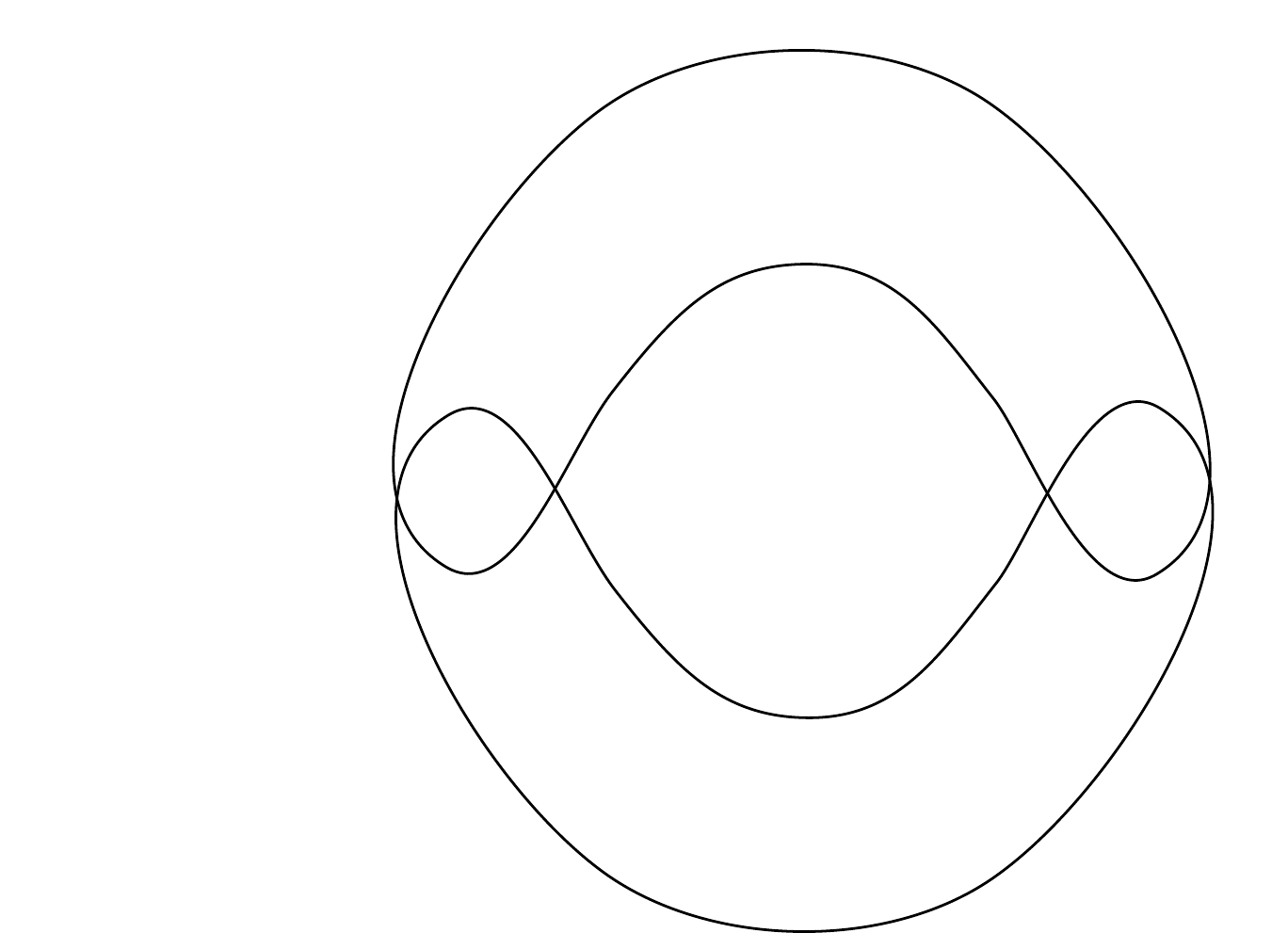
\end{subfigure}
\caption{Unpunctured (left) and punctured (right) Heegaard diagrams for the Hopf link}\label{something}
\end{figure}

\begin{rem} \label{remark1} 

A punctured Heegaard diagram for $L$  can be viewed as a Heegaard diagram for a the link $L\sqcup (\unknot)$ for which the extra unknotted component contains only the basepoints $O_{0}$ and $X_{0}$, and $O_{0}=X_{0}$. The point $p$ is recording the location of these two basepoints. This puncture comes up in \cite{BLS} as well when they are moving from a reduced to an unreduced theory.

\end{rem}

\subsubsection{Knot Floer Chain Complexes} All of our complexes will be defined over $\Q$. Suppose $\cH$ is an unpunctured Heegaard diagram. The corresponding symplectic manifold is $\Sym_{k-1}(\Sigma)$, with Lagrangians $\T_{\alpha}=\alpha_{1} \times \alpha_{2} \times... \times \alpha_{k-1}$ and $\T_{\beta} = \beta_{1} \times \beta_{2} \times... \times \beta_{k-1}$. 

\begin{defn}

The \emph{minus complex} for knot Floer homology $\cfk^{-}(\cH)$ is freely generated over $\Q[U_{1},...,U_{k}]$ by the intersection points $\T_{\alpha} \cap \T_{\beta}$. The differential is given by 

\[  \partial^{-}(x) = \sum_{y \in \mathbb{T_{\alpha}} \cap \mathbb{T_{\beta}}} \sum_{\substack{\phi \in \pi_{2}(x,y) \\ \mu(\phi)=1 \\ n_{\bf{X}}(\phi)=0}}    \# \widehat{\mathcal{M}}(\phi) U_{1}^{n_{O_{1}}(\phi)}\cdot \cdot \cdot U_{k}^{n_{O_{k}}(\phi)} y \]

\end{defn}

\noindent
In this definition, $n_{q}(\phi)$ counts the multiplicity of the pseudo-holomorphic curve $\phi$ at a point $q \in \Sigma$, $\mu$ is the Maslov index, and $ \#  \widehat{\mathcal{M}}(\phi)$ is a signed count of the number of curves in the moduli space $\mathcal{M}(\phi)$. 

In order to define the differential with signs, we need to choose a trivialization of the determinant line bundle $det(\phi)$ for each class $\phi$ subject to certain compatibility conditions. Such a choice determines an orientation on each moduli space $\widehat{\mathcal{M}}(\phi)$ allowing us to count the discs with signs, and the compatibility conditions guarantee that $(\partial^{-})^{2}=0$. All of these choices together make up a \emph{system of orientations}. It was shown in \cite{Akram} that one can always find a system of orientations in which the $\alpha$-degenerations come with positive sign and the $\beta$-degenerations come with negative sign -- we will always work with such a system of orientations. Any two systems of orientations with this property are strongly equivalent, so they have the same chain homotopy type \cite{SucharitSign}.

It was shown by Oszv\'{a}th and Szab\'{o} that for any link $L$, the chain homotopy type of $\cfk^{-}(\cH)$ does not depend on the choice of Heegaard diagram for $L$. The homology of this complex is denoted $\hfk^{-}(L)$ and by a minor abuse of notation, the chain complex is denoted $\cfk^{-}(L)$.

The complex comes equipped with two gradings, the Maslov grading $M$ and the Alexander grading $A$. In this paper we will only define them as relative gradings, though absolute versions can be pinned down with some extra work. If $\phi \in \pi_{2}(x,y)$, then
\[ M(x) - M(y) = \mu(\phi) - 2n_{\bf{O}}(\phi) \hspace{3mm} \text{ and }\hspace{3mm}A(x) - A(y) = n_{\bf{X}}(\phi) - n_{\bf{O}}(\phi) \]

\noindent
Multiplication by any of the $U_{i}$ lowers Maslov grading by 2 and lowers Alexander grading by 1. It follows from these definitions that the differential $\partial^{-}$ drops Maslov grading by 1 and preserves Alexander grading.

The same theory can be defined for punctured Heegaard diagrams, with the restriction that we don't count discs which pass through the puncture $p$:

\[  \partial^{-}(x) = \sum_{y \in \mathbb{T_{\alpha}} \cap \mathbb{T_{\beta}}} \sum_{\substack{\phi \in \pi_{2}(x,y) \\ \mu(\phi)=1 \\ n_{\bf{X}}(\phi)=0 \\ n_{p}=0} }   \# \widehat{\mathcal{M}}(\phi) U_{1}^{n_{O_{1}}(\phi)}\cdot \cdot \cdot U_{k}^{n_{O_{k}}(\phi)} y \]

\begin{defn}

Let $L$ be a link in $S^{3}$. The \emph{unreduced} knot Floer complex is given by $\cfk(L) = \cfk(\cH)$, where $\cH$ is a punctured Heegaard diagram for $L$. The unreduced homology is denoted $\hfk(L)$.

\end{defn}

Let $C\{i,j\}$ denote the complex $C$ shifted in Maslov grading by $i$ and Alexander grading by $j$.

\begin{lem}\label{lem2.7}

Let $\cH$ and $\cH'$ be a punctured and an unpunctured Heegaard diagram, respectively, for the same link $L$. Then 

\[ \hfk(\cH) \cong \hfk(\cH') \otimes (R\{0,0\} \oplus R\{-1,0\}) \]

\noindent
where the tensor product is taken over $R$.

\end{lem}

There is also a reduced theory:

\begin{defn}

Let $L$ be a $l$-component link, and suppose we order the $O_{i}$ such that $O_{1},...,O_{l}$ each lie on a different component. Then the reduced complex $\widehat{\cfk}(L)$ is given by 

\[ \widehat{\cfk}(L) = \cfk^{-}(L) \otimes \left( \bigotimes_{i=1}^{l} R \xrightarrow{U_{i}} R \right) \]

\noindent
The homology of this complex, i.e. the reduced knot Floer homology, is denoted $\widehat{\hfk}(L)$.

\end{defn}

This homology is well-defined because if $O_{i}$ and $O_{j}$ lie on the same component of $L$, then the actions of $U_{i}$ and $U_{j}$ on $\cfk^{-}(L)$ are homotopic. This is true on $\cfk(L)$ as well.

\subsection{The oriented cube of resolutions for knot Floer homology} \label{cube}

In \cite{Szabo}, Oszv\'{a}th and Szab\'{o} constructed an oriented cube of resolutions for knot Floer homology. This complex was defined with twisted coefficients, i.e. over $R[t^{-1}, t ]]$, in order to make it particularly well-behaved. However, the additional $t$ variables make it harder to compare the homology to the quantum invariants, so in this section we will describe the $t=1$ specialization defined by Manolescu in \cite{Manolescu}.

Let $D$ be a decorated braid diagram for a link $L$, and let $\cH$ be the Heegaard diagram for $L$ which has the diagram in Figure \ref{HDCrossing} at each crossing. We also allow the diagram $D$ to have bivalent vertices -- the Heegaard diagram at each bivalent vertex is shown in Figure \ref{bivalentHD}. The marked edge will be treated as a special bivalent vertex splitting the edge into two edges, which will always be labeled $e_{1}$ and $e_{2}$. The Heegaard diagram at the marked edge is shown in Figure \ref{MarkedEdgeHD}.

 Knot Floer homology is typically defined using just the information in a Heegaard diagram, but there is no known diagram for which the associated chain complex splits as an oriented cube of resolutions. However, using the diagrams in Figure \ref{crossings} together with some algebraic techniques, there is larger complex $C^{-}_{F}(D)$ which is chain homotopy equivalent to $\mathit{CFK}^{-}(L)$ such that $C^{-}_{F}(D)$ does split into a cube of resolutions.

\begin{figure}[h!]
\tiny
\begin{subfigure}{.5\textwidth}
 \centering
   \begin{overpic}[width=.9\textwidth]{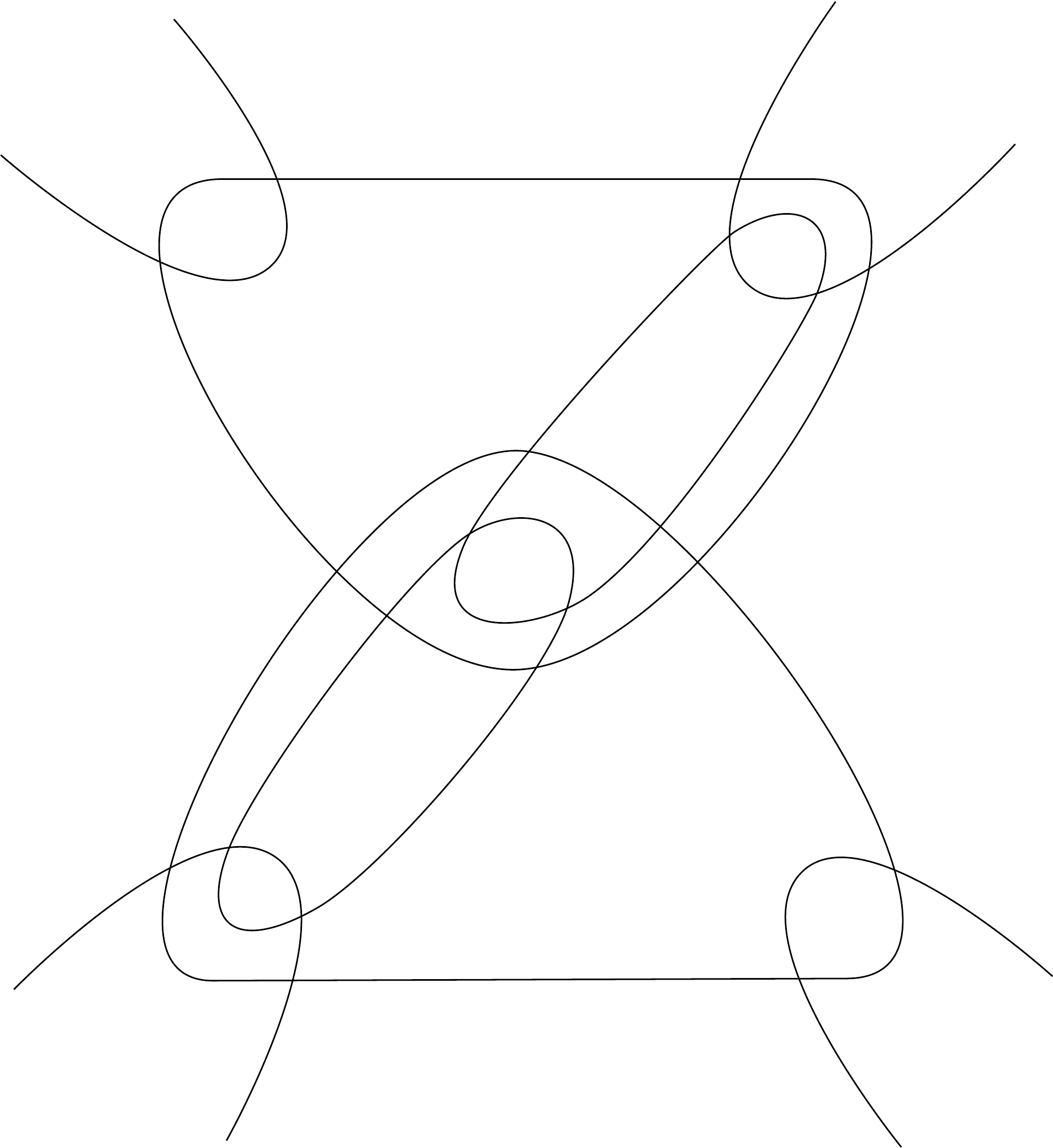}
   \put(17.6,79){$O_{a}$}
   \put(66,77){$O_{b}$}
   \put(43,49.7){$X$}
   \put(53,47.2){$X$}
   \put(21,21.6){$O_{c}$}
   \put(71,19){$O_{d}$}
   \put(43,86){$\alpha_{1}$}
   \put(50,70){$\alpha_{2}$}
   \put(45,12){$\beta_{1}$}
   \put(40,30.4){$\beta_{2}$}
   \end{overpic}
  \caption{Positive crossing}  \label{fig2a}
\end{subfigure}%
\begin{subfigure}{.5\textwidth}
  \centering
   \begin{overpic}[width=.9\textwidth]{initial_diagram.pdf}
   \put(17.6,79){$O_{a}$}
   \put(66,77){$O_{b}$}
   \put(43,49.7){$X$}
   \put(34,52){$X$}
   \put(21,21.6){$O_{c}$}
   \put(71,19){$O_{d}$}
   \put(43,86){$\alpha_{1}$}
   \put(50,70){$\alpha_{2}$}
   \put(45,12){$\beta_{1}$}
   \put(40,30.4){$\beta_{2}$}
   \end{overpic}
  \hspace{7mm} \caption{Negative crossing}
  \label{fig2b}
\end{subfigure}
\caption{The Heegaard diagrams at positive and negative crossings}\label{HDCrossing}
\end{figure}

\begin{figure}[h!]
\scriptsize
\begin{subfigure}{.5\textwidth}
 \centering
\def\svgwidth{1.5cm}
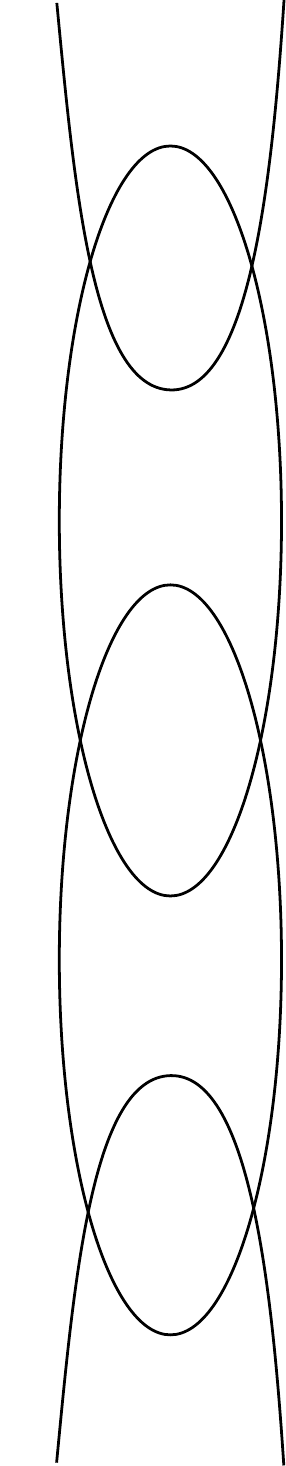
  \caption{A bivalent vertex}  \label{bivalentHD}
\end{subfigure}%
\begin{subfigure}{.5\textwidth}
  \centering
\def\svgwidth{1.5cm}
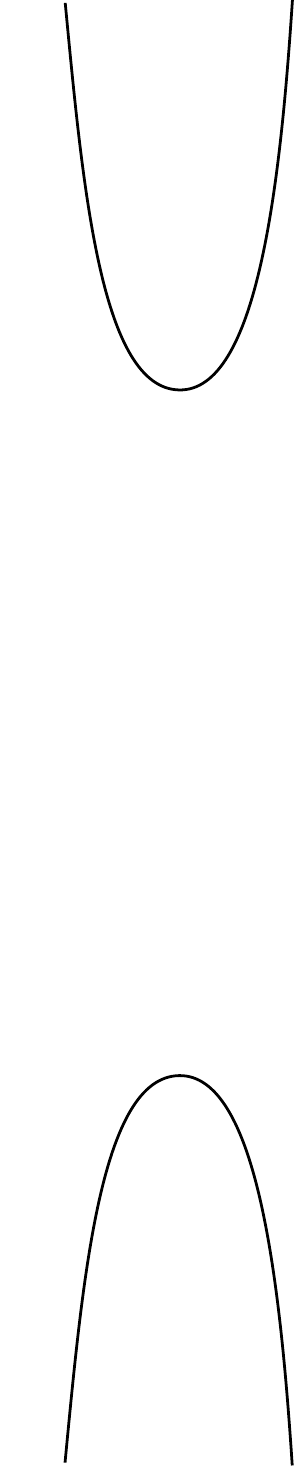
\caption{The decorated edge}
  \label{MarkedEdgeHD}
\end{subfigure}
\caption{The Heegaard diagrams at a bivalent vertex and the decorated edge}
\end{figure}

Note that the diagrams for the two crossings use the same $\alpha$ and $\beta$ curves, but differ in the location of one of the $X$ basepoints. The diagram in Figure \ref{HDCrossing10} is thus able to include the data from both the positive and the negative crossings, as well as the oriented smoothing. More specifically, taking the $A_{0}$ and $A^{+}$ labels to denote $X$ basepoints gives a positive crossing, taking $A_{0}$ and $A^{-}$ to denote $X$ basepoints gives a negative crossing, and taking the two $B$ labels to be $X$ basepoints gives the oriented smoothing. Given a braid diagram $D$ with a choice of crossing $c$, let $D_{+}$ denote the diagram where that crossing is a positive crossing, $D_{-}$ the diagram where it is a negative crossing, $D_{s}$ the oriented resolution, and $D_{\x}$ the singularization.

\begin{figure}[h!] 

 \centering
   \begin{overpic}[width=.75\textwidth]{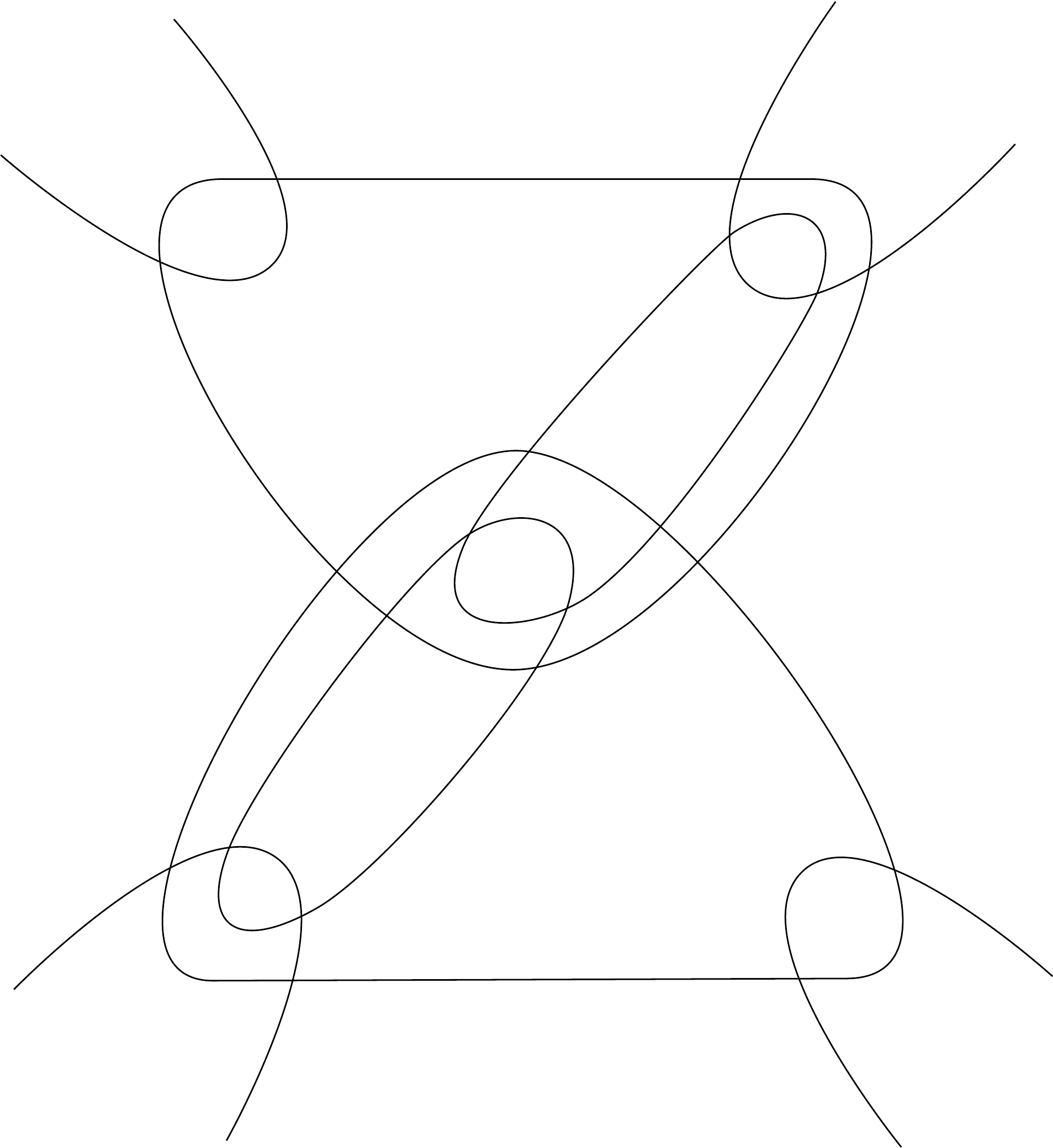}
   \put(39,54){$x$}
   \put(40.4,53){$\bullet$}
   \put(49.3,45){$x'$}
   \put(48.5,46.3){$\bullet$}
   \put(17.6,79){$O_{a}$}
   \put(66,77){$O_{b}$}
   \put(48,55.7){$B$}
   \put(43,49.7){$A_{0}$}
   \put(38.2,44){$B$}
   \put(34,52){$A_{-}$}
   \put(53,48){$A_{+}$}
   \put(21,21.6){$O_{c}$}
   \put(71,19){$O_{d}$}
   \put(43,86){$\alpha_{1}$}
   \put(50,70){$\alpha_{2}$}
   \put(45,12){$\beta_{1}$}
   \put(40,30.4){$\beta_{2}$}
   \end{overpic}
\caption{The Heegaard diagram at a crossing}\label{HDCrossing10}
\end{figure}

\subsubsection{The negative crossing} We will start by defining the cube of resolutions for a negative crossing. Since there are $X$ basepoints at $A_{0}$ and $A^{-}$, the intersection point $x$ becomes special, as it induces a filtration on the complex. Let $\X$ denote the component of the knot Floer complex with generators containing the $x$ intersection point, and let $\Y$ denote the component of the complex with generators which do not contain $x$. If we let $\Phi_{B}$ denote those differentials with multiplicity $0$ at $A_{0}$ and $A^{-}$ and multiplicity $1$ at one of the $B$'s, then 
\[ \mathit{CFK}^{-}(D_{-}) \cong \Y \xrightarrow{\hspace{3mm}\Phi_{B}\hspace{3mm}} \X \]

The cube complex $C^{-}_{F}(D_{-})$ is defined to be the complex in Figure \ref{negcomplex}, where $\Phi_{A^{-}}$ counts discs with multiplicity $0$ at the $B$ basepoints and $1$ at $A_{0}$ or $A^{-}$ and $\Phi_{A^{-}B}$ counts discs with multiplicity $1$ at one of the $B$ basepoints and multiplicity $1$ at $A_{0}$ or $A^{-}$ (and multiplicity $0$ at the other in both cases). The proof that $d^{2}=0$ comes from counting Maslov index $2$ degenerations (see \cite{Szabo}, p. 33).

\begin{figure}[!h]
\centering
\begin{tikzpicture}
  \matrix (m) [matrix of math nodes,row sep=5em,column sep=6em,minimum width=2em] {
     \X & \X \\
     \Y & \X \\};
  \path[-stealth]
    (m-1-1) edge node [left] {$\Phi_{A^{-}}$} (m-2-1)
            edge node [above] {$1$} (m-1-2)
            edge node [right]{$\Phi_{A^{-}B}$} (m-2-2)
    (m-2-1.east|-m-2-2) edge node [above] {$\Phi_{B}$} (m-2-2)
    (m-1-2) edge node [right] {$U_{a}+U_{b}-U_{c}-U_{d}$ \hspace{ 15mm}} (m-2-2);
\end{tikzpicture}
\caption{Complex for the negative crossing} \label{negcomplex}
\end{figure}

Note that $C_{F}(D_{-})$ admits two filtrations - a horizontal filtration and a vertical filtration (the filtrations induced by the horizontal and vertical coordinates, respectively). Using the vertical filtration, we see that 

\[C^{-}_{F}(D_{-}) \cong \mathit{CFK}^{-}(D_{-}) \]

\noindent
where the chain homotopy equivalence comes from contracting the quotient complex
\[ \X \xrightarrow{\hspace{3mm}1\hspace{3mm}} \X \]

The interesting filtration on this complex is the horizontal filtration - this filtration corresponds to height in the cube. The complex in the lower filtration level is given by 
\[ \Y \xrightarrow{\hspace{3mm}\Phi_{A^{-}}\hspace{3mm}} \X \]

\noindent
We define $C^{-}_{F}(D_{s})$ to be this complex. Note that this is exactly the complex obtained by placing $X$ basepoints at the $B$ markings, so it gives the knot Floer homology of the oriented smoothing, $\mathit{CFK}^{-}(D_{s})$. 

The complex in the higher filtration level is given by 
\[ \X \xrightarrow{\hspace{1mm}U_{a}+U_{b}-U_{c}-U_{d}\hspace{1mm}} \X \]

\noindent
The chain complex $\X$ is the knot Floer complex of the Heegaard diagram complex obtained by deleting $\alpha_{2}$ and $\beta_{2}$ from the Heegaard diagram, giving the diagram in Figure \ref{singdiagram}. This is the Heegaard diagram used to describe the singularization $D_{\x}$. The differential is analogous to standard Heegaard diagrams, with the $XX$ basepoint being blocked

\[  \partial^{-}(x) = \sum_{y \in \mathbb{T_{\alpha}} \cap \mathbb{T_{\beta}}} \sum_{\substack{\phi \in \pi_{2}(x,y) \\ \mu(\phi)=1 \\ n_{\bf{X}}(\phi)=0 \\   n_{\bf{XX}}(\phi)=0} }   \# \widehat{\mathcal{M}}(\phi) U_{1}^{n_{w_{1}}(\phi)}\cdot \cdot \cdot U_{k}^{n_{w_{k}}(\phi)} y \]

Since the linear term $U_{a}+U_{b}-U_{c}-U_{d}$ depends only on the four edges adjacent to the crossing and not the sign of the crossing, we define 
\[ C^{-}_{F}(D_{\x}) = \X \xrightarrow{\hspace{1mm}U_{a}+U_{b}-U_{c}-U_{d}\hspace{1mm}} \X \]

\noindent
This gives the cube of resolutions decomposition 
\[ C^{-}_{F}(D_{-}) = C_{F}(D_{s}) \xrightarrow{\hspace{10mm}} C_{F}(D_{\x}) \]

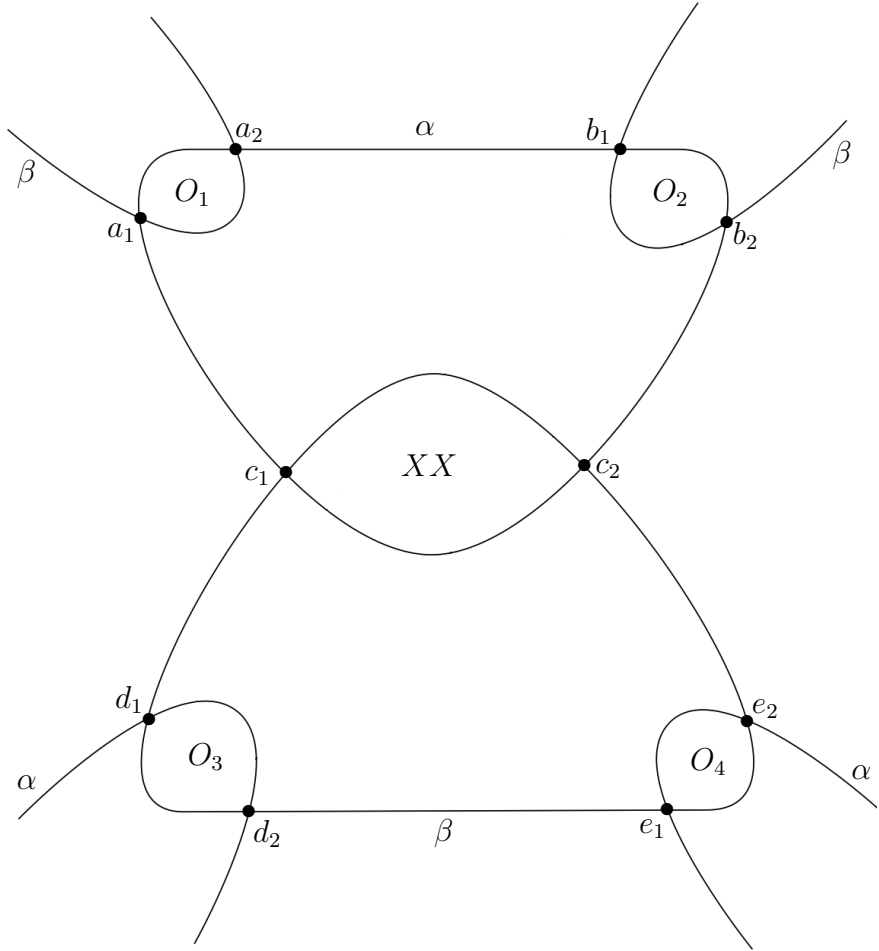
\begin{figure}[!h]
\centering
   \begin{overpic}[width=.7\textwidth]{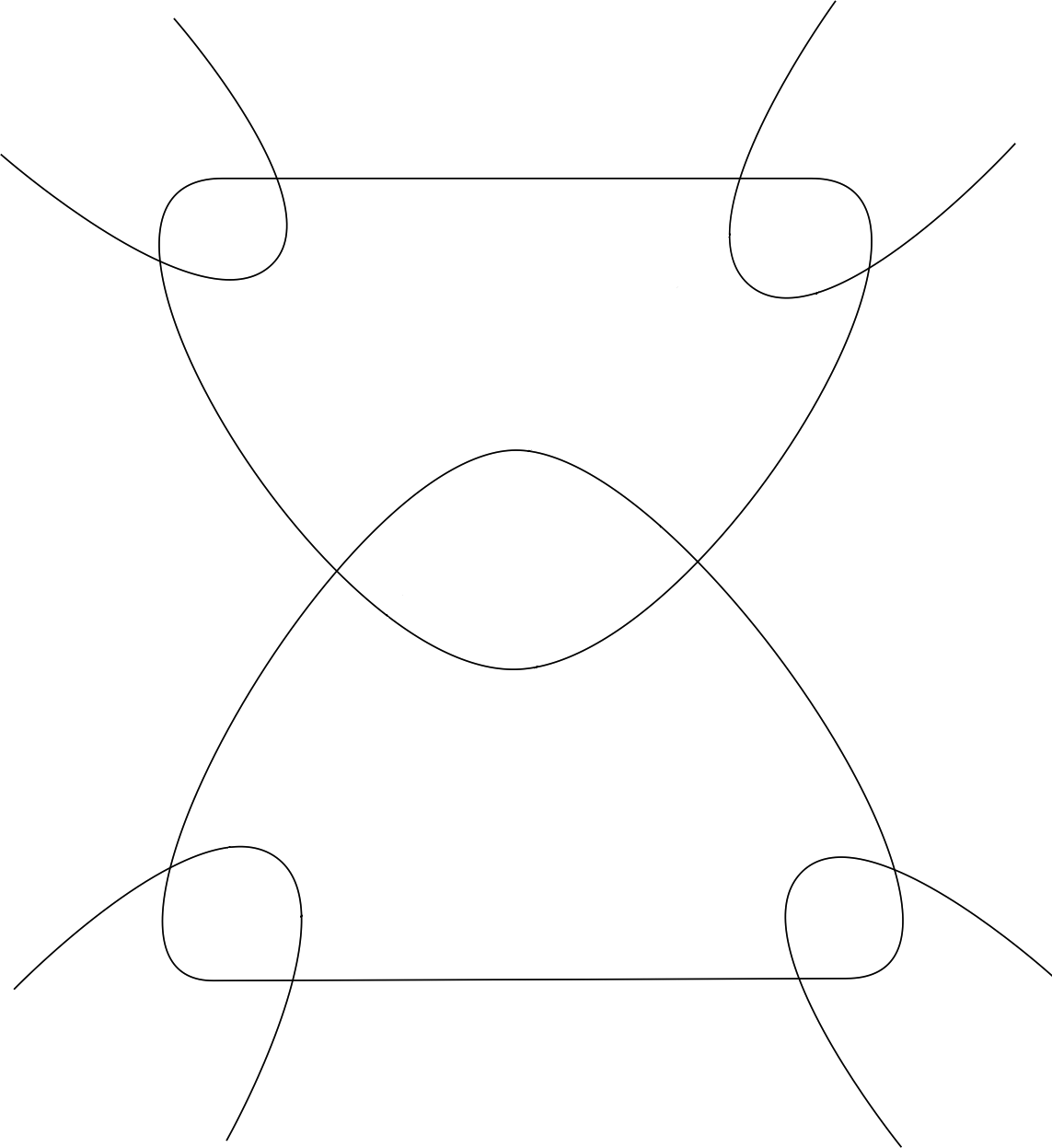} 
   \put(17.6,79){$O_{1}$}
   \put(68,79){$O_{2}$}
   \put(19,19.5){$O_{3}$}
   \put(72,19){$O_{4}$}
   \put(41.5,50){$XX$}
   \put(43,86){$\alpha$}
   \put(1,17){$\alpha$}
   \put(89,18){$\alpha$}
   \put(1,81){$\beta$}
   \put(87,83){$\beta$}
   \put(45,11.5){$\beta$}
   \put(13.2,76.4){$\bullet$}
   \put(23.25,83.7){$\bullet$}
   \put(24,85.8){$a_{2}$}
   \put(10.5,75.2){$a_{1}$}
   \put(63.8,83.7){$\bullet$}
   \put(75,76){$\bullet$}
   \put(61,85.5){$b_{1}$}
   \put(76.5,74.5){$b_{2}$}
   \put(28.55,49.55){$\bullet$}
   \put(25,49.55){$c_{1}$}
   \put(60,50.3){$\bullet$}
   \put(62,50.3){$c_{2}$}
   \put(14.1,23.5){$\bullet$}
   \put(11.35,25.5){$d_{1}$}
   \put(24.6,13.8){$\bullet$}
   \put(25.9,11.34){$d_{2}$}
   \put(68.7,14){$\bullet$}
   \put(66.6,12.3){$e_{1}$}
   \put(77.15,23.3){$\bullet$}
   \put(78.5,25){$e_{2}$}
   \end{overpic}
   \caption{Heegaard diagram for a singularization}\label{singdiagram}
\end{figure}

\subsubsection{The positive crossing} The positive crossing has a similar story, except that the focus is on the intersection point $x'$. Let $\X'$ denote the component of the complex with generators containing the $x'$ intersection point, and let $\Y'$ denote the the component of the complex with generators that do not contain $x'$. Then  
\[ \mathit{CFK}^{-}(D_{+}) \cong \X' \xrightarrow{\hspace{3mm}\Phi'_{B}\hspace{3mm}} \Y' \]

\noindent
where $\Phi'_{B}$ counts discs with multiplicity $1$ at one of the $B$ markings and $0$ at $A_{0}$ and $A^{+}$. We can now define the cube complex for the positive crossing $C^{-}_{F}(D_{+})$ to be the complex in Figure \ref{poscomplex}. The maps are defined analogously to the positive crossing diagram, where $\Phi_{A^{+}}$ counts discs with multiplicity $0$ at the $B$ basepoints and $1$ at $A_{0}$ or $A^{+}$ and $\Phi_{A^{+}B}$ counts discs with multiplicity $1$ at one of the $B$ basepoints and multiplicity $1$ at $A_{0}$ or $A^{+}$ (and multiplicity $0$ at the other in both cases).

\begin{figure}[!h]
\centering
\begin{tikzpicture}
  \matrix (m) [matrix of math nodes,row sep=5em,column sep=6em,minimum width=2em] {
     \X' & \Y' \\
     \X' & \X' \\};
  \path[-stealth]
    (m-1-1) edge node [left] {$U_{a}+U_{b}-U_{c}-U_{d}$} (m-2-1)
            edge node [above] {$\Phi'_{B}$} (m-1-2)
            edge node [right]{$\Phi_{A^{+}B}$} (m-2-2)
    (m-2-1.east|-m-2-2) edge node [above] {1} (m-2-2)
    (m-1-2) edge node [right] {$\Phi_{A^{+}}$ \hspace{ 15mm}} (m-2-2);
\end{tikzpicture}
\caption{Complex for the positive crossing} \label{poscomplex}
\end{figure}

Looking at the vertical filtration, we see that 
\[ C^{-}_{F}(D_{+}) \cong \mathit{CFK}^{-}(D_{+}) \]

\noindent
The horizontal filtration gives us the cube of resolutions. The complex in the higher filtration level is the knot Floer complex of the oriented smoothing
\[ \X' \xrightarrow{\hspace{3mm}\Phi_{A^{+}}\hspace{3mm}} \Y' \]

\noindent
which is isomorphic to $C^{-}_{F}(D_{s})$. The complex in the higher lower level is 
\[ \X' \xrightarrow{\hspace{1mm}U_{a}+U_{b}-U_{c}-U_{d}\hspace{1mm}} \X' \]

\noindent
where $\X'$, like $\X$ in the negative crossing case, is the complex coming from the diagram in Figure \ref{singdiagram}. Thus, the complex in the higher filtration level is precisely $C^{-}_{F}(D_{\x})$. We therefore have the decomposition
\[ C^{-}_{F}(D_{+}) = C^{-}_{F}(D_{\x}) \xrightarrow{\hspace{10mm}} C^{-}_{F}(D_{s}) \]

For a diagram $D$, the cube of resolutions complex $C^{-}_{F}(D)$ is defined to be the complex obtained by iterating this construction over all crossings in the diagram (see \cite{Szabo}, p. 37 for details on how to iterate). For a positive crossing, we call the singularization the 0-resolution of the crossing, and the smoothing the 1-resolution of the crossing. For a negative crossing, the smoothing is the 0-resolution and the singularization is the 1-resolution. As with $C_{2}^{-}(D)$, the height in the cube gives a grading, with respect to which it can be decomposed as 
\[ d = d_{0}+d_{1}+...+d_{k} \]

\noindent
where $d_{j}$ increases the cube grading by $j$. Unlike $C_{2}^{-}(D)$, $d_{i}$ for $i \ge 2$ can be non-zero.

\subsubsection{Reduced and Unreduced Versions} There are also a reduce and unreduced complexes $\overline{C}_{F}(D)$ and $C_{F}(D)$ analogous to the reduced and unreduced knot Floer complexes. If $D$ is a diagram for an $l$-component link and $e_{i_{1}},...,e_{i_{l}}$ each lie on a different component, then

\[ \overline{C}_{F}(D) =C^{-}_{F}(D)  \otimes \left( \bigotimes_{j=1}^{l} R \xrightarrow{U_{i_{j}}} R \right) \]

For the unreduced complex $C_{F}(D)$, simply replace the unpunctured planar diagram in the construction of $C_{F}^{-}(D)$ with a punctured diagram.

\subsubsection{Complete Resolutions and the Non-Local Ideal}

The object at a vertex in the cube is a \emph{complete resolution}, also known as a singular graph. It has only bivalent and 4-valent vertices (no crossings). Given a complete resolution $S$, 

\[ C^{-}_{F}(S) = \cfk^{-}(\cH) \otimes \left( \bigotimes_{v \in v_{4}(S)} R \xrightarrow{ \hspace{3mm} L(v) \hspace{3mm}} R \right) \]

\noindent
where $\cH$ is a Heegaard diagram for $S$ and $L(v)=U^{(v)}_{a}+U^{(v)}_{b}-U^{(v)}_{c}-U^{(v)}_{d}$ as in Section \ref{consection}. The complex $\cfk^{-}(\cH)$ is defined in the usual way, with both $X$ and $XX$ basepoints blocked:

\[  \partial^{-}(x) = \sum_{y \in \mathbb{T_{\alpha}} \cap \mathbb{T_{\beta}}} \sum_{\substack{\phi \in \pi_{2}(x,y) \\ \mu(\phi)=1 \\ n_{\bf{X}}(\phi)=0 \\   n_{\bf{XX}}(\phi)=0} }   \# \widehat{\mathcal{M}}(\phi) U_{1}^{n_{w_{1}}(\phi)}\cdot \cdot \cdot U_{k}^{n_{w_{k}}(\phi)} y \]

\noindent
The relative Maslov and Alexander gradings on this complex are given by 
\[ M(x) - M(y) = \mu(\phi) +2n_{\bf{XX}}(\phi) - 2n_{\bf{O}}(\phi) \]
\[ A(x) - A(y) = n_{\bf{X}}(\phi)+2n_{\bf{XX}}(\phi) - n_{\bf{O}}(\phi) \]

\subsubsection{The Non-local Ideal} Let $\cH$ be an unpunctured Heegaard diagram for $S$. It turns out that when $S$ is a complete resolution, the homology of $\cfk^{-}(\cH)$ can be computed \cite{Szabo, Manolescu}. Viewing $S$ as a graph in $S^{2}$, let $N(S)$  be the non-local ideal defined in Section \ref{consection}, and let $L(S)$ be the ideal generated by $L(v)$ over all $v \in v_{4}(S)$.

\begin{lem}\label{nonloc}

Let $\cH$ be a Heegaard diagram for $S$. Then 
\[H_{*}(\cfk^{-}(\cH)) \cong R/N(S) \]

\end{lem}

A key ingredient in this proof is a grading on the knot Floer complex called the \emph{algebraic grading}, denoted $\gr_{N}$. It is defined by 
\[ \gr_{N}(x) = M(x) - 2 A(x) \]

\noindent
The differential is homogeneous of degree $-1$ with respect to $\gr_{N}$, and $U_{i}$ preserves $\gr_{N}$. Lemma \ref{nonloc} is proved by showing that $H_{*}(\cfk^{-}(\cH))$ lies in a single algebraic grading, then computing the homology in that grading.

The complex $C_{F}^{-}(S)$ is given by  $\cfk^{-}(\cH) \otimes \mathcal{L}(S)$, where 
\[ \mathcal{L}(S) = \bigotimes_{v} R \xrightarrow{ \hspace{3mm} L(v) \hspace{3mm}} R \]

\noindent
The elements $L(v)$ form a regular sequence in $R$, so the homology of $C_{F}^{-}(S)$ can be described in terms of Tor groups.

\begin{cor}

The homology $H_{*}(C_{F}^{-}(S))$ is given by $\Tor_{R}(R/N(S), R/L(S)) $.

\end{cor}

Since $\Tor^{0}$ is just the standard tensor product, the homology in the lowest algebraic grading is given by $R/(N(S)+L(S))$ - in particular, it is generated by $1 \in R$. Oszv\'{a}th and Szab\'{o} compute the higher edge maps among these lowest grading generators, because with twisted coefficients, $\Tor^{0}$ is the only non-trivial term. Since each vertex in the cube has a single generator, the edge maps are uniquely determined by an element of $R$.

Let $D_{I}$ and $D_{J}$ be two connected complete resolutions of a braid diagram $D$ with $I \lessdot J$. The $\Tor^{0}$ portion of the edge maps are described as follows:

\begin{lem}[\hspace{1sp}\cite{Szabo}]

 On $\Tor^{0}$, the positive crossing edge map is multiplication by $(-1)^{\epsilon_{I,J}}$, while the negative crossing edge map is multiplication by $(-1)^{\epsilon_{I,J}}(U_{b}-U_{c})$, where $\epsilon$ is an edge assignment.

\end{lem}

\subsection{The homology theory $\hfk_{2}(L)$}

Given a Heegaard diagram $\cH$, there is a more general complex called the \emph{master complex}, which we will denote $\cfk_{U,V}(\cH)$. The ground ring is $\Q[U_{1},...U_{k},V_{1},...,V_{k}]$, and the differential is given by 

\[  \partial_{U,V}(x) = \sum_{y \in \mathbb{T_{\alpha}} \cap \mathbb{T_{\beta}}} \sum_{\substack{\phi \in \pi_{2}(x,y) \\ \mu(\phi)=1}}    \# \widehat{\mathcal{M}}(\phi) U_{1}^{n_{O_{1}}(\phi)}\cdot \cdot \cdot U_{k}^{n_{O_{k}}(\phi)}V_{1}^{n_{X_{1}}(\phi)} \cdot \cdot \cdot V_{k}^{n_{X_{k}}(\phi)} y \]

\noindent
Multiplication by any of the $V_{i}$ preserves Maslov grading but raises Alexander grading by 1. Unfortunately $\partial_{U,V}^{2} \ne 0$, but it can be computed by counting the Maslov index 2 degenerations.

\begin{lem}[\hspace{1sp}\cite{Dowlin}] The map $\partial_{U,V}: \cfk_{U,V}(\cH) \to \cfk_{U,V}(\cH)$ satisfies
\begin{equation} \label{d2}
\partial_{U,V}^{2} = \sum_{i=1}^{k} (U_{a(i)}-U_{b(i)})V_{i}     
\end{equation}
\end{lem}

The same theory can be defined for punctured Heegaard diagrams, again with the restriction that we don't count discs which pass through the puncture $p$. Since the region containing the puncture $p$ is blocked, $\partial_{U,V}^{2}$ is the same as in the unpunctured case. The relative gradings are given by 
\[ M(x) - M(y) = \mu(\phi) - 2n_{\bf{O}}(\phi) - 2n_{p}(\phi) \hspace{3mm}  \text{ and }\hspace{3mm}A(x) - A(y) = n_{\bf{X}}(\phi) - n_{\bf{O}}(\phi) \]

In order to make $\cfk_{U,V}(\cH)$ into an (uncurved) chain complex, some identifications need to be made in the ground ring. 

\begin{defn}

Define the chain complex $\cfk_{2}(\cH)$ to be the quotient 
\begin{equation}\label{substitute}
\cfk_{2}(\cH) = \cfk_{U,V}(\cH)/\{ V_{i} = U_{a(i)} + U_{b(i)} \} 
\end{equation}

\noindent
The induced differential on this complex will be denoted $\partial_{2}$.

\end{defn}

Plugging into (\ref{d2}) shows that $\partial_{2}^{2}=0$. Define $\hfk_{2}(L)$ to be the homology of this complex.

\begin{rem}

More generally, one can substitute $V_{i}=(U_{a(i)}^{n} - U_{b(i)}^{n})/(U_{a(i)} - U_{b(i)})$ to get a complex $\cfk_{n}(L)$ whose homology appears to be related to the $\sln$ homology of Khovanov and Rozansky.

\end{rem}

\begin{defn}

Let $\cfk_{2}^{-}(L)$ (resp. $\cfk_{2}(L)$) be the chain complex $\cfk(\cH)$ where $\cH$ is an unpunctured (resp. punctured) Heegaard diagram for $L$. Let $\hfk^{-}_{2}(L)$ and $\hfk_{2}(L)$ be the associated homologies.

\end{defn}

As with regular knot Floer homology, multiplication by $U_{i}$ is homotopic to multiplication by $U_{j}$ if $O_{i}$ and $O_{j}$ lie on the same component of $L$, so there is a well-defined reduced complex obtained from the minus complex: 

\[ \widehat{\cfk}_{2}(L) = \cfk^{-}_{2}(L) \otimes \left( \bigotimes_{j=1}^{l} R \xrightarrow{U_{i_{j}}} R \right) \]

\noindent
The reduced homology is denoted $\widehat{\hfk}_{2}(L)$.

The differential $\partial_{2}$ is not homogeneous with respect to the Maslov and Alexander gradings, but if we define the $\delta$-grading $\gr_{\delta} = 2M - 2A$, then $\partial_{2}$ is homogeneous of degree $-2$ with respect to $\gr_{\delta}$ and each $U_{i}$ has grading $-2$. With respect to the $\delta$-grading, the homology theories $\hfk_{2}(L)$, $\hfk_{2}^{-}(L)$, and $\widehat{\hfk}_{2}(L)$ are graded link invariants \cite{Dowlin}. The unreduced and reduced invariants are finite-dimensional over $\Q$, while the minus invariant is not.

\subsection{Extension of $\hfk_{2}$ to singular links}

Let $\cH$ be a partially singular Heegaard diagram for a singular link $S$. The complex $\cfk^{-}_{2}(S)$ is obtained by adding differentials to both $\cfk^{-}(\cH)$ and $\mathcal{L}(S)$.

Define $\cfk^{-}_{2}(\cH)$ to have the same generators as $\cfk^{-}(\cH)$, but with differential

\[  \partial(x) = \sum_{y \in \mathbb{T_{\alpha}} \cap \mathbb{T_{\beta}}} \sum_{\substack{\phi \in \pi_{2}(x,y) \\ \mu(\phi)=1}}    \# \widehat{\mathcal{M}}(\phi) (-2)^{n_{XX}(\phi)} U_{1}^{n_{O_{1}}(\phi)}\cdot \cdot \cdot U_{k}^{n_{O_{k}}(\phi)}V_{1}^{n_{X_{1}}(\phi)} \cdot \cdot \cdot V_{k}^{n_{X_{k}}(\phi)} y \]

\noindent
with the same identification $V_{i}=U_{a(i)}+U_{b(i)}$. Note that we have set $XX=-2$, a unit in the ground ring, so the singularization bears a strong resemblance to the unoriented smoothing from the Heegaard Floer perspective.

The Koszul complex $\mathcal{L}(S)$ is replaced with a Koszul matrix factorization $\mathcal{L}^{+}(S)$:

\[\mathcal{L}^{+}(S) = \bigotimes_{v}   \xymatrix{R\ar@<1ex>[r]^{L(v)}&R\ar@<1ex>[l]^{L^{+}(v)}} \] 

\noindent
where $L^{+}(v)=U_a^{(v)}+U_b^{(v)}+U_c^{(v)}+U_d^{(v)}$. When $D_{I}$ is a complete resolution of a partially singular diagram $D$, it will also be useful to have the matrix factorization
\[\mathcal{L}_{I}^{+}(D_{I}) = \bigotimes_{v_{4}(I)}   \xymatrix{R\ar@<1ex>[r]^{L(v)}&R\ar@<1ex>[l]^{L^{+}(v)}} \] 

\noindent
so that 
\[ \mathcal{L}^{+}(D_{I}) = \mathcal{L}_{I}^{+}(D_{I}) \otimes \mathcal{L}_{D}^{+}(D_{I}) \]

\begin{defn}

The complex $\cfk^{-}_{2}(S)$ is defined to be $\cfk^{-}_{2}(\cH) \otimes \mathcal{L}^{+}(S)$. The differential on this complex is denoted $\partial_{2}$.

\end{defn}

\noindent
The fact that $\partial_{2}^{2}=0$ follows from the Maslov index 2 degenerations canceling the component of $\partial_{2}^{2}$ coming from $\mathcal{L}^{+}$ (see \cite{Dowlin}, Lemma 3.6).

There is also an unreduced complex $\cfk_{2}(S)$ obtained by replacing the unpunctured Heegaard diagram for $S$ with a punctured Heegaard diagram. The reduced complex is defined as follows.

Let $D$ be a partially singular diagram, and suppose $\sm(D)$ is an $l$-component link. Let $\cH$ be a Heegaard diagram for $D$ with $O_{i_{j}}$ chosen such that $O_{i_{1}},...,O_{i_{l}}$ are each on a different component of $\sm(D)$.

\begin{defn}

Define the reduced complex $\widehat{\cfk}_{2}(D)$ by 
\[\widehat{\cfk}_{2}(D) = \cfk^{-}_{2}(D) \otimes \left( \bigotimes_{j=1}^{l} R \xrightarrow{U_{i_{j}}} R \right)\]

\end{defn}

\noindent
The associated homology is denoted $\widehat{\hfk}_{2}(D)$.

In \cite{Me2}, we showed that the unreduced theory $\hfk_{2}(S)$ is particularly well-behaved when $S$ is a complete resolution.

\begin{lem}[\hspace{1sp}\cite{Me2}] \label{isolemma}

For a completely singular link $S$, the homology $\hfk_{2}(S)$ is isomorphic to $Kh(\sm(S))$ as graded vector spaces.

\end{lem}

\section{The homology $\hfk_{2}^{-}(S)$ as an $R$-module} We will need a stronger version of the previous lemma lemma. In particular, we want to show that this is true for the minus theories as well, and we want to show it as $R$-modules instead of as graded vector spaces.

Before proving this result, we need several lemmas.

\begin{lem}

Let $S$ be a singular braid, and suppose $e_{a},e_{b}$ are two incoming or outgoing edges at a 4-valent vertex $v$, and $e_{c}, e_{d}$ are the two outgoing edges. Then on $\hfk^{-}_{2}(S)$, $U_{a}=-U_{b}$ and $U_{c}=-U_{d}$.

\end{lem}

\begin{proof}
The complex $\cfk^{-}_{2}(S)$ includes the matrix factorization
\[  \xymatrix@C=5em{R\ar@<1ex>[r]^{U_{a}+U_{b}-U_{c}-U_{d}}&R\ar@<1ex>[l]^{U_{a}+U_{b}+U_{c}+U_{d}}} \]

\noindent
There are two homotopies on this factorization, $h_{1}$ and $h_{2}$, where $h_{1}$ is given by 
\[  \xymatrix@C=5em{R\ar@<1ex>[r]^{1}&R\ar@<1ex>[l]^{1}} \]

\noindent
and $h_{2}$ is given by 
\[  \xymatrix@C=5em{R\ar@<1ex>[r]^{1}&R\ar@<1ex>[l]^{-1}} \]

\noindent
Then $dh_{1}+h_{1}d = 2(U_{a}+U_{b})$ and $dh_{2}+h_{2}d = 2(U_{c}+U_{d})$. Thus, $U_{a}$ is homotopic to $-U_{b}$ and $U_{c}$ is homotopic to $-U_{d}$.

\end{proof}

\begin{defn}

Given a decorated, partially singular braid $D$, define the \emph{standard planar Heegaard diagram} for $D$ to be the diagram which locally has Figure \ref{HDCrossing} at the crossings, Figure \ref{singdiagram} at the singularizations, Figure \ref{bivalentHD} at the bivalent vertices, and Figure \ref{MarkedEdgeHD} at the decorated edge.

\end{defn}

Consider the change of basis on the complex $\cfk^{-}_{2}(S)$ given by
\[ x \mapsto (-1/2)^{A(x)} x \]

\noindent
This change has the following effects: $\mathcal{L}^{+}(v)$ is now given by
\[ \xymatrix@C=2cm{R\ar@<1ex>[r]^{L(v)}&R\ar@<1ex>[l]^{(-1/2)L^{+}(v)}} \]

\noindent
the basepoint $X_{i}$ now counts with coefficient $(-1/2)(U_{a(i)}+U_{b(i)})$, and the double-point $XX$ counts with coefficient $1$. The last is the significant one, as we can forget about the $XX$ basepoints when calculating $\cfk^{-}_{2}(\cH)$. In particular, it allows the Heegaard diagram at each 4-valent vertex to be isotoped as in Figure \ref{IsotopedHD}. 

\begin{figure}
 \centering
 \small
\def\svgwidth{8cm}
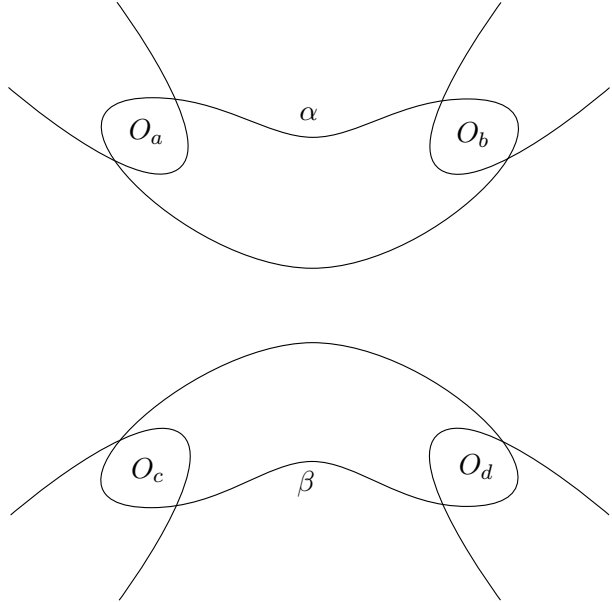
  \caption{A bivalent vertex}  \label{IsotopedHD}
\end{figure}%

\begin{lem} \label{MOYtwo}

Consider the diagrams $S_{II}$ and $S'_{II}$ in Figure \ref{M2}. Then 
\[\hfk_{2}^{-}(S_{II}) \cong \hfk^{-}_{2}(S'_{II})[U_{c}, U_{d}]/(U_{c}=-U_{d}, U_{c}^{2} = U_{1}^{2}) \]

\noindent
as $R$-modules.
\end{lem}

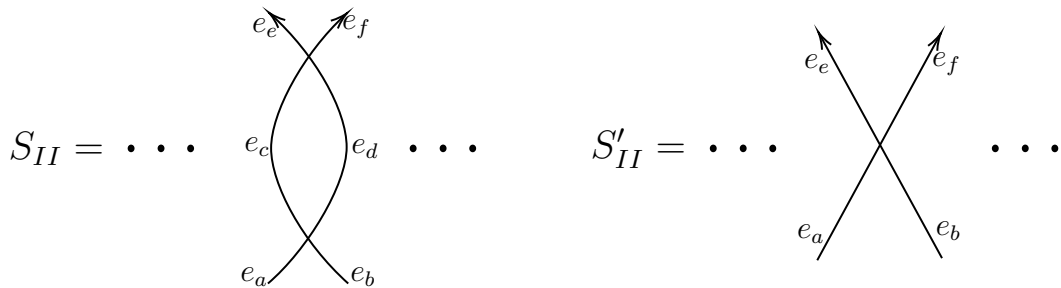
\begin{figure}[ht]
\centering
\def\svgwidth{6.5cm}
\begingroup%
  \makeatletter%
  \providecommand\color[2][]{%
    \errmessage{(Inkscape) Color is used for the text in Inkscape, but the package 'color.sty' is not loaded}%
    \renewcommand\color[2][]{}%
  }%
  \providecommand\transparent[1]{%
    \errmessage{(Inkscape) Transparency is used (non-zero) for the text in Inkscape, but the package 'transparent.sty' is not loaded}%
    \renewcommand\transparent[1]{}%
  }%
  \providecommand\rotatebox[2]{#2}%
  \ifx\svgwidth\undefined%
    \setlength{\unitlength}{1390.83860382bp}%
    \ifx\svgscale\undefined%
      \relax%
    \else%
      \setlength{\unitlength}{\unitlength * \real{\svgscale}}%
    \fi%
  \else%
    \setlength{\unitlength}{\svgwidth}%
  \fi%
  \global\let\svgwidth\undefined%
  \global\let\svgscale\undefined%

\tikzset{every picture/.style={line width=0.75pt}} 

\begin{tikzpicture}[x=0.75pt,y=0.75pt,yscale=-.8,xscale=.8]

\draw    (243.37,229.35) .. controls (268,208.35) and (293,167.2) .. (293,143.35) .. controls (293,119.86) and (272.95,84.6) .. (245.18,59.49) ;
\draw [shift={(243.9,58.35)}, rotate = 401.52] [color={rgb, 255:red, 0; green, 0; blue, 0 }  ][line width=0.75]    (10.93,-3.29) .. controls (6.95,-1.4) and (3.31,-0.3) .. (0,0) .. controls (3.31,0.3) and (6.95,1.4) .. (10.93,3.29)   ;

\draw    (293.99,229.35) .. controls (265,203.35) and (245.37,166.6) .. (245.37,144.01) .. controls (245.37,121.76) and (265.39,84.49) .. (292.12,59.79) ;
\draw [shift={(293.35,58.67)}, rotate = 497.93] [color={rgb, 255:red, 0; green, 0; blue, 0 }  ][line width=0.75]    (10.93,-3.29) .. controls (6.95,-1.4) and (3.31,-0.3) .. (0,0) .. controls (3.31,0.3) and (6.95,1.4) .. (10.93,3.29)   ;

\draw  [fill={rgb, 255:red, 3; green, 3; blue, 3 }  ,fill opacity=1 ] (155.56,144.12) .. controls (155.56,143.14) and (156.25,142.35) .. (157.1,142.35) .. controls (157.96,142.35) and (158.65,143.14) .. (158.65,144.12) .. controls (158.65,145.09) and (157.96,145.88) .. (157.1,145.88) .. controls (156.25,145.88) and (155.56,145.09) .. (155.56,144.12) -- cycle ;
\draw  [fill={rgb, 255:red, 3; green, 3; blue, 3 }  ,fill opacity=1 ] (174.44,144.09) .. controls (174.44,143.12) and (175.14,142.32) .. (175.99,142.32) .. controls (176.84,142.32) and (177.53,143.12) .. (177.53,144.09) .. controls (177.53,145.07) and (176.84,145.86) .. (175.99,145.86) .. controls (175.14,145.86) and (174.44,145.07) .. (174.44,144.09) -- cycle ;
\draw  [fill={rgb, 255:red, 3; green, 3; blue, 3 }  ,fill opacity=1 ] (193.09,144.13) .. controls (193.09,143.15) and (193.79,142.36) .. (194.64,142.36) .. controls (195.49,142.36) and (196.18,143.15) .. (196.18,144.13) .. controls (196.18,145.11) and (195.49,145.9) .. (194.64,145.9) .. controls (193.79,145.9) and (193.09,145.11) .. (193.09,144.13) -- cycle ;
\draw  [fill={rgb, 255:red, 3; green, 3; blue, 3 }  ,fill opacity=1 ] (333.37,144.12) .. controls (333.37,143.14) and (334.07,142.35) .. (334.92,142.35) .. controls (335.77,142.35) and (336.46,143.14) .. (336.46,144.12) .. controls (336.46,145.09) and (335.77,145.88) .. (334.92,145.88) .. controls (334.07,145.88) and (333.37,145.09) .. (333.37,144.12) -- cycle ;
\draw  [fill={rgb, 255:red, 3; green, 3; blue, 3 }  ,fill opacity=1 ] (352.26,144.09) .. controls (352.26,143.12) and (352.95,142.32) .. (353.8,142.32) .. controls (354.66,142.32) and (355.35,143.12) .. (355.35,144.09) .. controls (355.35,145.07) and (354.66,145.86) .. (353.8,145.86) .. controls (352.95,145.86) and (352.26,145.07) .. (352.26,144.09) -- cycle ;
\draw  [fill={rgb, 255:red, 3; green, 3; blue, 3 }  ,fill opacity=1 ] (370.91,144.13) .. controls (370.91,143.15) and (371.6,142.36) .. (372.45,142.36) .. controls (373.31,142.36) and (374,143.15) .. (374,144.13) .. controls (374,145.11) and (373.31,145.9) .. (372.45,145.9) .. controls (371.6,145.9) and (370.91,145.11) .. (370.91,144.13) -- cycle ;
\draw  [fill={rgb, 255:red, 3; green, 3; blue, 3 }  ,fill opacity=1 ] (522.56,144.12) .. controls (522.56,143.14) and (523.25,142.35) .. (524.1,142.35) .. controls (524.96,142.35) and (525.65,143.14) .. (525.65,144.12) .. controls (525.65,145.09) and (524.96,145.88) .. (524.1,145.88) .. controls (523.25,145.88) and (522.56,145.09) .. (522.56,144.12) -- cycle ;
\draw  [fill={rgb, 255:red, 3; green, 3; blue, 3 }  ,fill opacity=1 ] (541.44,144.09) .. controls (541.44,143.12) and (542.14,142.32) .. (542.99,142.32) .. controls (543.84,142.32) and (544.53,143.12) .. (544.53,144.09) .. controls (544.53,145.07) and (543.84,145.86) .. (542.99,145.86) .. controls (542.14,145.86) and (541.44,145.07) .. (541.44,144.09) -- cycle ;
\draw  [fill={rgb, 255:red, 3; green, 3; blue, 3 }  ,fill opacity=1 ] (560.09,144.13) .. controls (560.09,143.15) and (560.79,142.36) .. (561.64,142.36) .. controls (562.49,142.36) and (563.18,143.15) .. (563.18,144.13) .. controls (563.18,145.11) and (562.49,145.9) .. (561.64,145.9) .. controls (560.79,145.9) and (560.09,145.11) .. (560.09,144.13) -- cycle ;
\draw  [fill={rgb, 255:red, 3; green, 3; blue, 3 }  ,fill opacity=1 ] (700.37,144.12) .. controls (700.37,143.14) and (701.07,142.35) .. (701.92,142.35) .. controls (702.77,142.35) and (703.46,143.14) .. (703.46,144.12) .. controls (703.46,145.09) and (702.77,145.88) .. (701.92,145.88) .. controls (701.07,145.88) and (700.37,145.09) .. (700.37,144.12) -- cycle ;
\draw  [fill={rgb, 255:red, 3; green, 3; blue, 3 }  ,fill opacity=1 ] (719.26,144.09) .. controls (719.26,143.12) and (719.95,142.32) .. (720.8,142.32) .. controls (721.66,142.32) and (722.35,143.12) .. (722.35,144.09) .. controls (722.35,145.07) and (721.66,145.86) .. (720.8,145.86) .. controls (719.95,145.86) and (719.26,145.07) .. (719.26,144.09) -- cycle ;
\draw  [fill={rgb, 255:red, 3; green, 3; blue, 3 }  ,fill opacity=1 ] (737.91,144.13) .. controls (737.91,143.15) and (738.6,142.36) .. (739.45,142.36) .. controls (740.31,142.36) and (741,143.15) .. (741,144.13) .. controls (741,145.11) and (740.31,145.9) .. (739.45,145.9) .. controls (738.6,145.9) and (737.91,145.11) .. (737.91,144.13) -- cycle ;
\draw    (668.33,213.35) -- (591.29,72.44) ;
\draw [shift={(590.33,70.69)}, rotate = 421.33000000000004] [color={rgb, 255:red, 0; green, 0; blue, 0 }  ][line width=0.75]    (10.93,-3.29) .. controls (6.95,-1.4) and (3.31,-0.3) .. (0,0) .. controls (3.31,0.3) and (6.95,1.4) .. (10.93,3.29)   ;

\draw    (589.67,214.69) -- (666.71,73.11) ;
\draw [shift={(667.67,71.35)}, rotate = 478.55] [color={rgb, 255:red, 0; green, 0; blue, 0 }  ][line width=0.75]    (10.93,-3.29) .. controls (6.95,-1.4) and (3.31,-0.3) .. (0,0) .. controls (3.31,0.3) and (6.95,1.4) .. (10.93,3.29)   ;

\draw (109.94,143.86) node [scale=1.25]  {$S_{II} =$};
\draw (476.94,143.86) node [scale=1.25]  {$S'_{II} =$};
\draw (241.73,67.73) node  [align=left] {$e_{e}$};
\draw (299.27,67.57) node  [align=left] {$e_{f}$};
\draw (234.73,225) node  [align=left] {$e_{a}$};
\draw (303.34,225) node  [align=left] {$e_{b}$};
\draw (236.59,143.17) node  [align=left] {$e_{c}$};
\draw (304.59,144.17) node  [align=left] {$e_{d}$};
\draw (589.73,90.73) node  [align=left] {$e_{e}$};
\draw (671.27,90.57) node  [align=left] {$e_{f}$};
\draw (585.73,198) node  [align=left] {$e_{a}$};
\draw (672.34,196) node  [align=left] {$e_{b}$};

\end{tikzpicture}

\endgroup%
\caption{The diagrams for MOY II}\label{M2}
\end{figure}

\begin{proof}

Consider the standard planar diagram for $S_{II}$. With the change of basis to make $XX$ count with coefficient $1$, each $\alpha$ and $\beta$ curve at a 4-valent vertex can be isotoped as in Figure \ref{IsotopedHD}. The resulting Heegaard diagram $\cH$ for $S_{II}$ is shown in Figure \ref{SIIHD}. After additional isotopy, we get the diagram $\cH'$ in Figure \ref{SII2HD}. There are two holomorphic curves from $x$ to $y$ which count with coefficient $1$ and $-1$ respectively, so they cancel. There is one bigon in the diagram from $y$ to $x$ which comes with coefficient $\pm U_{c}U_{d}$. However, there is also a `large' bigon from $y$ to $x$ which wraps around the back of $S^{2}$. This bigon with either the interiors of the $\alpha$ circles removed or the interiors of the $\beta$ circles removed will contribute, depending on the choice of almost complex structure. Thus, the coefficient will come from the basepoints passed through at the marked edge where the $\alpha$ and $\beta$ circle are missing (see Figure \ref{MarkedEdgeHD}). Without loss of generality, suppose the disc that contributes is the complement of the $\alpha$ circles. Then it comes with coefficient $\pm(1/2)U_{1}(U_{1}+U_{2})$.

\begin{figure}[ht]
\begin{subfigure}{.28\textwidth}
 \centering
 \scriptsize
\def\svgwidth{3.6cm}
\input{SII.pdf_tex}
  \caption{The Heegaard diagram $\cH$ for $S_{II}$}  \label{SIIHD}
 \end{subfigure} \hspace{2mm}
\begin{subfigure}{.28\textwidth}
 \centering
 \tiny
\def\svgwidth{3.6cm}
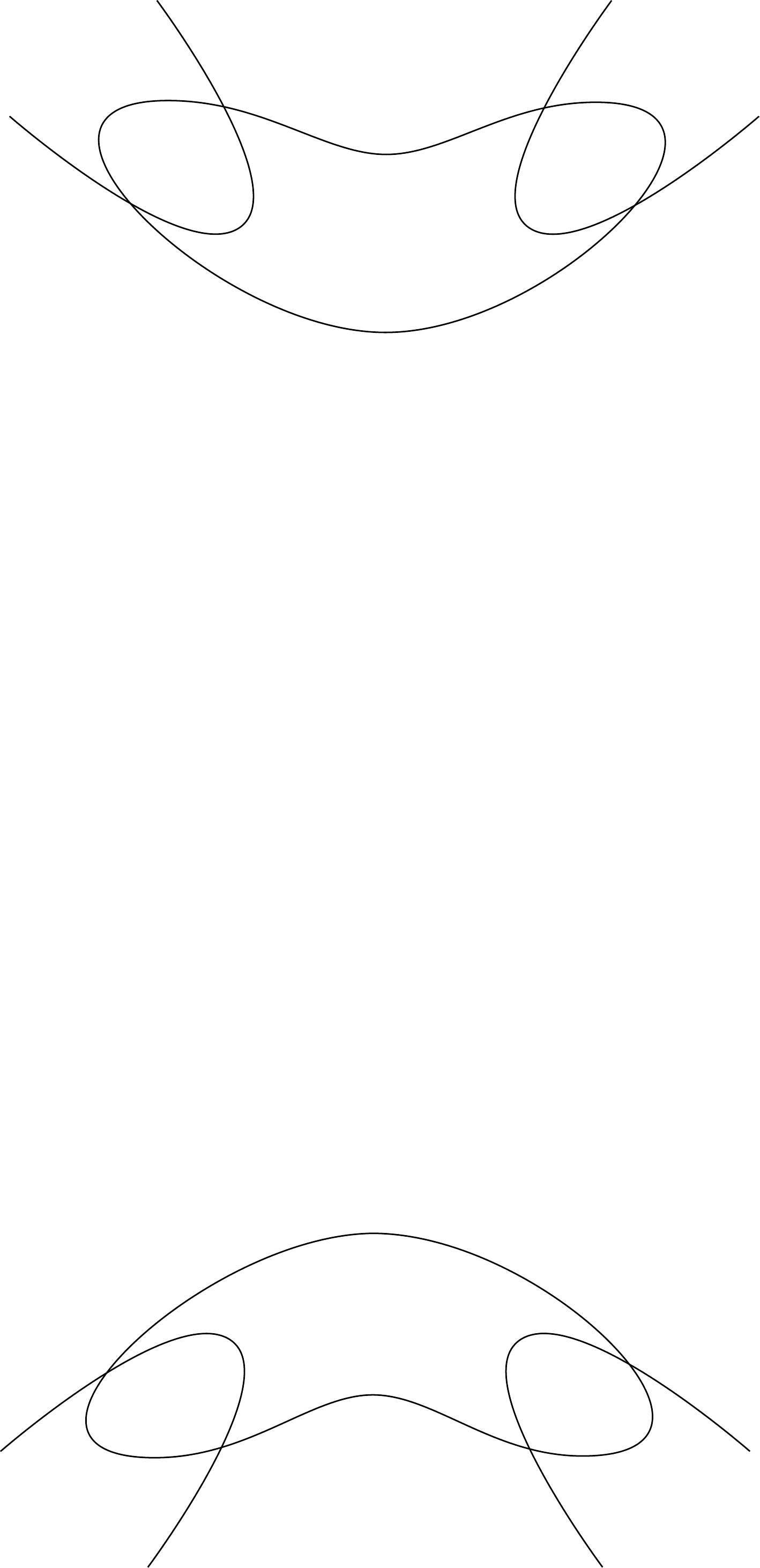
  \caption{The Heegaard diagram $\cH'$ for $S_{II}$ after isotopy}  \label{SII2HD}
  \end{subfigure}\hspace{2mm}
  \begin{subfigure}{.28\textwidth}
 \centering
 \tiny
\def\svgwidth{3.6cm}
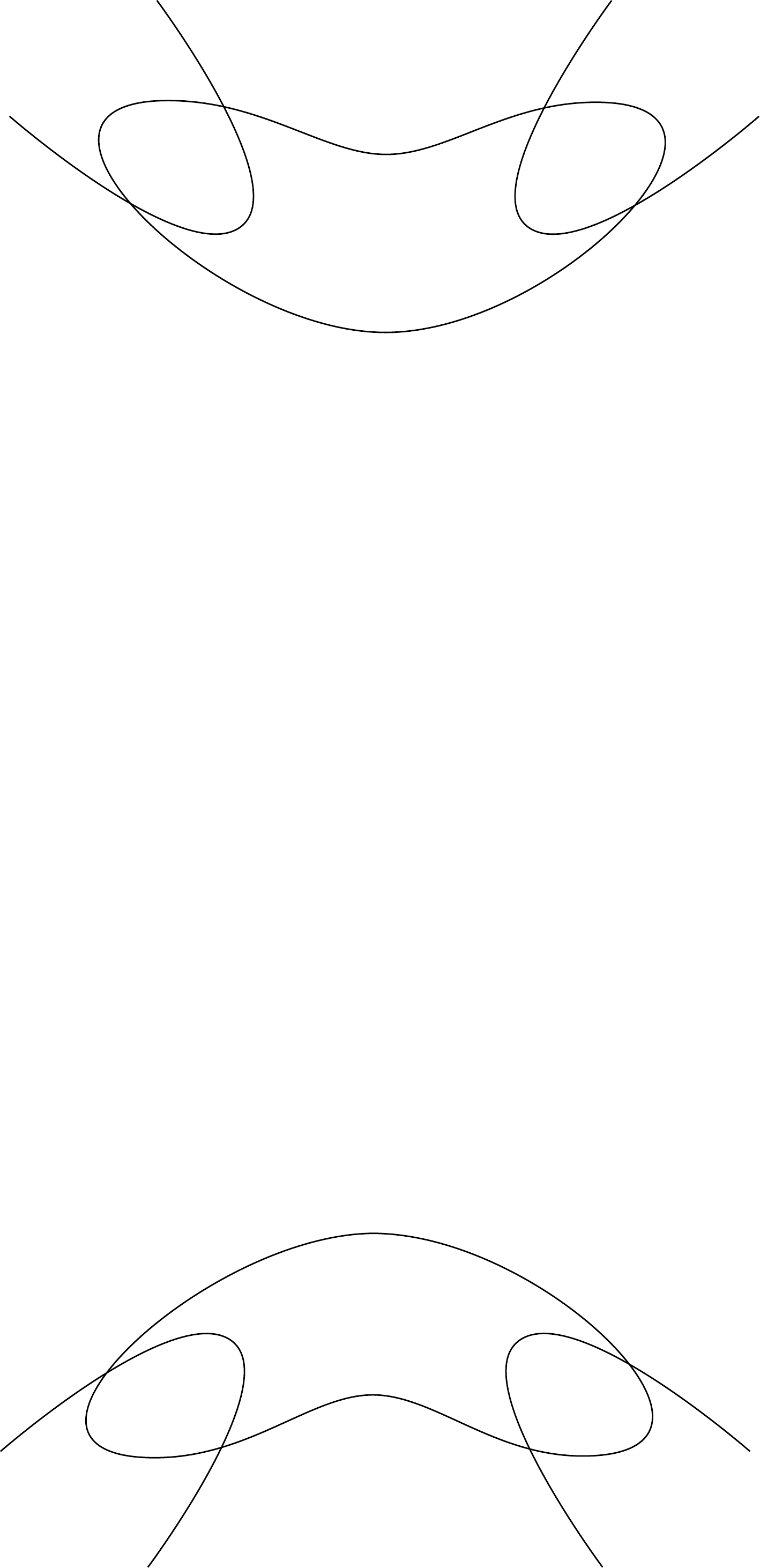
  \caption{The Heegaard diagram $\cH''$ for $S'_{II}$}  \label{SII3HD}
  \end{subfigure}
  \caption{}
\end{figure}%

Thus, the complex $\cfk_{2}^{-}(\cH')$ is given by 
\[ \cfk_{2}^{-}(\cH') \cong \cfk_{2}^{-}(\cH'')[U_{c},U_{d}] \otimes (R \xrightarrow{U_{c}U_{d}+\frac{1}{2}U_{1}(U_{1}+U_{2})} R) \]

\noindent
where $\cH''$ is the Heegaard diagram for $S'_{II}$ in Figure \ref{SII3HD}. The two 4-valent vertices in $S_{II}$ contribute the Koszul complex 
\[ \mathsf{K}_{II} = 
(\xymatrix@C=2.5cm{R\ar@<1ex>[r]^{U_{e}+U_{f} -U_{c}-U_{d}}&R\ar@<1ex>[l]^{-\frac{1}{2}(U_{e}+U_{f} +U_{c}+U_{d})}} ) \otimes (\xymatrix@C=2.5cm{R\ar@<1ex>[r]^{U_{c}+U_{d} -U_{a}-U_{b}}&R\ar@<1ex>[l]^{-\frac{1}{2}(U_{c}+U_{d} +U_{a}+U_{b})}}) 
\]

\noindent
while the single 4-valent vertex in $S'_{II}$ contributes 
\[ \mathsf{K}'_{II} = \xymatrix@C=2.5cm{R\ar@<1ex>[r]^{U_{e}+U_{f}-U_{a}-U_{b}}&R\ar@<1ex>[l]^{-\frac{1}{2}(U_{e}+U_{f}+U_{a}+U_{b})}} \]

\noindent
Let $\mathsf{K}_{0}$ denote the Koszul complex coming from the remaining 4-valent vertices. Then 
\[ \cfk^{-}_{2}(S_{II}) \cong \cfk^{-}_{2}(\cH') \otimes \mathsf{K}_{II} \otimes \mathsf{K}_{0} \]
\[ \cfk^{-}_{2}(S'_{II}) \cong \cfk^{-}_{2}(\cH'') \otimes \mathsf{K}'_{II} \otimes \mathsf{K}_{0} \]

\noindent
Substituting for $\cfk^{-}_{2}(\cH')$, we get
\[ \cfk^{-}_{2}(S_{II}) \cong \cfk_{2}^{-}(\cH'')[U_{c},U_{d}] \otimes (R \xrightarrow{U_{c}U_{d}+\frac{1}{2}U_{1}(U_{1}+U_{2})} R)  \otimes \mathsf{K}_{II} \otimes \mathsf{K}_{0} \]

\noindent
After canceling the map by $U_{e}+U_{f} -U_{c}-U_{d}$ in the first factor of $ \mathsf{K}_{II}$, we get
\[ \cfk^{-}_{2}(S_{II}) \simeq \cfk^{-}_{2}(\cH')[U_{c}] \otimes (R \xrightarrow{U_{c}U_{d}-\frac{1}{2}U_{1}(U_{1}+U_{2})} R) \otimes  (\xymatrix@C=2.5cm{R\ar@<1ex>[r]^{U_{c}+U_{d} -U_{a}-U_{b}}&R\ar@<1ex>[l]^{-\frac{1}{2}(U_{c}+U_{d} +U_{a}+U_{b})}}) \otimes \mathsf{K}_{0} \]

\noindent
with $U_{d}=U_{e}+U_{f}-U_{c}$, where $\simeq$ denotes quasi-isomorphism. Then the factor   \[     (\xymatrix@C=2.5cm{R\ar@<1ex>[r]^{U_{c}+U_{d} -U_{a}-U_{b}}&R\ar@<1ex>[l]^{-\frac{1}{2}(U_{c}+U_{d} +U_{a}+U_{b})}})\] is precisely $\mathsf{K}'_{II}$, so we have 
\[ \cfk^{-}_{2}(S_{II}) \simeq \cfk^{-}_{2}(S'_{II})[U_{c}] \otimes (R \xrightarrow{U_{c}U_{d}+\frac{1}{2}U_{1}(U_{1}+U_{2})} R)   \]

Since on homology $U_{c}=-U_{d}$ and $U_{1}=U_{2}$, this gives 
\[ \hfk_{2}^{-}(S_{II}) \cong \hfk^{-}_{2}(S'_{II})[U_{c}]/U_{c}^{2}=U_{1}^{2} \]

\noindent
with $U_{d}$ acting by $-U_{c}$.

\end{proof}

\begin{lem} \label{MOYthree}

Consider the diagrams $S_{IIIa}$, $S'_{IIIa}$, $S_{IIIb}$, and $S'_{IIIb}$ in Figure \ref{M3}. Then 
\[\hfk_{2}^{-}(S_{IIIa}) \cong \hfk_{2}^{-}(S'_{IIIa})[U_{f},U_{g},U_{h},U_{i}]/(U_{f}=U_{i} = -U_{h}=-U_{g}=U_{c}) \]

\noindent
and
\[\hfk_{2}^{-}(S_{IIIb}) \cong \hfk_{2}^{-}(S'_{IIIb})[U_{d},U_{g},U_{h},U_{i}]/(U_{d}=U_{i} = -U_{h}=-U_{g}=U_{a}) \]

\noindent
as $R$-modules.

\end{lem}

\begin{figure}[ht]
\centering
\def\svgwidth{6.5cm}
\input{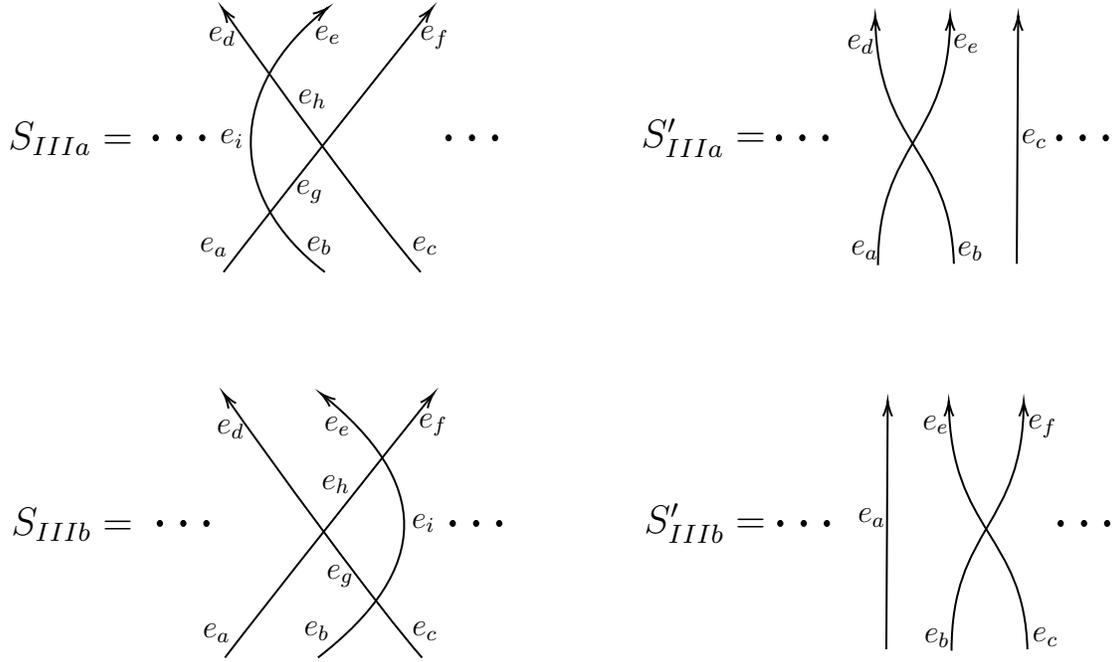}
\caption{The diagrams for MOY IIIa and IIIb}\label{M3}
\end{figure}

\begin{proof}

Consider the standard planar Heegaard diagram for $S_{IIIa}$. With the change of basis to make $XX$ count with coefficient $1$, each $\alpha$ and $\beta$ curve at a 4-valent vertex can be isotoped as in Figure \ref{IsotopedHD}. The resulting Heegaard diagram $\cH$ for $S_{IIIa}$ is shown in Figure \ref{SIIIaHD}. Applying isotopy and handleslides, we get the diagram $\cH'$ in Figure \ref{Handleslide}.

\begin{figure}[ht]
\begin{subfigure}{.49\textwidth}
 \centering
 \scriptsize
\def\svgwidth{7cm}
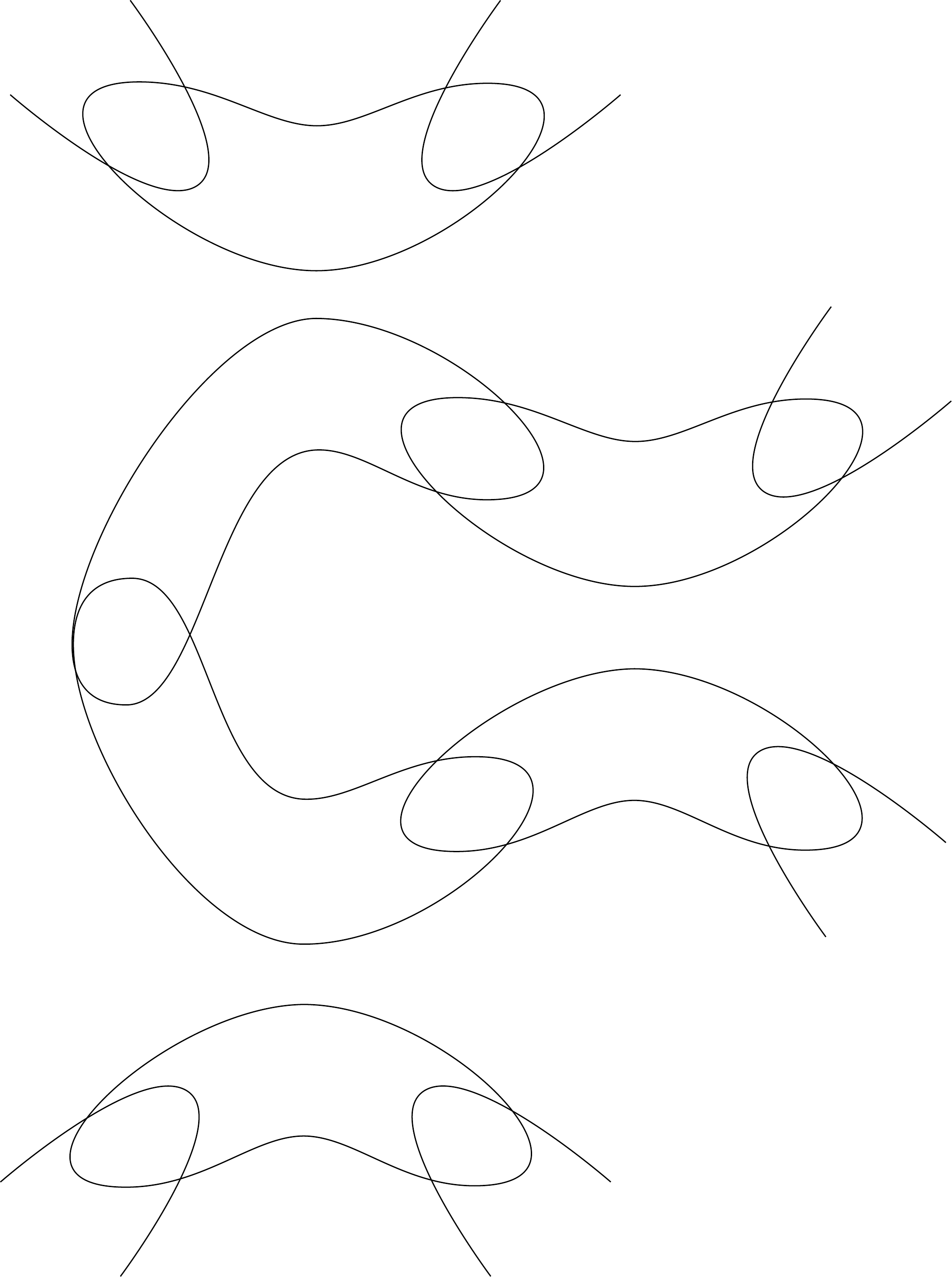
  \caption{The Heegaard diagram $\cH$ for $S_{IIIa}$}  \label{SIIIaHD}
 \end{subfigure}
\begin{subfigure}{.5\textwidth}
 \centering
 \tiny
\def\svgwidth{8cm}
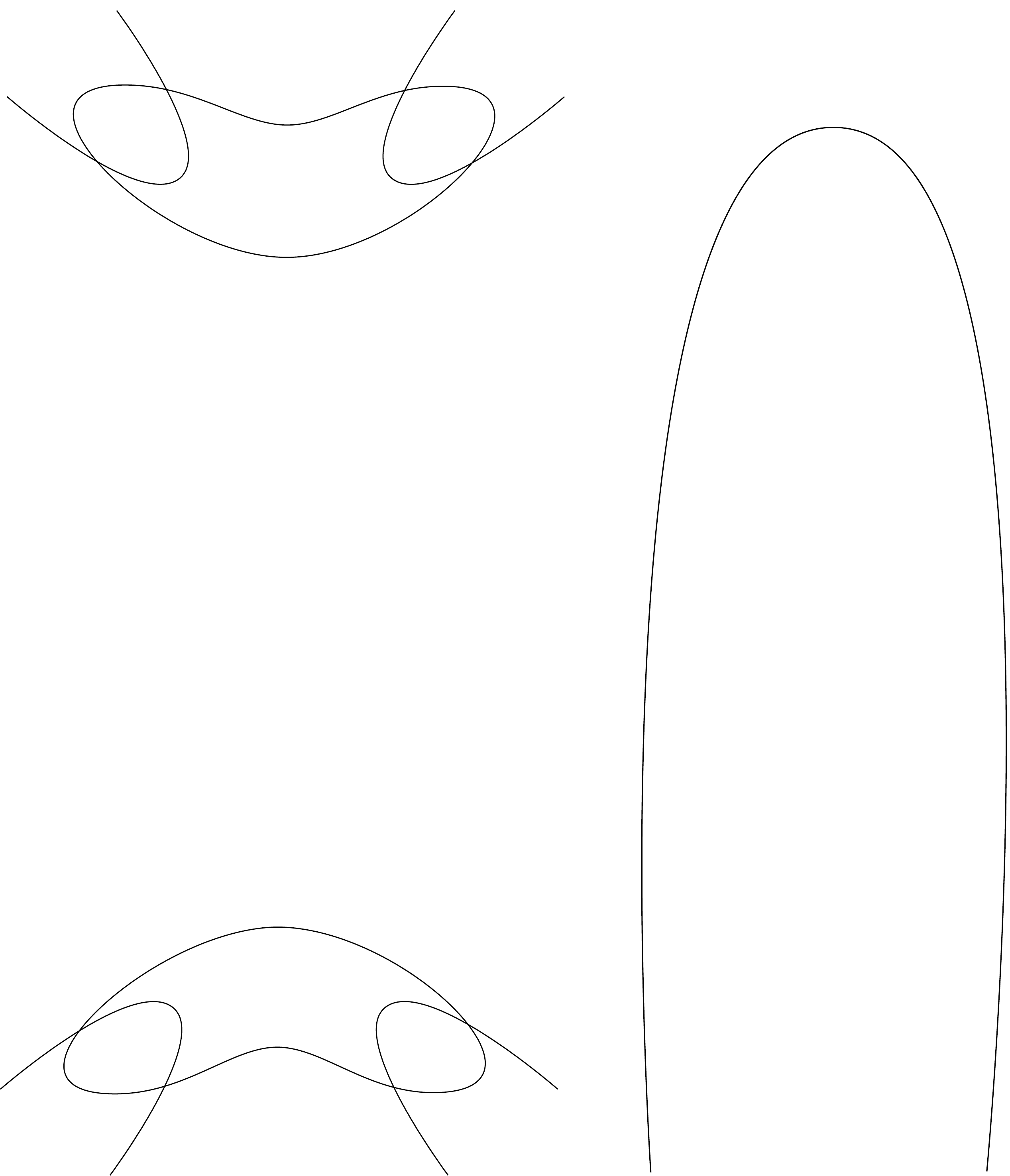
  \caption{The Heegaard diagram $\cH'$ for $S_{IIIa}$ after isotopy and handleslides}  \label{Handleslide}
  \end{subfigure}
  \caption{}
\end{figure}%

\begin{figure}
 \centering
 \footnotesize
\def\svgwidth{6cm}
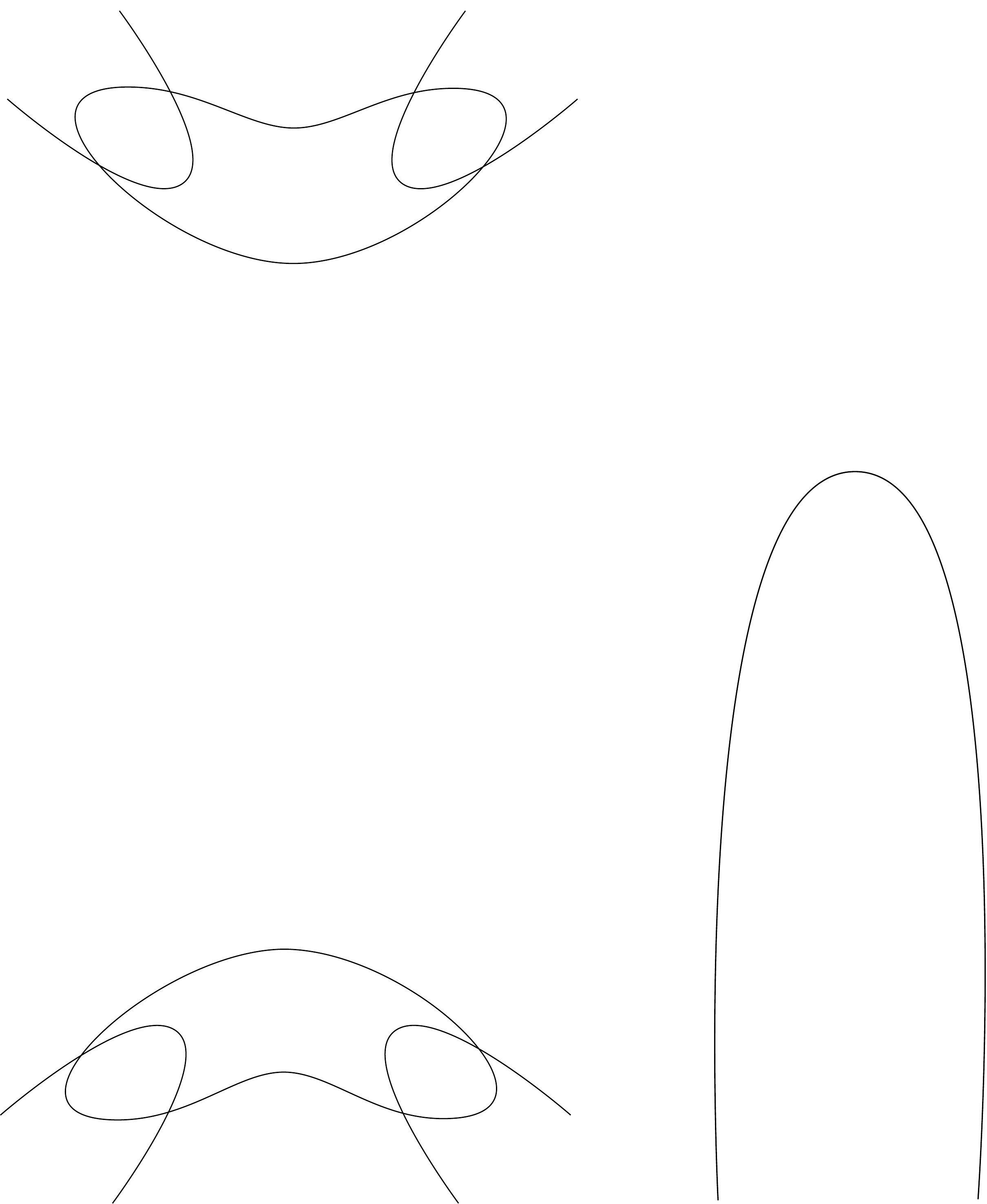
  \caption{The destabilized diagram $\cH''$}  \label{destabHD}
\end{figure}%

The Heegaard diagram $\cH'$ corresponds to two $(0,3)$ stabilizations of the diagram $\cH''$ in Figure \ref{destabHD}. In particular, 
\begin{equation}\label{eqIII}
 \cfk_{2}(\cH') \cong \cfk_{2}(\cH'')[U_{f}, U_{g}, U_{h}, U_{i}] \otimes (\xymatrix@C=1.5cm{R\ar@<1ex>[r]^{U_{i}-U_{f}}&R\ar@<1ex>[l]^{U_{g}-U_{h}}} ) \otimes (\xymatrix@C=1.5cm{R\ar@<1ex>[r]^{U_{f}-U_{c}}&R\ar@<1ex>[l]^{U_{g}-P}}) \end{equation}

\noindent
where $P$ is the contribution of the basepoint connected to $O_{f}$ not pictured in Figure \ref{SIIIaHD}. See \cite{Zemke} for a detailed treatment of quasi-stabilization in the master complex.

The three 4-valent vertices in $S_{IIIa}$ contribute the Koszul complex 
\[ \mathsf{K}_{IIIa} = 
(\xymatrix@C=2.5cm{R\ar@<1ex>[r]^{U_{d}+U_{e} -U_{i}-U_{h}}&R\ar@<1ex>[l]^{-\frac{1}{2}(U_{d}+U_{e} +U_{i}+U_{h})}} ) \otimes (\xymatrix@C=2.5cm{R\ar@<1ex>[r]^{U_{i}+U_{g} -U_{a}-U_{b}}&R\ar@<1ex>[l]^{-\frac{1}{2}(U_{i}+U_{g} +U_{a}+U_{b})}})  \otimes (\xymatrix@C=2.5cm{R\ar@<1ex>[r]^{U_{h}+U_{f}-U_{g}-U_{c}}&R\ar@<1ex>[l]^{-\frac{1}{2}(U_{h}+U_{f}+U_{g}+U_{c})}}) 
\]

\noindent
while the single 4-valent vertex in $S'_{IIIa}$ contributes 
\[ \mathsf{K'}_{IIIa} = \xymatrix@C=2.5cm{R\ar@<1ex>[r]^{U_{d}+U_{e}-U_{a}-U_{b}}&R\ar@<1ex>[l]^{-\frac{1}{2}(U_{d}+U_{e}+U_{a}+U_{b})}} \]

\noindent
Let $\mathsf{K}_{0}$ denote the Koszul complex coming from the remaining 4-valent vertices. Then 
\[ \cfk^{-}_{2}(S_{IIIa}) \cong \cfk^{-}_{2}(\cH') \otimes \mathsf{K}_{IIIa} \otimes \mathsf{K}_{0} \]
\[ \cfk^{-}_{2}(S'_{IIIa}) \cong \cfk^{-}_{2}(\cH'') \otimes \mathsf{K'}_{IIIa} \otimes \mathsf{K}_{0} \]

Applying equation \ref{eqIII}, we get
\[\cfk^{-}_{2}(S_{IIIa}) \cong  \cfk^{-}_{2}(\cH'')[U_{f},U_{g}, U_{h}, U_{i}] \otimes (\xymatrix@C=1.5cm{R\ar@<1ex>[r]^{U_{i}-U_{f}}&R\ar@<1ex>[l]^{U_{g}-U_{h}}} ) \otimes (\xymatrix@C=1.5cm{R\ar@<1ex>[r]^{U_{f}-U_{c}}&R\ar@<1ex>[l]^{U_{g}-P}})  \otimes \mathsf{K}_{IIIa} \otimes \mathsf{K}_{0} \]

\noindent
Contracting the maps by $U_{i}-U_{f}$ and $U_{f}-U_{c}$, we get
\[\cfk^{-}_{2}(S_{IIIa}) \simeq  \cfk^{-}_{2}(\cH'')[U_{g}, U_{h}] \otimes \mathsf{K}_{IIIa} \otimes \mathsf{K}_{0} \]
\noindent
with $U_{f}=U_{i}=U_{c}$, where $\simeq$ denotes quasi-isomorphism. Similarly, we contract the maps by $U_{d}+U_{e}-U_{i}-U_{h}$ and $U_{h}+U_{f}-U_{g}-U_{c}$  in $\mathsf{K}_{IIIa}$ to substitute for $U_{g}$ and $U_{h}$, giving
\[\cfk^{-}_{2}(S_{IIIa}) \simeq  \cfk^{-}_{2}(\cH'')\otimes  (\xymatrix@C=2.5cm{R\ar@<1ex>[r]^{U_{i}+U_{g} -U_{a}-U_{b}}&R\ar@<1ex>[l]^{-\frac{1}{2}(U_{i}+U_{g} +U_{a}+U_{b})}}) \otimes \mathsf{K}_{0} \]

\noindent
with $U_{f}=U_{i}=U_{c}, U_{g}=U_{h}$, and $U_{i}+U_{h}=U_{d}+U_{e}$. These relations make the middle term equal to $\mathit{K}'_{IIIa}$, so 
\[\cfk^{-}_{2}(S_{IIIa}) \simeq  \cfk^{-}_{2}(\cH'')\otimes \mathit{K}'_{IIIa} \otimes \mathsf{K}_{0} \cong \cfk^{-}_{2}(S'_{IIIa}) \]

This proves the lemma for $S_{IIIa}$ and $S'_{IIIa}$ - the argument for $S_{IIIb}$ and $S'_{IIIb}$ is similar.
\end{proof}

The isomorphism in Lemma \ref{MOYtwo} is called a MOY II move, while the isomorphisms in Lemma \ref{MOYthree} are called MOY III moves. We can now put these lemmas together to get our main theorem of the section:

\begin{thm} \label{hfk2kh}

For a completely singular link $S$, the homology $\hfk_{2}^{-}(S)$ is isomorphic to $Kh^{-}(\sm(S))$ as an $R$-module.

\end{thm}

\begin{proof}

We will prove this by induction on the number of components in $\sm(S)$. Recall that $S$ is a singular braid with the decorated edge on the left-most strand. Viewing this diagram as being in $\R^{2}$ instead of $S^{2}$, each component has a well-defined interior. A component is called \emph{innermost} if its interior does not contain any other components.

Let $\mathcal{C}$ be a component of $\sm(S)$ which is not the marked component and is innermost. Note that as long as there is more than one component this is possible, as the marked component will never be contained in another component. Consider the triangular region bounded by $e_{g},e_{h},e_{i}$ in the two MOY III moves in Figure \ref{M3}. We say that a MOY III move \emph{simplifies} $\mathcal{C}$ if this triangular region is in the interior of $\mathcal{C}$ in $\sm(S)$.

After applying MOY III moves that simplify $\mathcal{C}$ until there are none left, we get a diagram $S'$ with $\hfk^{-}(S) \cong \hfk^{-}(S')$. There are two possibilities for how the component $\mathcal{C}$ appears in $S'$: as the two edges $e_{c}$ and $e_{d}$ in the MOY II move (see Figure \ref{M2}), or it is the right-most strand in the braid making up an unknotted component. In the first case, we can apply a MOY II move to get a diagram $S''$ with 
\[ \hfk^{-}_{2}(S) \cong \hfk_{2}^{-}(S'')[U_{i}]/U_{i}^{2}=U_{1}^{2} \]

\noindent
and for each $e_{j}$ which lies on $\mathcal{C}$, $U_{j}$ acts on $\hfk^{-}_{2}(S)$ by $(-1)^{\str(i)-\str(j)}U_{i}$. In the second case, let $S''$ be the diagram obtained by deleting the right-most strand. Then again we have
\[ \hfk^{-}_{2}(S) \cong \hfk_{2}^{-}(S'')[U_{i}]/U_{i}^{2}=U_{1}^{2} \]

\noindent
To see this, apply isotopies and destabilizations to the unknotted component until it is given by the Heegaard diagram in Figure \ref{unknotHD}. Then we get
\[ \cfk^{-}_{2}(S) \cong \cfk_{2}^{-}(S'')[U_{i}] \otimes (R \xrightarrow{U_{i}^{2} - \frac{1}{2}U_{1}(U_{1}+U_{2})} R) \]

\noindent
by the same holomorphic disc count as in the proof of Lemma \ref{MOYtwo}. Taking homology yields the desired formula.

\begin{figure}
 \centering
\def\svgwidth{4cm}
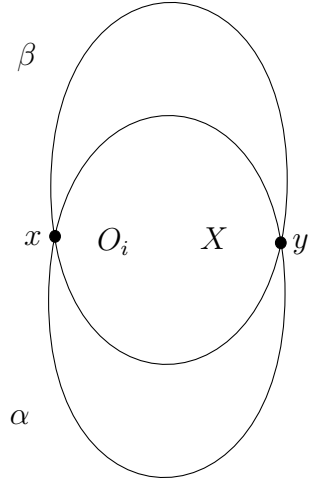
  \caption{The Heegaard diagram for the unknotted component}  \label{unknotHD}
\end{figure}%

This establishes the inductive step. For the base case, apply MOY III moves until the diagram is one of the two shown in Figure \ref{2unknots}. In both cases, the homology can be readily computed to be $\Q[U]$, with each $U_{i}$ acting by $(-1)^{\str(i)} U$.

\end{proof}

\begin{figure}[ht]
\begin{subfigure}{.49\textwidth}
 \centering
 \small
\def\svgwidth{4cm}
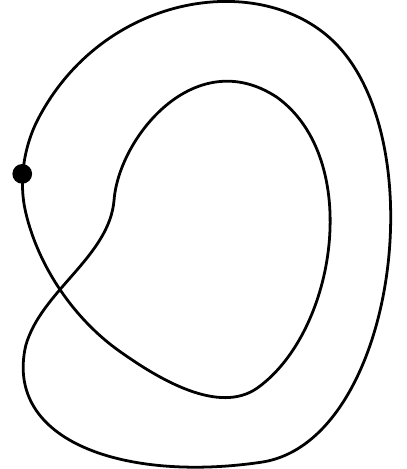
 \end{subfigure}
\begin{subfigure}{.5\textwidth}
 \centering
 \small
\def\svgwidth{4cm}
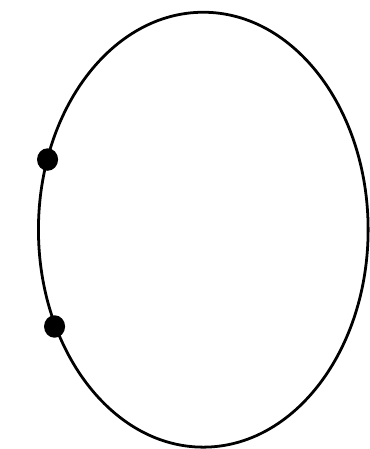
  \end{subfigure}
  \caption{The two possible diagrams for the last component} \label{2unknots}
\end{figure}%

\section{Relationship with Khovanov homology}

In this section we willl show that the $E_{2}$ page of the spectral sequence on $C^{-}_{2}(D)$ induced by the cube filtration is the minus version of Khovanov homology $Kh^{-}(\sm(D))$. The same result for the reduced theories will follow.

The first step is to show that at a vertex in the cube, this complex computes $\hfk^{-}_{2}(S)$. By Theorem \ref{hfk2kh}, $\hfk^{-}_{2}(S) \cong Kh^{-}(\sm(S))$. We will then show that the edge maps are the Khovanov edge maps, proving the theorem.

\subsection{Some algebraic lemmas} Before getting into the proof, we'll need some results in homological algebra which will be useful throughout.

\begin{lem} [\hspace{1sp}\cite{SSBook}]
\label{homalglemma}
Suppose $F$ is a filtered chain map between filtered complexes $C$ and $C'$. If $F$ induces an isomorphism on the homologies of the associated graded objects of $C$ and $C'$, then $F$ induces isomorphisms on all higher pages of the corresponding spectral sequences. In particular, if the spectral sequences are bounded, then $F$ is a quasi-isomorphism.

\end{lem}

\begin{lem} \label{derivedlemma}

Let $C_{1}$, $C_{1}'$ and $C_{2}$ be finitely generated chain complexes over the ring $R=\Q[U_{1},...,U_{n}]$, and suppose $C_{2}$ is a free $R$-module. If $f: C_{1} \to C'_{1}$ is a quasi-isomorphism, then

\[ f \otimes 1: C_{1} \otimes C_{2} \to C'_{1} \otimes C_{2} \]

\noindent
is also a quasi-isomorphism.
\end{lem}

\begin{proof}

If we filter the complexes $C_{1} \otimes C_{2}$ and $C'_{1} \otimes C_{2}$ by the homological grading on $C_{2}$, then $f \otimes 1$ is a filtered chain map which induces an isomorphism on the $E_{1}$ pages of the associated spectral sequences. By the previous lemma, it is a quasi-isomorphism.

\end{proof}

\subsection{Khovanov homology and $\widehat{C}_{2}(D)$}

\begin{lem} For any diagram $D \in \mathcal{D^{R}}$ and any complete resolution $D_{I}$ of $D$, the homology $H^{-}_{2}(D_{I})$ is isomorphic to $\hfk^{-}_{2}(D_{I})$ as graded $R$-modules. 

\end{lem}

\begin{proof}

The complex $\cfk^{-}_{2}(D_{I})$ is given by 
\[ \cfk^{-}_{2}(\h(D_{I})) \otimes  \mcL_{I}^{+}(D_{I}) \otimes \mcL_{D}^{+}(D)  \]

\noindent
where $\cH$ is a Heegaard diagram for $D_{I}$. Suppose $\cH$ is the standard planar Heegaard diagram for $D_{I}$. Let $x_{1}$ denote the generator of $\cfk^{-}_{2}(\cH)$ with minimal algebraic grading, and let $y_{1}$ denote the generator of  $\mcL_{I}^{+}$ with minimal algebraic grading.

Let $f: \cfk^{-}_{2}(\cH) \otimes  \mcL_{I}^{+}(D_{I}) \to R/(N(D_{I}) +L_{I}(D_{I}))$ be the map which sends $x_{1} \otimes y_{1}$ to the generator of $R/(N(D_{I}) +L_{I}(D_{I}))$, and is trivial on all other generators. The map $f$ is a morphism of curved complexes. Moreover, if we filter with respect to the Alexander grading, then it follows from Lemma \ref{nonloc} and the fact that $\{L(v)|v \in v_{4}(I) \}$ form a regular sequence over $R/N(D_{I})$ that $f$ induces an isomorphism on the associated graded objects.

The map $f \otimes \id: \cfk^{-}_{2}(D_{I}) \to C^{-}_{2}(D_{I})$ is also chain map. Since $ \mcL_{I}(D_{I})$ is a free $R$-module, $f \otimes \id$ still induces an isomorphism on the associated graded objects with respect to the Alexander grading. It follows from Lemma \ref{homalglemma} that $f \otimes \id$ is a quasi-isomorphism.

\end{proof}

Together with Theorem \ref{hfk2kh}, this gives the following corollary:

\begin{cor} \label{corkhov}

For any $D \in \mathcal{D^{R}}$ and any resolution $D_{I}$ of $D$, there is an isomorphism of graded $R$-modules
\[ C_{2}^{-}(D_{I}) \cong Kh^{-}(\sm(D_{I})) \]

\end{cor}

%
%
%
%

\begin{thm} \label{E2thm}

Let $D$ be a diagram in $\mathcal{D^{R}}$. The $E_{2}$ page of the spectral sequence on $C_{2}^{-}(D)$ induced by the cube filtration is isomorphic to $Kh^{-}(m(\sm(D)))$ as a (relatively) graded $R$-module.

\end{thm}

\begin{proof} 

Let $D_{+} \in \mathcal{D^{R}}$ be a diagram with a marked positive crossing $c$, and let $D_{-}$ be the diagram obtained by switching it to a negative crossing. Suppose the outgoing edges are labeled $e_{a}$, $e_{b}$ and the incoming edges are labeled $e_{c}$, $e_{d}$. Note that $D_{-}$ is also an element of $\mathcal{D^{R}}$.

Let $D_{I}$ and $D_{J}$ be two resolutions of $D_{+} $ with $I \lessdot J$. In $C_{2}^{-}(D_{+})$, there is an edge map
\[ d_{I,J}: C_{2}^{-}(D_{I}) \to C_{2}^{-}(D_{J}) \]

\noindent
But $D_{I}$ and $D_{J}$ are also resolutions of $D_{-}$, with $J \lessdot I$, and in $C_{2}^{-}(D_{-})$, there is an edge map
\[ d_{J,I}: C_{2}^{-}(D_{J}) \to C_{2}^{-}(D_{I}) \]

Recall that $d_{I,J}=\phi_{+} \otimes \id$ with $\phi_{+}(1)=1$, while $d_{J,I} = \phi_{-}(1) \otimes \id$ with $\phi_{-}(1)=U_{b}-U_{c}$, so 
\[ d_{I,J} \circ d_{J,I} = d_{J,I} \circ d_{I,J} = U_{b}-U_{c} \]

It follows that the composition of these two edge after taking homology with respect to $d_{0}$ is also multiplication by $U_{b}-U_{c}$. But after taking homology with respect to $d_{0}$, the complex is isomorphic to the Khovanov complex as an $R$-module. We claim that since the edge maps are $R$-module homomorphisms, this composition uniquely determines the induced maps $d^{*}_{I,J} $ and $ d^{*}_{J,I} $ up to multiplication by some non-zero element of $\Q$.

Without loss of generality, suppose $d^{*}_{I,J} $ corresponds to split of two circles and $ d^{*}_{J,I} $ to a merge of two circles. Viewing $d^{*}_{I,J} $ as a map from $Kh^{-}(\sm(D_{I}))$ to $Kh^{-}(\sm(D_{J}))$, we can not have $d^{*}_{I,J}(1)=1$, as this is not an $R$-module map. Thus, it must have polynomial degree at least one. Since $d^{*}_{I,J} \circ d^{*}_{J,I}$ has degree $1$, this implies $d^{*}_{I,J}$ has polynomial degree $1$ and $d^{*}_{J,I}$ has degree $0$.

It follows that $d^{*}_{J,I}(1) = r$ for some $r \in \Q$. Then since $d^{*}_{I,J} \circ d^{*}_{J,I}(1) = U_{b}-U_{c}$, $d^{*}_{I,J}(1) = (1/r)(U_{b}-U_{c})$. Thus, we have shown that $d^{*}_{I,J} $ and $ d^{*}_{J,I} $ are the Khovanov edge maps up to scaling. Any such complex is isomorphic to the Khovanov complex after rescaling the vertices as necessary. The mirroring comes from the fact that after smoothing, the complex $C_{2}^{-}(D_{+})$ maps from the unoriented smoothing to the oriented smoothing which corresponds to $Kh(D_{-})$. Similarly, after smoothing, $C_{2}^{-}(D_{-})$ maps from the oriented smoothing to the unoriented smoothing, which corresponds to $Kh(D_{+})$.

\end{proof}

\begin{cor}\label{khovanovE2}

Let $D$ be a diagram in $\mathcal{D^{R}}$. The $E_{2}$ page of the spectral sequence on $\widehat{C}_{2}(D)$ induced by the cube filtration is isomorphic to the pointed Khovanov homology $\widehat{Kh}(m(\sm(D)))$ as a (relatively) graded vector space.

\end{cor}

\begin{proof}

Recall that the complex $\widehat{C}_{2}(D)$ is given by
\[ \widehat{C}_{2}(D) =  C_{2}^{-}(D) \otimes \bigotimes_{j=1}^{l} R \xrightarrow{U_{i_{j}}} R \]

\noindent
where one edge $e_{i_{j}}$ lies on each component of $\sm(D)$ and the maps $R \xrightarrow{U_{i_{j}}} R$ have cube grading $1$. Consider the spectral sequence induced by the cube filtration on $\widehat{C}_{2}(D)$. Then the complex $(E_{1}, d_{1})$ is precisely the pointed Khovanov complex $CKh(m(\sm(D)), \bf{p})$, where $ \mathbf{p} = \{p_{1},...,p_{l}\}$ and $p_{j}$ lies on $e_{i_{j}}$. The $E_{2}$ page is the homology of this complex, which is defined to be $\widehat{Kh}(m(\sm(D)))$.

\end{proof}

\section{Relationship with $\widehat{\hfk}(L)$}

In this section we will show that for any $D \in \mathcal{D^{R}}$, $C_{2}^{-}(D)$ is quasi-isomorphic to $\cfk_{2}^{-}(D)$. Thus, in order to get a comparison with $\widehat{\hfk}(\sm(D))$, we will need to relate $\widehat{\cfk}_{2}(D)$ and $\widehat{\hfk}(\sm(D))$.

\subsection{$\widehat{\cfk}_{2}(D)$ and $\widehat{\cfk}(\sm(D))$.}

Let $D$ be a partially singular diagram.

\begin{lem}

There is an isomorphism $\widehat{\hfk}_{2}(D) \cong \widehat{\hfk}(\sm(D))$ as $\delta$-graded vector spaces.

\end{lem}

\begin{proof}

Let $\cH$ be a Heegaard diagram for $D$, and let $\cH'$ be a Heegaard diagram for $\sm(S)$ obtained from $\cH$ by deleting the $XX$ basepoints and swapping some $X$ basepoints with $O$ basepoints. Note that this choice is not unique, as it requires placing an orientation on $\sm(S)$, but the $\delta$-graded, reduced knot Floer homology does not depend on orientation, so  the choice doesn't matter.

Let $m$ be the number of $XX$ basepoints in $\cH$, $n$ be the number of $O$ basepoints in $\cH$, and $n'$ the number of $O$ basepoints in $\cH'$. Then $n'=n-m$. Finally, let $l$ denote the number of components of $L$.

Let $\widetilde{\cfk}_{2}(D)$ be the complex $\cfk^{-}_{2}(D)$ with every $U_{i}=0$, and define $\widetilde{\cfk}(\sm(D))$ similarly. If two $O$ basepoints in $\cH$ lie on the same component of $\sm(D)$, the corresponding $U$ actions are homotopic (up to sign). Similarly, if two $O$ basepoints in $\cH'$ lie on the same component of $\sm(D)$, the corresponding $U$ actions are homotopic. It follows that 
\[ \widetilde{\hfk}_{2}(D) \cong \widehat{\hfk}_{2}(D) \otimes V ^{\otimes n-l} \hspace{2cm} \widetilde{\hfk}(\sm(D)) \cong \widehat{\hfk}(\sm(D)) \otimes V ^{\otimes n'-l} \]

\noindent
where $V = \Q \oplus \Q$ concentrated in $\delta$-grading $0$.

Since  $\cfk_{2}^{-}(D) = \cfk^{-}_{2}(\cH) \otimes \mathcal{L}^{+}(D)$, we can write 
\[\widetilde{\cfk}_{2}(D) = \widetilde{\cfk}_{2}(\cH) \otimes V^{\otimes m} \]

\noindent
as $L(v)$ and $L^{+}(v)$ are both zero in $\widetilde{\cfk}_{2}(D)$.

The complex $\widetilde{\cfk}_{2}(\cH)$ is nearly identical to $\widetilde{\cfk}_{2}(\cH')$ - in both cases all $X$ and $O$ basepoints are blocked, and they have $X$ and $O$ basepoints in the same locations (though some $X$'s have been swapped for $O$'s and vice versa). The only difference is that $\cH$ has a collection of $XX$ basepoints which contribute a factor of $(-2)^{n_{XX}(\phi)}$ to each homotopy class $\phi$. If $x$ is a generator of $ \cfk^{-}_{2}(\cH)$, we can perform a change of basis 

\[ x \mapsto (-1/2)^{A(x)}x \]

\noindent
where $A(x)$ is the Alexander grading of $x$. Then after this change of basis, we have an isomorphism

\[\widetilde{\cfk}_{2}(\cH) \cong \widetilde{\cfk}(\cH')\]

\noindent
The relative $\delta$-grading is given by 
\[ \gr_{\delta}(x) - \gr_{\delta}(y) = 2\mu(\phi) - 2n_{\bf{X}}(\phi)  - 2n_{\bf{O}}(\phi) \]

\noindent
so it is the same on $\widetilde{\cfk}_{2}(\cH)$ and $\widetilde{\cfk}(\cH')$. Thus, the above isomorphism is an isomorphism of (relatively) $\delta$-graded chain complexes.

It follows that 

\[\widetilde{\hfk}_{2}(D) \cong \widetilde{\hfk}_{2}(\cH) \otimes V^{\otimes m} \cong \widetilde{\hfk}(\cH') \otimes V^{\otimes m} \cong \widetilde{\hfk}(\sm(D)) \otimes V^{\otimes m} \]

\noindent
Thus,

\[ \widehat{\hfk}_{2}(D) \otimes V^{\otimes n-l} \cong \widehat{\hfk}(\sm(D)) \otimes V^{\otimes (m+n'-l)} \]

\noindent
Since $m+n'=n$, it follows that $\widehat{\hfk}_{2}(D) \cong \widehat{\hfk}(\sm(D))$ as $\delta$-graded vector spaces.

\end{proof}

\subsection{ $C^{-}_{2}(S)$ and $\cfk^{-}_{2}(S)$} Let $D \in \mathcal{D^{R}}$ be a partially singular braid. In this section we will define a quasi-isomorphism from $\cfk^{-}_{2}(D)$ to $C^{-}_{2}(D)$.

Let $\cH$ be the standard planar Heegaard diagram for $D$. We can write $\cfk^{-}(D) = \cfk^{-}(\cH') \otimes \mathcal{L}_{D}$. We can apply the cube of resolutions construction from Section \ref{cube} to $\cfk^{-}(D)$ to obtain the complex $C^{-}_{F}(D)$. There is a chain homotopy equivalence 
\[ f_{1}: \cfk^{-}(D) \to C^{-}_{F}(D) \]

\noindent
This homotopy equivalence is described at a particular negative and positive crossing, respectively, in Figures \ref{f1minus} and \ref{f1plus}, and the total map is obtained by iterating over all crossings.

\begin{figure}[!h]
\centering
\begin{tikzpicture}
  \matrix (m) [matrix of math nodes,row sep=5em,column sep=6em,minimum width=2em] {
     \X & \X \\
     \Y & \X \\
     \Y & \X \\};
  \path[-stealth]
    (m-1-1) edge node [left] {$\Phi_{A^{-}}$} (m-2-1)
            edge node [above] {$1$} (m-1-2)
            edge node [right]{$\Phi_{A^{-}B}$} (m-2-2)
    (m-2-1.east|-m-2-2) edge node [above] {$\Phi_{B}$} (m-2-2)
    (m-3-1.east|-m-3-2) edge node [above] {$\Phi_{B}$} (m-3-2)
    (m-3-1) edge [dashed] node [left] {$1$} (m-2-1)
    (m-3-2) edge [dashed] node [left] {$1$} (m-2-2)
    (m-1-2) edge node [right] {$L(v)$ \hspace{ 15mm}} (m-2-2);
\end{tikzpicture}
\caption{The chain homotopy equivalence $f_{1}$ for a negative crossing is depicted via the dashed arrows} \label{f1minus}
\end{figure}
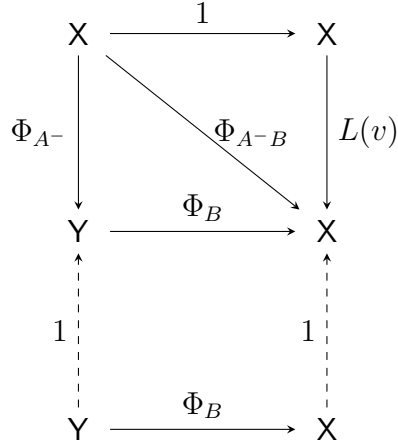

\begin{figure}[!h]
\centering
\begin{tikzpicture}
  \matrix (m) [matrix of math nodes,row sep=5em,column sep=6em,minimum width=2em] {
     \X & \Y \\
     \X & \Y \\
     \X & \X \\};
  \path[-stealth]
    (m-1-1) edge  [dashed]  node [left] {$1$} (m-2-1)
            edge node [above] {$\Phi_{B}$} (m-1-2)
            edge [bend right, dashed] node [left] {$\Phi_{A^{+}B}$} (m-3-1)
    (m-2-1.east|-m-2-2) edge node [above] {$\Phi_{B}$} (m-2-2)
    (m-2-1)    edge node [right]{$\Phi_{A^{-}B}$} (m-3-2)
    (m-3-1.east|-m-3-2) edge node [above] {$1$} (m-3-2)
    (m-2-1) edge  node [right] {$L(v)$} (m-3-1)
    (m-2-2) edge node [right] {$\Phi_{A}$} (m-3-2)
    (m-1-2) edge  [dashed]  node [right] {$1$ \hspace{ 15mm}} (m-2-2);
    
    \node at (3.6,-2.5) {$\Phi_{A}$};
       
    \draw [->, -stealth, dashed] (m-1-2) 
           to [out=-45,in=45] (2.5,-3.5) 
           to [out=-135,in=-60] (m-3-1);
           
\end{tikzpicture}
\caption{The chain homotopy equivalence $f_{1}$ for a positive crossing is depicted via the dashed arrows} \label{f1plus}
\end{figure}

Since $\mathcal{L}_{D}$ appears at each vertex in the cube of resolutions, we can write $C^{-}_{F}(D) = C'_{F}(D) \otimes \mathcal{L}_{D}$. The complex $\cfk^{-}(D)$ also has a $\mathcal{L}_{D}$ tensorand, and $f_{1}$ is the identity on $\mathcal{L}_{D}$. Define 

\[g_{1}: \cfk^{-}(\cH) \to C'_{F}(D) \]

\noindent
so that $f_{1}=g_{1} \otimes \id$.

Similarly, the complex $C_{2}^{-}(D)$ decomposes as a tensor product $C^{-}_{2}(D) = C_{2}'(D) \otimes \mathcal{L}_{D}$, where $C_{2}'(D)$ is the cube complex which has $R/(N(D_{I})+L_{I}(D_{I}))$ at each vertex $I$.

\begin{lem}\label{lem5.4}

The complex $C'_{F}(D)$ is quasi-isomorphic to $C_{2}^{'}(D)$.

\end{lem}

\begin{proof}

This lemma has nearly been proved for us in Lemma 5.1 and Proposition 5.2 of \cite{Szabo}. The complex $C'_{F}(D)$ decomposes as a direct sum 

\[ C'_{F}(D) = \bigoplus_{I} C'_{F}(D_{I}) \]

For each diagram $D_{I}$, the complex $C'_{F}(D_{I})$ is given by $\cfk^{-}(\cH_{I}) \otimes \mathcal{L}_{I}(D_{I})$. Since $D \in \mathcal{D^{R}}$, the homology of this complex is $R/(N(D_{I})+L_{I}(D_{I}))$. Ozsv\'{a}th and Szab\'{o} pick out generators for this homology from the cube complex. Before describing these generators, we'll need some notation.

For each crossing, there are three pieces of the corresponding chain complex labeled $\X$. To distinguish them, we will call them $\X_{\alpha}$, $\X_{\beta}$, and $\X_{\gamma}$. A generator of each of these complexes corresponds to a tuple of intersection points in the Heegaard diagram in Figure \ref{HDCrossing2} containing the intersection point $x$. When writing generators of the complex in terms of their intersection points, we will write this intersection point as $x_{\alpha}$, $x_{\beta}$, or $x_{\gamma}$, depending on which summand of the complex it corresponds to.

\begin{figure}[!h]
\centering
\begin{tikzpicture}
  \matrix (m) [matrix of math nodes,row sep=5em,column sep=6em,minimum width=2em] {
     \X_{\alpha} & \X_{\beta} \\
     \Y & \X_{\gamma} \\};
  \path[-stealth]
    (m-1-1) edge node [left] {$\Phi_{A^{-}}$} (m-2-1)
            edge node [above] {$1$} (m-1-2)
            edge node [right]{$\Phi_{A^{-}B}$} (m-2-2)
    (m-2-1.east|-m-2-2) edge node [above] {$\Phi_{B}$} (m-2-2)
    (m-1-2) edge node [right] {$L(v)$ \hspace{ 15mm}} (m-2-2);
\end{tikzpicture}
\caption{A labeling of the $\X$ summands in the negative crossing complex} \label{negcomplex2}
\end{figure}

Let $D_{I}$ and $D_{J}$ be two complete resolutions of $D$ with $I \lessdot J$. Suppose first that the corresponding crossing is negative, so that $D_{I}$ has the smoothing and $D_{J}$ the singularization.

Ozsv\'{a}th and Szab\'{o} show that in $C_{F}'(D_{I})$, there are two homologous generators: $\mathbf{y}_{0}$, which has intersection points $x'$ and $y$, and $\mathbf{y}_{1}$, which has intersection points $u$ and $v$ in Figure \ref{HDCrossing2}. They are both cycles, and they both generate the homology at that vertex in the cube of resolutions. Since they are homologous, if we perform a change of basis $\mathbf{y}_{1} \mapsto  \mathbf{y}_{1}-\mathbf{y}_{0}=\mathbf{y}_{1}'$ at each crossing, then there is a unique generator $\mathbf{y}_{0}$ which generates the homology $R/(N(D_{I})+L_{2}(D_{I}))$. When the crossing is singularized, there is a unique local generator $\mathbf{x}_{0}$, which has intersection points $x^{\gamma}$ and $y$, which generates the homology of $C'_{F}(D_{J})$. They also show that the edge map in the cube of resolutions $C_{F}^{-}(D)$ carries $\mathbf{y}_{0}$ to 
$(-1)^{\epsilon_{I,J}}(U_{b}-U_{c})\mathbf{x}_{0}$, where $\epsilon$ is a sign assignment.

\begin{figure}[h!] 

 \centering
   \begin{overpic}[width=.75\textwidth]{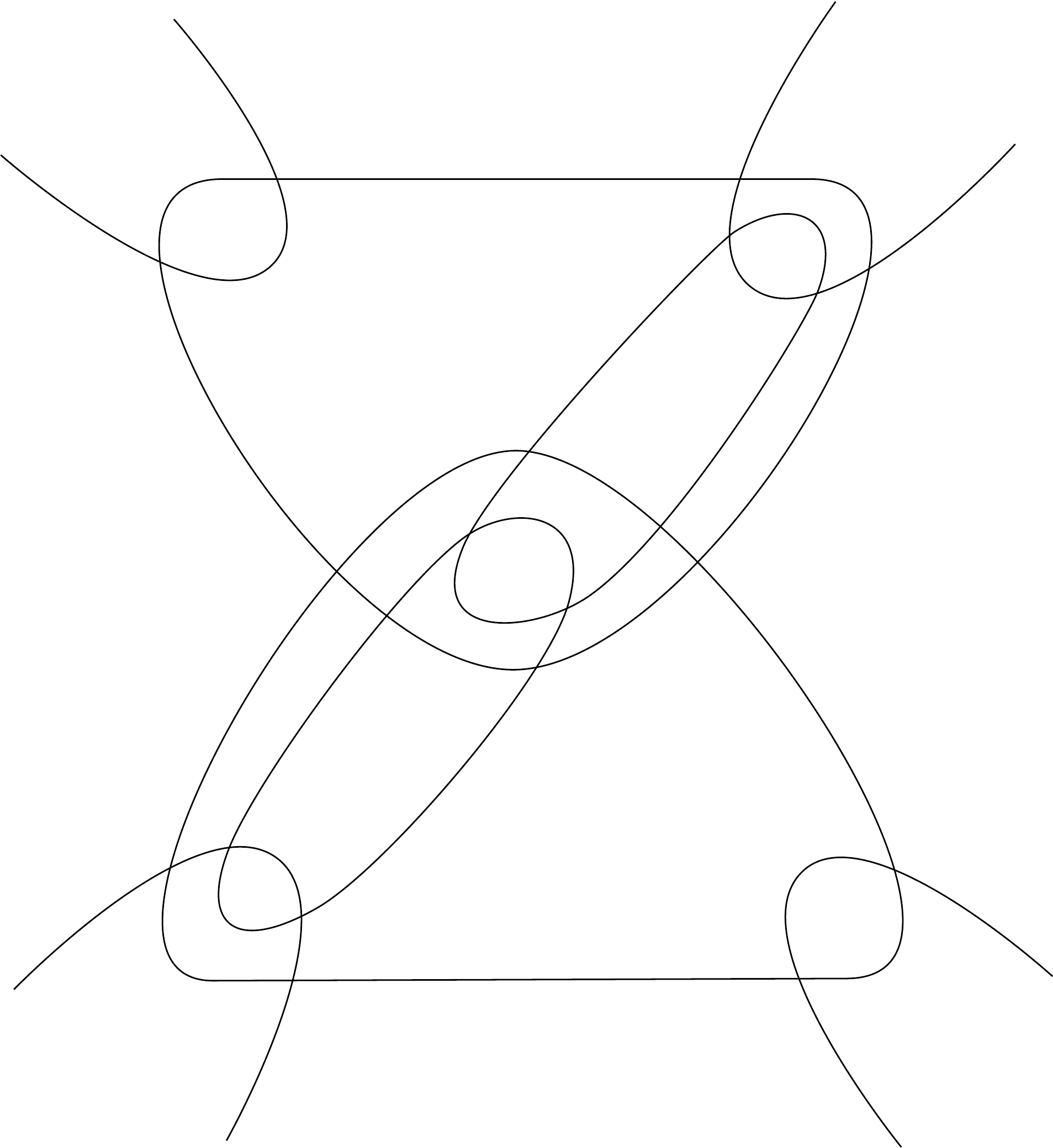}
   \put(39,54){$x$}
   \put(40.4,53){$\bullet$}
   \put(49.3,45){$x'$}
   \put(48.5,46.3){$\bullet$}
   \put(17.6,79){$O_{1}$}
   \put(66,77){$O_{2}$}
   \put(48,55.7){$B$}
   \put(43,49.7){$A_{0}$}
   \put(53,48){$A_{+}$}
   \put(38.2,44){$B$}
   \put(34,52){$A_{-}$}
   \put(21,21.6){$O_{3}$}
   \put(71,19){$O_{4}$}
   \put(43,86){$\alpha_{1}$}
   \put(50,70){$\alpha_{2}$}
   \put(45,12){$\beta_{1}$}
   \put(40,30.4){$\beta_{2}$}

   \put(44,62.1){$v$}
   \put(58.5,53.6){$v'$}
   \put(45.4,59.95){$\bullet$}
   \put(56.8,53.45){$\bullet$}
   \put(28.55,49.55){$\bullet$}
   \put(25.5,49.55){$y$}
   \put(60,50.3){$\bullet$}
   \put(62,50.3){$y'$}
   \put(33,45.5){$\bullet$}
   \put(30,45.65){$u$}
   \put(46,41.2){$\bullet$}
   \put(47.2,40){$u'$}

   \end{overpic}
  
\caption{Labeled intersection points in the Heegaard diagram for a crossing}\label{HDCrossing2}
\end{figure}

Consider now the positive crossing. In this case $D_{I}$ is the singularization and $D_{J}$ the smoothing. The complexes have the same generators of homology $\mathbf{x}_{0}$ for $C'(D_{I})$ and homologous generators $\mathbf{y}_{0}$, $\mathbf{y}_{1}$ for $C'(D_{J})$, though now $\mathbf{x}_{0}$ lies in $\mathsf{X}'_{\beta}$ while $\mathbf{y}_{0}$ lies in $\mathsf{X}'_{\gamma}$. Once again, we perform the change of basis $\mathbf{y}_{1} \mapsto  \mathbf{y}_{1}-\mathbf{y}_{0}=\mathbf{y}_{1}'$ so that there is a unique generator of homology at each vertex in the cube. Ozsv\'{a}th and Szab\'{o} show that the the edge map in the cube of resolutions carries $\mathbf{x}_{0}$ to $(-1)^{\epsilon_{I,J}} \mathbf{y}_{0}$.

\begin{figure}[!h]
\centering
\begin{tikzpicture}
  \matrix (m) [matrix of math nodes,row sep=5em,column sep=6em,minimum width=2em] {
     \X_{\alpha}' & \Y' \\
     \X_{\beta}' & \X_{\gamma}' \\};
  \path[-stealth]
    (m-1-1) edge node [left] {$U_{a}+U_{b}-U_{c}-U_{d}$} (m-2-1)
            edge node [above] {$\Phi_{B'}$} (m-1-2)
            edge node [right]{$\Phi_{A^{+}B}$} (m-2-2)
    (m-2-1.east|-m-2-2) edge node [above] {1} (m-2-2)
    (m-1-2) edge node [right] {$\Phi_{A^{+}}$ \hspace{ 15mm}} (m-2-2);
\end{tikzpicture}
\caption{A labeling of the $\X$ summands in the positive crossing complex}
\end{figure}

Suppose we choose the same assignment on the edges of $C'_{F}(D)$. Then sending each unique generator of the homology at a vertex in $C'_{F}(D)$ to the generator at the corresponding vertex in $C^{-}(D)$ is a quasi-isomorphism. For grading reasons, there are no higher face maps. In particular, the generator at each vertex in the cube has algebraic grading $-|I|$ up to an overall shift and all differentials have algebraic grading $-1$.

\end{proof}

Let $g_{2}$ denote this quasi-isomorphism. This map can be lifted up to the total complexes by tensoring with the identity:

\[ g_{2} \otimes \id: C'_{F}(D) \otimes \mathcal{L}_{D}\to C^{'}(D) \otimes  \mathcal{L}_{D} \]

\noindent
Call this map $f_{2}$. 

\begin{lem}The map $f_{2}$ is a quasi-isomorphism.
\end{lem}

\begin{proof}

This follows from Lemma \ref{derivedlemma}, as $\mathcal{L}_{D}$ is a free $R$-module.

\end{proof}

\begin{cor}\label{happycor}
The map $f_{2} \circ f_{1}: \cfk^{-}(D) \to C_{2}'(D) \otimes  \mathcal{L}_{D}$ is a quasi-isomorphism. 
\end{cor}

As modules, $\cfk^{-}(D)$ is canonically isomorphic to $\cfk^{-}_{2}(D)$, and $C'_{2}(D) \otimes  \mathcal{L}_{D}$ is canonically isomorphic to $C^{-}_{2}(D)$, so $f_{2} \circ f_{1}$ can also be viewed as a map from $\cfk^{-}_{2}(D)$ to $C_{2}^{-}(D)$.

\begin{lem}

Viewed as a map from $\cfk^{-}_{2}(D)$ to $C^{-}_{2}(D)$, $f_{2} \circ f_{1}$ is a chain map.

\end{lem}

\begin{proof} The complex $\cfk^{-}_{2}(D)$ can be written 
\[ \cfk_{2}(D) = \cfk^{-}_{2}(\cH') \otimes \mathcal{L}_{D}^{+}  \]

\noindent
while $C^{-}_{2}(D)$ can be written 
\[ C^{-}_{2}(D) = C'_{2}(D) \otimes \mathcal{L}_{D}^{+}  \]

Since the identity on $\mathcal{L}_{D}^{+}$ is a chain map, it suffices to show that 
\[ g_{2} \circ g_{1}: \cfk^{-}_{2}(\cH) \to C'_{2}(D) \]

\noindent
is a morphism of curved complexes, i.e. it commutes with the differentials. Let $\omega$ denote the curvature on $\cfk^{-}_{2}(\cH)$ so that $d^{2}=\omega I $ on this complex. Note that it makes sense to view $C_{2}'(D)$ as a curved complex with curvature $\omega$ even though $d^{2}=0$ because $\omega C_{2}'(D)=0$.

The complex $ \cfk^{-}_{2}(\cH)$ can be filtered by the Alexander grading, so that $d=d_{0} +d_{1}+...$ where $d_{i}$ decreases the Alexander grading by $2i$. We already have that
\[ (g_{2} \circ g_{1}) \circ d_{0} = d \circ (g_{2} \circ g_{1}) \]

\noindent
since this was the quasi-isomorphism from Lemma \ref{lem5.4}. Thus, it suffices to show that \[(g_{2} \circ g_{1}) \circ d_{i} = 0\] for all $i \ge 1$.

The differential $d_{i}$ has algebraic grading $1-2i$. The generators of homology have minimal algebraic grading at their respective vertices in the cube, so in order for $(g_{2} \circ g_{1}) \circ d_{i}$ to be non-zero, it must decrease the cube grading by at least $2i-1$. However, we see by inspection that it can decrease the cube grading by at most $i$. Thus, $(g_{2} \circ g_{1}) \circ d_{i} = 0$ for all $i > 1$.

The last case to check is $i=1$. The underlying disc $\phi$ must pass through a single $A$ basepoint - let $c$ be the corresponding crossing. The composition $(g_{2} \circ g_{1}) \circ d_{1}(a)$ can only be non-zero in the case that $a$ has minimal algebraic grading at $c$, so we only need to check the cases $a=\mathbf{x}_{0}$ in the case of a negative crossing and $a=\mathbf{y}_{1}$ in the case of a positive crossing. (Note that we don't have to check $\mathbf{y}_{0}$, as it belongs to $C_{F}^{-}(D)$, not $\cfk^{-}(D)$.)

Suppose first that $c$ is a negative crossing, so the only choice of generator where the image of $(g_{2} \circ g_{1}) \circ d_{1}$ may not be zero is $\mathbf{x}_{0}$. We need to check the coefficient of $\partial(\mathbf{x}_{0})$ at $\mathbf{y}_{0}$ and $\mathbf{y}_{1}$.

The bigon containing $A_{0}$ gives a differential from $\mathbf{x}_{0}$ to $\mathbf{y}_{0}$ with coefficient $\pm(U_{b}+U_{c})$, while the rectangle containing $A^{-}$ gives a differential from $\mathbf{x}_{0}$ to $\mathbf{y}_{1}$ with coefficient $\pm(U_{a}+U_{d})$. We have chosen a system of orientations so that these discs come with the same sign - this was why $\mathbf{y}_{0}$ was homologous to $\mathbf{y}_{1}$ instead of $-\mathbf{y}_{1}$. After the change of basis $\mathbf{y}_{1} \mapsto \mathbf{y}'_{1}$, the differential from $\mathbf{x}_{0}$ to $\mathbf{y}_{0}$ has coefficient $\pm(-U_{a}+U_{b}+U_{c}-U_{d})$, which is zero in $R/(N(D_{I})+L_{I}(D_{I}))$. 

The positive crossing case is a bit trickier as the relevant part of $g_{1}$ is now $\Phi_{A}$ instead of an identity map, so we have to count all compositions of two discs with multiplicity 1 at the $A$ basepoints at $c$ which take $\mathbf{y}_{1}$ to $\mathbf{x}_{0}$. The disc involved in $d_{1}$ is allowed to pass through $B$ basepoints as well, while the one in $\Phi_{A}$ is not. Additionally, the $A$ basepoints in $d_{1}$ count with their coefficients $U_{i}+U_{j}$, while in $\Phi_{A}$ no coefficient is picked up.

In the language of \cite{Szabo}, all of the generators involved so far have belonged to the empty cycle $Z = \emptyset$, as these have the lowest algebraic grading. The only way for a disc to map from the empty cycle to a non-empty cycle is by passing through the special $X$ basepoint corresponding to the marked edge (see Figure \ref{MarkedEdgeHD}). Passing through this basepoint would increase the algebraic grading by an additional two, so we can rule out any such discs. It follows that the only non-zero multiplicities our two discs are allowed to have must be in the interior of $\alpha_{1}$ union the interior of $\beta_{1}$ in Figure \ref{HDCrossing2}.

The result now follows from direct computation. In particular, 
\[ d_{1}(\mathbf{y}_{1}) = (U^{2}_{a}-U^{2}_{d})(x,y) + (U_{b}+U_{c})(u,v') +(U_{b}+U_{c})(u',v) \]

\noindent
where the ordered pairs indicate the intersection points. Next we will count the coefficient of each of these intersection points at $\mathbf{x}_{0}=(x',y)$ under $\Phi_{A}$.
\[ \Phi_{A}(x,y) = \mathbf{x}_{0}\] \[ \Phi_{A}(u,v') = U_{d}(U_{b}+U_{c})\mathbf{x}_{0} \] \[ \Phi_{A}(u',v) = -U_{a}(U_{b}+U_{c})\mathbf{x}_{0} \]

Putting this together, we get that 
\begin{equation}
\begin{split}
g_{2} \circ g_{1} \circ d_{1}(\mathbf{y}_{1}) & = (U_{a}^{2}-U_{d}^{2} + (U_{d}-U_{a})(U_{b}+U_{c}))  \\
& = (U_{a}+U_{d})(U_{a}+U_{b}-U_{c}-U_{d}) -2(U_{a}U_{b}-U_{c}U_{d})  \\
& = 0  \in R/(N(D_{I})+L_{I}(D_{I}))
\end{split}
\end{equation}

\noindent
This completes the proof.

\end{proof}

\begin{thm} \label{hfkthm} 

Let $D$ be a diagram in $\mathcal{D^{R}}$, and consider $f_{2} \circ f_{1}$ as a map from $\cfk^{-}_{2}(D)$ to $C^{-}_{2}(D)$. Then up to an overall grading shift, $f_{2} \circ f_{1}$ is a quasi-isomorphism of graded complexes, where the grading on $\cfk^{-}_{2}(D)$ is $\gr_{\delta}$ and the grading on $C^{-}_{2}(D)$ is $\gr_{2}$.
\end{thm}

\begin{proof}

The complex $C^{-}_{2}(D)$ inherits an Alexander filtration from the surjection $f_{2}$. The differentials which break the Alexander filtration are the maps with coefficient $L^{+}(v)$ in $\mathcal{L}_{D}^{+}$ - they all decrease the Alexander grading by $2$.

The chain map $f_{2} \circ f_{1}$ is a filtered chain map with respect to the Alexander filtration. Moreover, by Corollary \ref{happycor}, it induces an isomorphism on the $E_{1}$ pages of the associated spectral sequences. Applying Lemma \ref{homalglemma}, $f_{2} \circ f_{1}$ induces an isomorphism on the total homologies. The $\delta$-grading on $\cfk^{-}_{2}(D)$ clearly descends to the grading $\gr_{2}$ on $C_{2}^{-}(D)$.

\end{proof}

The quasi-isomorphism $f=f_{2} \circ f_{1}$ extends to a map 
\[ \widehat{f}: \widehat{\cfk}_{2}  \to \widehat{C}_{2}(D) \]

\noindent
by acting as the identity on the factor \[\bigotimes_{j=1}^{l} R \xrightarrow{U_{i_{j}}} R\] By Lemma \ref{derivedlemma}, $\widehat{f}$ is also a quasi-isomorphism, so we have the following corollary:

\begin{cor}

Let $D$ be a diagram in $\mathcal{D^{R}}$. Then there is an isomorphism of (relatively) graded vector spaces
\[ \widehat{\hfk}_{2}(D) \cong \widehat{H}_{2}(D) \]

\end{cor}

\noindent
Putting this together with the isomorphism of $\delta$-graded homologies $\widehat{\hfk}_{2}(D) \cong \widehat{\hfk}(\sm(D))$, we get a reduced version of the previous theorem:

\begin{thm}\label{reducediso}

Let $D$ be a diagram in $\mathcal{D^{R}}$. There is an isomorphism of (relatively) graded vector spaces
\[ \widehat{\hfk}(\sm(D)) \cong \widehat{H}_{2}(D) \]

\noindent
where the grading on $\widehat{\hfk}(\sm(D))$ is the $\delta$-grading.

\end{thm}

\section{Constructing the spectral sequence} \label{lastsection}

We have now related the $E_{2}$ and $E_{\infty}$ pages of $\widehat{C}_{2}(D)$ to $\widehat{Kh}(m(\sm(D)))$ and $\widehat{\hfk}(\sm(D))$, respectively, for any diagram $D \in \mathcal{D^{R}}$. It remains to show that the diagrams $\mathcal{D^{R}}$ are sufficient.

Let $S_{2n}$ be the singular open braid on $2n$ strands which consists of two layers: on the top layer, it has a single 4-valent vertex between strands $2i-1$ and $2i$ for $i=1,...,n$, and on the bottom layer, it has a single 4-valent vertex between strands $2i$ and $2i+1$ for $i=1,...,n-1$. See Figure \ref{singulartop} for the $n=4$ diagram.

\begin{figure}
 \centering
\def\svgwidth{6cm}
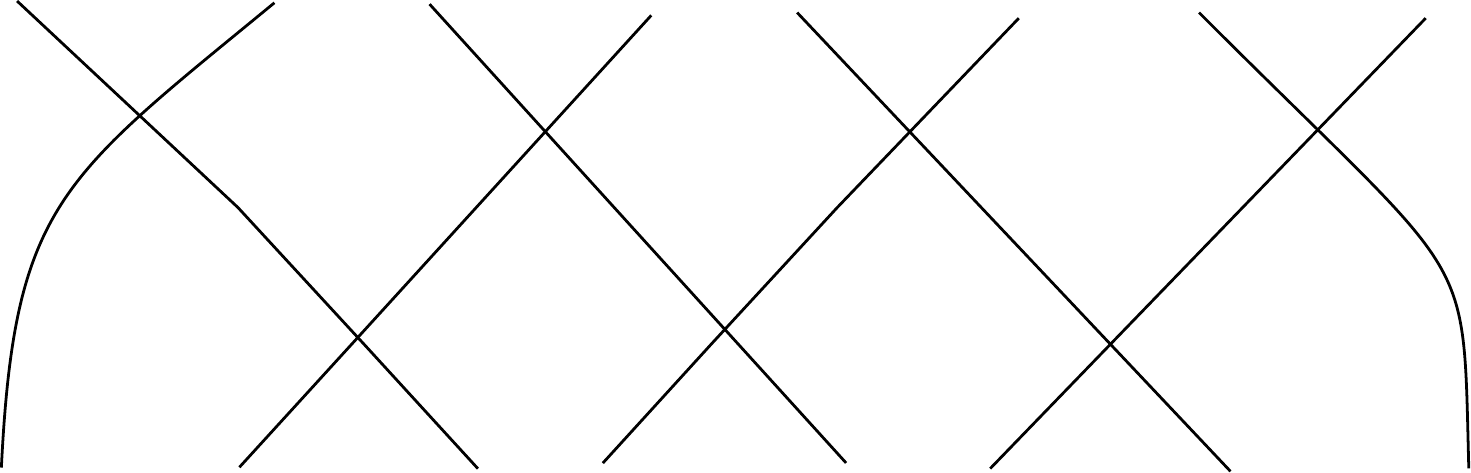
  \caption{The diagram $S_{2n}$ for $n=4$.}  \label{singulartop}
\end{figure}%

\begin{lem} \label{regsequence}

Let $D$ be a (partially singular) braid diagram on $2n$ strands which contains the diagram $S_{2n}$. Then $D$ is an element of $\mathcal{D^{R}}$.

\end{lem}

\begin{proof}

Let $D$ be such a diagram, and let $D_{I}$ be a complete resolution of $D$. We need to show two things:

\begin{itemize}

\item $D_{I}$ is connected
\item The linear terms $L(v)$ for $v \in v_{4}(I)$ form a regular sequence over $R/N(S)$.

\end{itemize}

The diagram $D_{I}$ is clearly connected since $S_{2n}$ connects all $2n$ strands. Thus, it remains to show the regular sequence property.

Let $\cH$ be the standard planar Heegaard diagram for $D_{I}$. Since $H_{*}(\cfk^{-}(\cH)) \cong R/N(S)$, it suffices to show that the homology of $\cfk^{-}(\cH) \otimes \mathcal{L}_{I}$ lies in a single algebraic grading.

In \cite{OSS}, Ozsv\'{a}th, Stipsicz, and Szab\'{o} describe a generalized Kauffman states diagram for a singular braid $D_{I}$ with a marked edge. Suppose $D$ has $l+1$ singular vertices. By adding additional basepoints to the diagram, they give a filtration on the complex $\cfk^{-}(\cH)$ and show that the filtered homology splits over generalized Kauffman states.

To each generalized Kauffman state $\mathfrak{K}$, the filtered complex is given by 

\[\bigotimes_{j=1}^{l} (R \xrightarrow{U_{i_{j}}} R) \]

\noindent
There is a special singular vertex which has the marked edge as an incoming edge - if we label the other singular vertices $v_{1},...,v_{l}$, then each $U_{i_{j}}$ is an incoming edge at $v_{j}$. 

It's not hard to see that the $U_{i_{j}}$, together with the generators of $L_{I}$, form a regular sequence. Breaking $D_{I}$ into an open braid with $S_{2n}$ at the top, we can go layer by layer starting at the bottom. For each singular point $v_{j}$ in $D_{I}$ which is not in $S_{2n}$, we can view $U_{i_{j}}$ and $L(v_{j})$ as substituting for the two incoming edges. Since $L(v_{j})$ is only relating each edge to edges higher in the diagram $D_{I}$, these elements form a regular sequence. After taking homology of the Koszul complex on these elements, we have 

\[ R/\{U_{i_{j}}=0, L(v_{j})=0 \text{ for } v_{j} \in D_{I}\setminus S_{2n} \} \]

\noindent
This module is clearly free over $\Q[\{U_{k}|e_{k} \in S_{2n}\}]$, so the remaining elements $\{U_{i_{j}}|v_{j} \in S_{2n} \}$ form a regular sequence.

This shows that for each generalized Kauffman state, the filtered homology lies in a single algebraic grading. But Ozsv\'{a}th, Stipsicz, and Szab\'{o} show that the lowest algebraic grading generators for each Kauffman state have the same algebraic grading, so the $E_{1}$ page of the spectral sequence induced by the basepoint filtration all lies in a single algebraic grading. It follows that the total homology does as well.

\end{proof}

\begin{cor}

For any link $L$ in $S^{3}$, there is a diagram $D \in \mathcal{D^{R}}$ such that $\sm(D)$ is a diagram for $L$.

\end{cor}

\begin{proof}

Let $L$ be a link in $S^{3}$, and let $D_{1}$ be a braid diagram on $2n$ strands whose plat closure $\mathsf{p}(D_{1})$ is a diagram for $L$. Consider the open braid on $2n$ strands which is $S_{2n}$ with an additional singularization between strands $2i-1$ and $2i$ for $i=2,...n$ as in Figure \ref{s1diag}. Call this diagram $S'_{2n}$.

\begin{figure}
 \centering
\def\svgwidth{6cm}
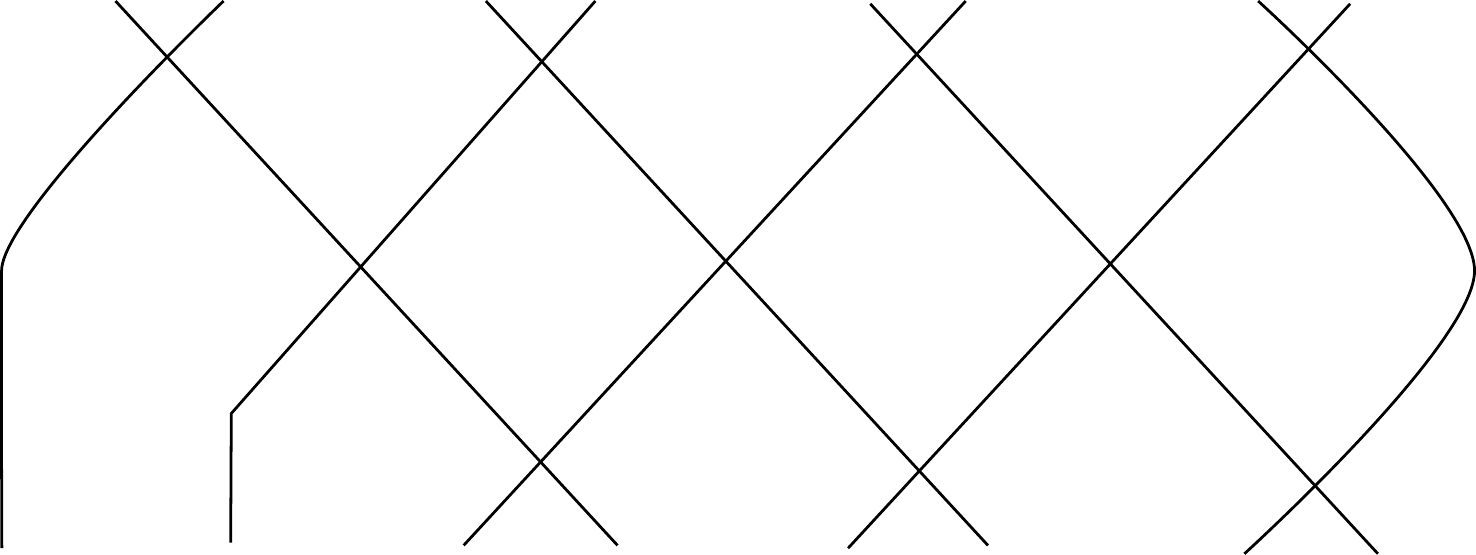
  \caption{The open braid $S_{2n}'$ for $n=4$}  \label{s1diag}
\end{figure}%

Let $D$ be the diagram obtained by concatenating $S'_{2n}$ with $D_{1}$, then taking the braid closure. Then $\sm(D)$ is a diagram for $L$, and since $D$ contains $S_{2n}$, it is an element of $\mathcal{D^{R}}$. See Figure \ref{tref} for an example with the trefoil.

\end{proof}

We now have all the necessary ingredients for the main theorem.

\begin{thm}\label{mainthm}

Let $L$ be a link in $S^{3}$. Then there is a spectral sequence from $\widehat{Kh}(L)$ to $\delta$-graded $\widehat{\hfk}(m(L))$, with the $\delta$-grading on $\widehat{Kh}(L)$ descending to the $\delta$-grading on $\widehat{\hfk}(L)$ up to an overall shift.

\end{thm}

\begin{proof}

By Lemma \ref{regsequence} there is a diagram $D \in \mathcal{D^{R}}$ such that $\sm(D)$ is a diagram for $m(L)$. Consider the spectral sequence on $\widehat{C}_{2}(D)$ induced by the cube filtration. By Corollary \ref{khovanovE2}, the $E_{2}$ page of this spectral sequence is $\widehat{Kh}(L)$, and by Theorem \ref{reducediso} the $E_{\infty}$ page is $\widehat{\hfk}(m(L))$. The grading $\gr_{2}$ on $\widehat{C}_{2}(D)$ corresponds to the $\delta$-gradings on both Khovanov homology and knot Floer homology, up to an overall grading shift.

\end{proof}

The differentials in this spectral sequence preserve the $\delta$-grading on Khovanov homology, so for $i \ge 1$ the differential $d_{i}$ has bi-grading $(i, 2i-2)$ with respect to $(\gr_{h}, \gr_{q})$.

One property of the complex $\widehat{C}_{2}(D)$ that we not utilized so far is a $\Z_{2}$ grading. In particular, the factor $\mathcal{L}^{+}_{D}$ can be written 
\[ \mathcal{L}_{D}^{+} = \bigotimes_{v \in v_{4}(D)}  \xymatrix{R[0]\ar@<1ex>[r]^{L(v)}&R[1]\ar@<1ex>[l]^{L^{+}(v)}} \]

\noindent
where $R[i]$ indicates a $\Z_{2}$ grading shift. The $U_{i}$ have grading $0$, and we define the remaining factor in $\widehat{C}_{2}(D)$ to lie in $\Z_{2}$ grading $0$.

Then the differential $d_{0}$ is homogeneous of degree $1$ with respect to this grading, while the edge map $d_{1}$ is homogeneous of degree $0$. The homology $H_{*}(\widehat{C}_{2}(D), d_{0})$ has a single generator at each vertex $I$ in the cube. In particular, the homology at each vertex lies in a single $\Z_{2}$ grading. But all of the edge maps are non-trivial, so they must all lie in the \emph{same} $\Z_{2}$ grading. In the spectral sequence induced by the cube filtration, the differential $d_{i}$ has $\Z_{2}$ grading $i+1$. Thus, we have shown the following:

\begin{thm}

Let $(E_{k}(D), d_{k})$ denote the spectral sequence on $\widehat{C}_{2}(D)$ induced by the cube filtration. Then the differential $d_{2i}=0$ for $i \ge 1$.

\end{thm}

\begin{ex}[The (4,5) torus knot] One of the simplest knots for which this spectral sequence is non-trivial is the (4,5) torus knot. Its homology is depicted in Figure \ref{khovt45}. The $\delta$-graded Poincar\'{e} polynomial is 
\[ 4\delta^{12}+2\delta^{10}+3\delta^{8} \]

\noindent
The $\delta$-graded knot Floer homology of the mirror of $T(4,5)$ has Poincar\'{e} polynomial
\[ 4\delta^{12}+1\delta^{10}+2\delta^{8} \]

\noindent
Thus, the non-trivial differential must map from $\delta$-grading 10 to $\delta$-grading $8$. But we know that it has odd homological grading, so the only possibility is the arrow shown in Figure \ref{nontriv}.

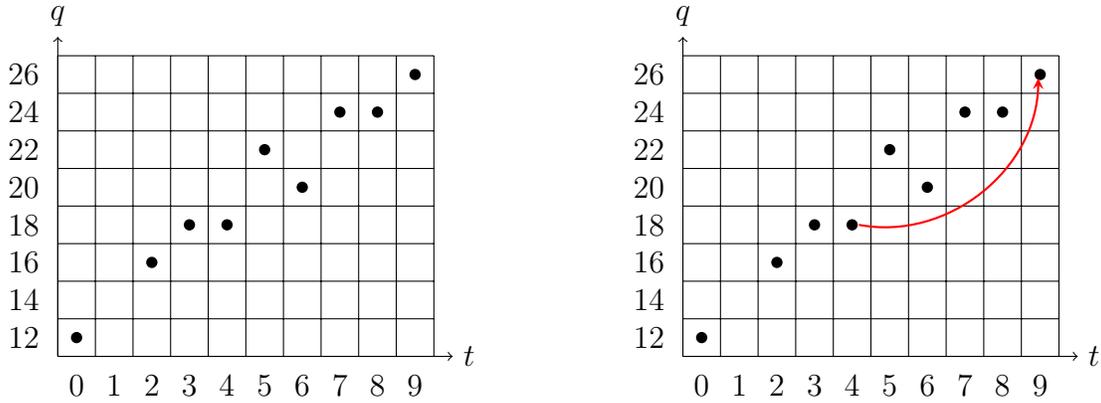
\begin{figure}[h!]
    \centering
    \begin{subfigure}{.5\textwidth}
    \centering
\begin{tikzpicture}[scale=.5]
  \draw[->] (0,0) -- (10.5,0) node[right] {$t$};
  \draw[->] (0,0) -- (0,8.5) node[above] {$q$};
  \draw[step=1] (0,0) grid (10,8);
  \draw (0.5,-.2) node[below] {$0$};
  \draw (1.5,-.2) node[below] {$1$};
  \draw (2.5,-.2) node[below] {$2$};
  \draw (3.5,-.2) node[below] {$3$};
  \draw (4.5,-.2) node[below] {$4$};
  \draw (5.5,-.2) node[below] {$5$};
  \draw (6.5,-.2) node[below] {$6$};
  \draw (7.5,-.2) node[below] {$7$};
  \draw (8.5,-.2) node[below] {$8$};
  \draw (9.5,-.2) node[below] {$9$};
  \draw (-.2,0.5) node[left] {$12$};
  \draw (-.2,1.5) node[left] {$14$};
  \draw (-.2,2.5) node[left] {$16$};
  \draw (-.2,3.5) node[left] {$18$};
  \draw (-.2,4.5) node[left] {$20$};
  \draw (-.2,5.5) node[left] {$22$};
  \draw (-.2,6.5) node[left] {$24$};
  \draw (-.2,7.5) node[left] {$26$};
  \fill (0.5, 0.5) circle (.15);
  \fill (2.5, 2.5) circle (.15);
  \fill (3.5, 3.5) circle (.15);
  \fill (4.5, 3.5) circle (.15);
  \fill (5.5, 5.5) circle (.15);
  \fill (6.5, 4.5) circle (.15);
  \fill (7.5, 6.5) circle (.15);
  \fill (8.5, 6.5) circle (.15);
  \fill (9.5, 7.5) circle (.15);
\end{tikzpicture}
\caption{Reduced Khovanov homology of $T(4,5)$.} \label{khovt45}
\end{subfigure}
    \begin{subfigure}{.49\textwidth}
    \centering
\begin{tikzpicture}[scale=.5]
  \draw[->] (0,0) -- (10.5,0) node[right] {$t$};
  \draw[->] (0,0) -- (0,8.5) node[above] {$q$};
  \draw[step=1] (0,0) grid (10,8);
  \draw (0.5,-.2) node[below] {$0$};
  \draw (1.5,-.2) node[below] {$1$};
  \draw (2.5,-.2) node[below] {$2$};
  \draw (3.5,-.2) node[below] {$3$};
  \draw (4.5,-.2) node[below] {$4$};
  \draw (5.5,-.2) node[below] {$5$};
  \draw (6.5,-.2) node[below] {$6$};
  \draw (7.5,-.2) node[below] {$7$};
  \draw (8.5,-.2) node[below] {$8$};
  \draw (9.5,-.2) node[below] {$9$};
  \draw (-.2,0.5) node[left] {$12$};
  \draw (-.2,1.5) node[left] {$14$};
  \draw (-.2,2.5) node[left] {$16$};
  \draw (-.2,3.5) node[left] {$18$};
  \draw (-.2,4.5) node[left] {$20$};
  \draw (-.2,5.5) node[left] {$22$};
  \draw (-.2,6.5) node[left] {$24$};
  \draw (-.2,7.5) node[left] {$26$};
  \fill (0.5, 0.5) circle (.15);
  \fill (2.5, 2.5) circle (.15);
  \fill (3.5, 3.5) circle (.15);
  \fill (4.5, 3.5) circle (.15);
  \fill (5.5, 5.5) circle (.15);
  \fill (6.5, 4.5) circle (.15);
  \fill (7.5, 6.5) circle (.15);
  \fill (8.5, 6.5) circle (.15);
  \fill (9.5, 7.5) circle (.15);
  \draw[-stealth, red, thick, scale=1](4.68,3.5) [bend right = 50] to (9.45,7.4);
\end{tikzpicture}
\caption{The non-trivial differential}\label{nontriv}
\end{subfigure}
\caption{The only non-zero differential in $E_{k}(T(4,5))$ has homological grading $5$ and quantum grading $8$.} \label{t452}
\end{figure}

\end{ex}

One interesting aspect of this construction is that it seems likely to extend to a minus theory. In particular, the same argument on $C^{-}_{2}(D)$ gives the following:

\begin{thm}

Let $D \in \mathcal{D^{R}}$. Then there is a spectral sequence from $Kh^{-}(\sm(D))$ to $\hfk_{2}^{-}(D)$.

\end{thm}

\noindent
In order to make this interesting, we need to show that $\hfk_{2}^{-}(D) \cong \hfk^{-}_{2}(\sm(D))$, which hopefully will be worked out in future work.

\newpage

\bibliography{TriplyGradedHomology}{}
\bibliographystyle{alpha}

\end{document}